\documentclass[12pt]{article}    % onecolumn (standardformat)

\usepackage[margin=0.75in]{geometry} %In Folder
\usepackage{amsmath,amssymb,amsthm}
\newtheorem{theorem}{Theorem}[section]

\newtheorem{lemma}[theorem]{Lemma}

\theoremstyle{remark}

\newcommand{\be}{\begin{equation}}
\newcommand{\ee}{\end{equation}}

\newcommand{\bea}{\begin{eqnarray}}
\newcommand{\eea}{\end{eqnarray}}

\numberwithin{equation}{section}
\begin{document}

\title{Linear statistics of matrix ensembles in classical background }
\author{Yang Chen and Chao Min\\
(yangbrookchen@yahoo.co.uk, chaominrunner@gmail.com)\\
Department of Mathematics, University of Macau,\\
Avenida da Universidade, Taipa, Macau, China}


\date{\today}
\maketitle
\begin{abstract}
Given a joint probability density function of $N$ real random variables, $\{x_j\}_{j=1}^{N},$
obtained from the eigenvector-eigenvalue decomposition of
 $N\times N$
random matrices, one constructs a random variable, the linear statistics,
defined by the sum of smooth functions evaluated at the eigenvalues or singular values of the random matrix, namely, $\sum_{j=1}^{N}F(x_j).$

For the jpdfs obtained from the Gaussian  and Laguerre ensembles, we compute, in this paper the moment
generating function $\mathbb{E}_{\beta}({\rm exp}(-\lambda\sum_{j}F(x_j))),$ where $\mathbb{E}_{\beta}$ denotes expectation value
over the Orthogonal ($\beta=1$) and Symplectic ($\beta=4)$ ensembles, in the
form  one plus a Schwartz function, none vanishing over $\mathbb{R}$ for the Gaussian ensembles and $\mathbb{R}^+$ for the Laguerre ensembles.

These are ultimately expressed in the form of the determinants of identity plus a scalar operator, from which
we obtained the large $N$ asymptotic of
the linear statistics from suitably scaled $F(\cdot).$
\end{abstract}

\section{Introduction}
The well-known joint probability density function for the eigenvalues $\{x_j\}_{j=1}^{N}$ of $N\times N$ Hermitian matrices from an
orthogonal ensemble ($\beta=1$), unitary ensemble ($\beta=2$) or symplectic ensemble ($\beta=4$) is given by \cite{Mehta}
$$
P^{(\beta)}(x_{1},x_{2},\ldots,x_{N})\prod_{j=1}^{N}dx_j=C_{N}^{(\beta)}\prod_{1\leq j<k\leq N}\left|x_{j}-x_{k}\right|^{\beta}
\prod_{j=1}^{N}w(x_{j})dx_j,
$$
where $w(x)$ is a weight function or a probability density supported on $[a,b]$. In this paper, we require that the moments of $w$, namely,
$$s_j:=\int_{a}^{b}x^j\;w(x)dx,$$
to exist for $j=0,1,2,\ldots.$
Here the normalization constant $C_{N}^{(\beta)}$ reads,
$$
C_{N}^{(\beta)}=\frac{1}{\int_{[a,b]^{N}}\prod_{1\leq j<k\leq N}\left|x_{j}-x_{k}\right|^{\beta}
\prod_{j=1}^{N}w(x_{j})dx_{j}}.
$$
Many years ago,
Selberg \cite{Selberg}, obtained closed form expression for $C_{N}^{(\beta)}$, where\\
$w(x)=x^{a}(1-x)^{b},\;a>-1,\;b>-1,\;x\in [0,1]$, the Jacobi weight.
The constant $C_N^{(\beta)}$
with the Gaussian weight $w(x)={\rm e}^{-x^2}$,$\;x\in \mathbb{R}$ and Laguerre weight, $w(x)=x^{\alpha}\:{\rm e}^{-x},\;\alpha>-1,\;\;x\in \mathbb{R}^+$, can be found in
\cite{Mehta}.

If we take $w(x)=\mathrm{e}^{-x^{2}},\; x\in \mathbb{R}$, these  are known as the Gaussian orthogonal ensemble (GOE),
Gaussian unitary ensemble (GUE) and Gaussian symplectic ensemble (GSE). If\\
$w(x)=x^{\alpha}\mathrm{e}^{-x},\;\;\alpha>-1,\;\;x\in \mathbb{R}^+,$ then we have analogously the LOE, LUE and LSE.

The moment generating function,
$$
G_{N}^{(\beta)}(f):=\mathbb{E}_{\beta}\left({\rm e}^{-\lambda\:\sum_{j=1}^{N}F(x_j)}\right)
:=\frac{\int_{[a,b]^{N}}\prod_{1\leq j<k\leq N}\left|x_{j}-x_{k}\right|^{\beta}\prod_{j=1}^{N}w(x_{j})\:{\rm e}^{-\lambda\:F(x_j)}dx_{j}}{\int_{[a,b]^{N}}\prod_{1\leq j<k\leq N}\left|x_{j}-x_{k}\right|^{\beta}\prod_{j=1}^{N}w(x_{j})dx_{j}},
$$
is given by
\be
G_{N}^{(\beta)}(f)=C_{N}^{(\beta)}\int_{[a,b]^{N}}\prod_{1\leq j<k\leq N}\left|x_{j}-x_{k}\right|^{\beta}
\prod_{j=1}^{N}w(x_{j})\left[1+f(x_{j})\right]dx_{j},\label{gnb}
\ee
where $\beta=1, 2, 4$. Here ${\rm e}^{-\lambda\:F(x)}:=1+f(x)$, and we do not indicate that $f(\cdot)$ also depends on $\lambda.$
Throughout this paper, we assume $f(x)$ lies in the Schwartz space, and $1+f(x)\neq 0$ over $[a,b]$.

The simplest, very well-studied, unitary case corresponds to $\beta=2$, see
\cite{Adler, Basor, Basor1993, Basor1997, Dieng, Tracy1998}, and the references therein.

We state here, for later development, facts on orthogonal polynomials.
Let\\
$\varphi_{j}(x):=P_{j}(x)\sqrt{w(x)}$, $j=0,1,2,\ldots$, where $P_{j}(x)$ are the orthonormal polynomials of
 degree $j$ with respect to the weight $w(x)$ supported on $[a,b],$
$$
\int_{a}^{b} P_j(x)P_k(x)w(x)dx=\delta_{j,k},\;\;j,k=0,1,2,\ldots.
$$
It is a well-known fact that,
$$
G_{N}^{(2)}(f)=\det\left(I+K_{N}^{(2)}f\right),
$$
where $K_{N}^{(2)}f$ is an integral operator with kernel $K_{N}^{(2)}(x,y)f(y)$.
Here
$K_{N}^{(2)}(x,y):=\sum_{j=0}^{N-1}\varphi_{j}(x)\varphi_{j}(y)$, can be evaluated via the Christoffel-Darboux formula.

With Gaussian background, where $w(x)=\mathrm{e}^{-x^{2}},\; x\in \mathbb{R}$, we have
$$P_{j}(x):=\frac{H_{j}(x)}{c_{j}},\;\;\;c_{j}=\pi^{\frac{1}{4}}2^{\frac{j}{2}}\sqrt{\Gamma(j+1)},\;\;j=0,1,2,...,$$
where $H_{j}(x)$ are the Hermite polynomials of degree $j$.

With Laguerre background, where $w(x)=x^{\alpha}\mathrm{e}^{-x},\;\alpha>-1,\;x\in \mathbb{R}^{+}$, we have
$$P_{j}(x):=\frac{L_{j}^{(\alpha)}(x)}{c_{j}^{(\alpha)}},\;\;\;c_{j}^{(\alpha)}=\sqrt{\frac{\Gamma(j+\alpha+1)}{\Gamma(j+1)}},\;\;j=0,1,2,...,$$
where $L_{j}^{(\alpha)}(x)$ are the Laguerre polynomials of degree $j$. Properties of the Hermite and Laguerre polynomials can be found in \cite{Szego}.

We shall mainly deal with the $\beta=4$ and $\beta=1$ cases, which are more complicated than $\beta=2$ situation, especially so for $\beta=1$.

We always begin with the general case, and then apply the results obtained to the two classical ensembles, i.e.,
the Gaussian ensembles and Laguerre ensembles.

Furthermore, in $\beta=1$ situation, for convenience, $N$ is taken to be even, and it is expeditious to make use of the square root of the Gaussian weight,
${\rm e}^{-x^2/2},\;x\in \mathbb{R}$  and the square root of the  Laguerre weight, $x^{\alpha/2}{\rm e}^{-x/2},\;\alpha>-2,\;x\in\mathbb{R}^+$ in later discussion.

In this paper, we shall be concerned with the large $N$ behavior of $G_{N}^{(\beta)}(f)$ for the Gaussian ensembles and Laguerre ensembles. For the Gaussian ensembles,
we replace $f(x)$ by $f\left(\sqrt{2N}x\right),$ in the orthogonal case,
and $f\left(\sqrt{4N}x\right)$, in the symplectic case.  For the Laguerre ensembles,
we replace $f(x)$ by $f\left(\sqrt{4Nx}\right)$, in the orthogonal case
and $f\left(\sqrt{8Nx}\right)$, in the symplectic case.

We will ultimately give the mean
and variance of the linear statistics, as $N\rightarrow\infty$, together with leading correction terms.

For comparison purposes, we write down results in the  GUE, where,
$w(x)=\mathrm{e}^{-x^{2}},\;x\in \mathbb{R}$.
Denote by $\mu_{N}^{(GUE)}$ and $\mathcal{V}_{N}^{(GUE)}$ the mean and variance of the linear statistics
 $\sum_{j=1}^{N}F\left(\sqrt{2N}x_{j}\right)$.
It is shown in \cite{Basor1997}, that,
\be
\mu_{N}^{(GUE)}\rightarrow\frac{1}{\pi}\int_{-\infty}^{\infty}F(x)dx,\;\;N\rightarrow\infty,\label{guem}
\ee
\be
\mathcal{V}_{N}^{(GUE)}\rightarrow\frac{1}{\pi}\int_{-\infty}^{\infty}F^{2}(x)dx
-\int_{-\infty}^{\infty}\int_{-\infty}^{\infty}\left[\frac{\sin(x-y)}{\pi(x-y)}\right]^{2}F(x)F(y)dx dy,\;\;N\rightarrow\infty.\label{guev}
\ee
For the LUE, where $w(x)=x^{\alpha}\mathrm{e}^{-x},\;\alpha>-1,\;x\in \mathbb{R}^{+}$. Denote by $\mu_{N}^{(LUE,\:\alpha)}$ and
$\mathcal{V}_{N}^{(LUE,\:\alpha)}$
the mean and variance of the linear statistics $\sum_{j=1}^{N}F\left(\sqrt{4N x_{j}}\right)$. It is shown in\cite{Basor1993}, that,
\be
\mu_{N}^{(LUE,\:\alpha)}\rightarrow\int_{0}^{\infty}B^{(\alpha)}(x,x)F(x)dx,\;\;N\rightarrow\infty,\label{luem}
\ee
\be
\mathcal{V}_{N}^{(LUE,\:\alpha)}\rightarrow\int_{0}^{\infty}B^{(\alpha)}(x,x)F^{2}(x)dx-\int_{0}^{\infty}
\int_{0}^{\infty}B^{(\alpha)}(x,y)B^{(\alpha)}(y,x)F(x)F(y)dx dy,\;\;N\rightarrow\infty,\label{luev}
\ee
where
\be
B^{(\alpha)}(x,y):=\frac{J_{\alpha}(x)y J_{\alpha}'(y)-J_{\alpha}(y)x J_{\alpha}'(x)}{x^{2}-y^{2}}x,\label{bxy}
\ee
\be
B^{(\alpha)}(x,x):=\frac{J_{\alpha}^{2}(x)-J_{\alpha-1}(x)J_{\alpha+1}(x)}{2}x.\label{bxx}
\ee
Here $J_{\alpha}(\cdot)$ is the Bessel function of order $\alpha$.

We want to point out, the motivation of this paper comes from \cite{Dieng}, which provided results both for symplectic ensembles and orthogonal ensembles, and
specialize to the Gaussian case, i.e., GSE and GOE. The \cite{Dieng} dealt with the situation where $f(.)$ is the characteristic function of an interval (or the union of disjoint intervals),
and focus on the distribution of the $m$th largest
eigenvalue in the GSE and GOE, while we are interested for "smooth" $f$, and we also consider the Laguerre case, i.e., LSE and LOE.

This paper is organized as follows. In Section 2, we recall a number of theorems, the operators $D$ and $\varepsilon,$ and
end with two Lemmas relevant
for later development. Section 3 begins with a general discussion of the Symplectic ensembles in a general setting, followed by
detailed discussions on the GSE and LSE cases
and ends with the computation of the mean and variance of linear statistics for large $N.$ Section 4 repeats the development in
Section 3 but for the Orthogonal ensembles, which is harder. We conclude in Section 5.

\section{Preliminaries}
For orientation purposes, we introduce here a number of results, which will be used throughout this paper.

\begin{theorem}
The Stirling's formula \cite{Lebedev}
\be
\Gamma(n)=\sqrt{2\pi}\:n^{n-\frac{1}{2}}\:\mathrm{e}^{-n}\left[1+O\left(n^{-1}\right)\right],\;\;n\rightarrow\infty. \label{stirling}
\ee
\end{theorem}

\begin{lemma}
$$
_{2}F_{1}\left(-N,1;\frac{3}{2};2\right)=(-1)^{N}\sqrt{\frac{\pi}{8N}}+O\left(N^{-1}\right),\;\;N\rightarrow\infty.\label{2f1}
$$
\end{lemma}

\begin{proof}
From the integral representation of the hypergeometric function,
\be
_{2}F_{1}(\alpha,\beta;\gamma;z)=\frac{\Gamma(\gamma)}{\Gamma(\beta)\Gamma(\gamma-\beta)}
\int_{0}^{1}t^{\beta-1}(1-t)^{\gamma-\beta-1}(1-tz)^{-\alpha}dt,\;\;
\mathrm{Re}\gamma>\mathrm{Re}\beta>0, \label{ir}
\ee
we obtain,
$$
_{2}F_{1}\left(-N,1;\frac{3}{2};2\right)=\frac{1}{2}\int_{0}^{1}\frac{(1-2t)^{N}}{\sqrt{1-t}}dt.
$$
Let
$$
x=\sqrt{1-t},
$$
then
\bea
_{2}F_{1}\left(-N,1;\frac{3}{2};2\right)
&=&\int_{0}^{1}(2x^{2}-1)^{N}dx\nonumber\\
&=&\int_{0}^{\frac{\sqrt{2}}{2}}(2x^{2}-1)^{N}dx+\int_{\frac{\sqrt{2}}{2}}^{1}(2x^{2}-1)^{N}dx\nonumber\\
&=&(-1)^{N}\int_{0}^{\frac{\sqrt{2}}{2}}(1-2x^{2})^{N}dx+\int_{\frac{\sqrt{2}}{2}}^{1}(2x^{2}-1)^{N}dx.\label{two}
\eea
Consider the first integral in (\ref{two}). Let
$$
x=\frac{\sqrt{2}}{2}\cos\theta, \;\;\theta\in\left[0,\frac{\pi}{2}\right],
$$
then
\bea
\int_{0}^{\frac{\sqrt{2}}{2}}(1-2x^{2})^{N}dx
&=&\frac{\sqrt{2}}{2}\int_{0}^{\frac{\pi}{2}}(\sin\theta)^{2N+1}d\theta\nonumber\\
&=&\frac{\sqrt{2}}{2}\frac{\left[2^{N}\Gamma(N+1)\right]^{2}}{\Gamma(2N+2)}\nonumber\\
&=&\sqrt{\frac{\pi}{8N}}\left[1+O\left(N^{-1}\right)\right],\;\;N\rightarrow\infty,\nonumber
\eea
where use has been made of the Stirling's formula (\ref{stirling}), in the last equality.

Now consider the second integral in (\ref{two}),
\bea
\int_{\frac{\sqrt{2}}{2}}^{1}(2x^{2}-1)^{N}dx
&\leq&\int_{\frac{\sqrt{2}}{2}}^{1}(2x-1)^{N}dx\nonumber\\
&=&\frac{1-(\sqrt{2}-1)^{N+1}}{2(N+1)}.\nonumber
\eea
Hence
$$
_{2}F_{1}\left(-N,1;\frac{3}{2};2\right)=(-1)^{N}\sqrt{\frac{\pi}{8N}}+O\left(N^{-1}\right),\;\;N\rightarrow\infty.
$$
\end{proof}

\begin{lemma}
If $\alpha>0$, then
$$
\sum_{m=0}^{2n}\frac{(-1)^{m}}{m!\prod_{l=m}^{2n}(\alpha+2l)}{2n+\alpha \choose 2n-m+1}=\frac{1}{(2n+1)!}, \;\;n=0,1,2,\ldots,\label{bi}
$$
where ${j \choose k}:=\frac{\Gamma(j+1)}{\Gamma(k+1)\Gamma(j-k+1)}$.
\end{lemma}

\begin{proof}
Firstly,
\bea
&&\sum_{m=0}^{2n}\frac{(-1)^{m}}{m!\prod_{l=m}^{2n}(\alpha+2l)}{2n+\alpha \choose 2n-m+1}\nonumber\\
&=&\frac{1}{(2n+1)!}+\frac{\Gamma(2n+\alpha+1)\Gamma\big(\frac{\alpha}{2}\big)\ _{2}F_{1}\big(-2n-1,\frac{\alpha}{2};\alpha;2\big)}{2^{2n+1}\Gamma(2n+2)
\Gamma\big(2n+\frac{\alpha}{2}+1\big)\Gamma(\alpha)}.\label{hy}
\eea
From (\ref{ir}), we see that,
$$
_{2}F_{1}\Big(-2n-1,\frac{\alpha}{2};\alpha;2\Big)
=\frac{\Gamma(\alpha)}{\Gamma^{2}\big(\frac{\alpha}{2}\big)}\int_{0}^{1}t^{\frac{\alpha}{2}-1}(1-t)^{\frac{\alpha}{2}-1}(1-2t)^{2n+1}dt.
$$
Let
$$
t=\cos^{2}\bigg(\frac{\theta}{2}\bigg),\;\; \theta\in[0,\pi],
$$
then
$$
\int_{0}^{1}t^{\frac{\alpha}{2}-1}(1-t)^{\frac{\alpha}{2}-1}(1-2t)^{2n+1}dt
=-\left(\frac{1}{2}\right)^{\alpha-1}\int_{0}^{\pi}(\sin\theta)^{\alpha-1}(\cos\theta)^{2n+1} d\theta.
$$
We carry out mathematical induction in $n$ to prove that for $\alpha>0,\;$
$\int_{0}^{\pi}(\sin\theta)^{\alpha-1}(\cos\theta)^{2n+1} d\theta=0.$
\\
Setting if $n=0$, in the integral above, we have,
\bea
&&\int_{0}^{\pi}(\sin\theta)^{\alpha-1}\cos\theta d\theta\nonumber\\
&=&\int_{0}^{\pi}(\sin\theta)^{\alpha-1} d\sin\theta\nonumber\\
&=&\frac{(\sin\theta)^{\alpha}}{\alpha}\bigg|_{0}^{\pi}\nonumber\\
&=&0.\nonumber
\eea
Suppose $\int_{0}^{\pi}(\sin\theta)^{\alpha-1}(\cos\theta)^{2k+1} d\theta=0$, for $\alpha>0$, then
\bea
&&\int_{0}^{\pi}(\sin\theta)^{\alpha-1}(\cos\theta)^{2k+3} d\theta\nonumber\\
&=&\int_{0}^{\pi}(\sin\theta)^{\alpha-1}(\cos\theta)^{2k+1}\cos^{2}\theta d\theta\nonumber\\
&=&\int_{0}^{\pi}(\sin\theta)^{\alpha-1}(\cos\theta)^{2k+1}\left(1-\sin^{2}\theta\right) d\theta\nonumber\\
&=&\int_{0}^{\pi}(\sin\theta)^{\alpha-1}(\cos\theta)^{2k+1}d\theta-\int_{0}^{\pi}(\sin\theta)^{\alpha+1}(\cos\theta)^{2k+1} d\theta\nonumber\\
&=&0.\nonumber
\eea
Hence,
$$
\int_{0}^{\pi}(\sin\theta)^{\alpha-1}(\cos\theta)^{2n+1} d\theta=0,\;\;\alpha>0.
$$
It follows that, for $\alpha>0,$
$$
_{2}F_{1}\Big(-2n-1,\frac{\alpha}{2};\alpha;2\Big)=0,
$$
and from (\ref{hy}), the lemma follows.
\end{proof}

\begin{theorem}
$$
\prod_{1\leq j<k\leq N}\left(x_{j}-x_{k}\right)^{4}=\det\left(x_{k}^{j},jx_{k}^{j-1}\right)_{{j=0,\ldots,2N-1}\atop{k=1,\ldots,N}},
\label{det}
$$
where the determinant on the right is a $2N\times2N$ determinant with alternating columns\\
 $\{x_{1}^{j}\},\; \{jx_{1}^{j-1}\},\; \{x_{2}^{j}\},\; \{jx_{2}^{j-1}\}, \ldots$ \cite{Kuramoto}.
\end{theorem}
The next theorem, due to de Bruijn \cite{de Bruijn}, is instrumental for the finite $N$ computations in Sections 3 and 4.

\begin{theorem}
For any integrable functions $p_{j}(x)$ and $q_{j}(x), j=1,2,\ldots$, we have
$$
\left(\int_{[a,b]^{N}}\det\left(p_{j}(x_{k}),q_{j}(x_{k})\right)_{{j=1,\ldots,2N}\atop{k=1,\ldots,N}}dx_{1}\cdots dx_{N}\right)^{2}
=(N!)^{2}\det\left(\int_{a}^{b}\left(p_{j}(x)q_{k}(x)-p_{k}(x)q_{j}(x)\right)dx\right)_{j,k=1}^{2N},
$$

$$
\left(\int_{a\leq x_{1}\leq\cdots\leq x_{N}\leq b}\det\left(p_{j}(x_{k})\right)_{j,k=1}^{N}dx_{1}\cdots dx_{N}\right)^{2}
=\det\left(\int_{a}^{b}\int_{a}^{b}\mathrm{sgn}(y-x)p_{j}(x)p_{k}(y)dxdy\right)_{j,k=1}^{N},\label{pf}
$$
where the determinant on the left of the first equality is a $2N\times2N$ determinant with alternating columns
 $\{p_{j}(x_{1})\},\; \{q_{j}(x_{1})\},\; \{p_{j}(x_{2})\},\; \{q_{j}(x_{2})\}, \ldots$. In addition, $N$ is even in the second equality.
\end{theorem}

\begin{theorem}
If $A, B$ are Hilbert-Schmidt operators on a Hilbert space $\mathcal{H}$, then \cite{Gohberg}
$$
\det(I+AB)=\det(I+BA).\label{hs}
$$
\end{theorem}

Following \cite{Tracy1996, Tracy1998, Widom}, we introduce here operators $\varepsilon$ and $D$ which will be crucial for later development.
Let $\varepsilon$ be the integral operator with kernel
$$
\varepsilon(x,y):=\frac{1}{2}\mathrm{sgn}(x-y),
$$
then for any integrable function $g$ defined on $[a,b],$
$$
\varepsilon g(x)=\int_{a}^{b}\varepsilon(x,t)g(t)dt=\frac{1}{2}\left(\int_{a}^{x}g(t)dt-\int_{x}^{b}g(t)dt\right),\;\;x\in[a,b].
$$
It is clear that $\varepsilon(y,x)=-\varepsilon(x,y)$, i.e., $\varepsilon^{t}=-\varepsilon$, where $^{t}$ denotes transpose.

Let $D$ be the operator that acts by differentiation, thus for any differentiable function $g$ defined on $[a,b],$
$$
Dg(x)=\frac{dg(x)}{dx}=g'(x).
$$
With further conditions on $g(x)$, we prove an easy lemma on the commutator $[D,\varepsilon].$
\begin{lemma}
For any function $g\in C^{1}[a,b]$ and $g(a)=g(b)=0$, $D\varepsilon g(x)=\varepsilon D g(x)=g(x)$, i.e.,
$D\varepsilon=\varepsilon D=I$.\label{de}
\end{lemma}

\begin{proof}
For any function $g\in C^{1}[a,b]$ and $g(a)=g(b)=0$, we have
\bea
(D\varepsilon)g(x)
&=&\frac{d}{dx}\int_{a}^{b}\varepsilon(x,t)g(t)dt\nonumber\\
&=&\frac{1}{2}\frac{d}{dx}\left[\int_{a}^{x}g(t)dt-\int_{x}^{b}g(t)dt\right]\nonumber\\
&=&\frac{1}{2}g(x)+\frac{1}{2}g(x)=g(x),\nonumber
\eea
and
\bea
(\varepsilon D)g(x)
&=&\int_{a}^{b}\varepsilon(x,t)g'(t)dt\nonumber\\
&=&\frac{1}{2}\int_{a}^{x}g'(t)dt-\frac{1}{2}\int_{x}^{b}g'(t)dt\nonumber\\
&=&\frac{1}{2}[g(x)-g(a)]-\frac{1}{2}[g(b)-g(x)]\nonumber\\
&=&g(x)-\frac{1}{2}[g(a)+g(b)]\nonumber\\
&=&g(x).\nonumber
\eea
The proof is complete.
\end{proof}
%{\bf Remark} If we replace $g(x)$ by $\chi_{[a,b]}(x)f(x),$ and  $f$ a Schwarz function supported on $\mathbb{R}$ or
% $\mathbb{R}^+$ with $f(0)=0$,
%we see the $\varepsilon\:D\neq I.$
Denote by $u\otimes v$ the integral operator with kernel $(u\otimes v)(x,y):=u(x)v(y)$. We have the following lemma.
\begin{lemma}
If $A$ is an integral operator with kernel $A(x,y)$, then
$$
A(u\otimes v)=(Au)\otimes v,\;\; (u\otimes v)A=u\otimes(A^{t}v).\label{ab}
$$
\end{lemma}

\section{The symplectic ensembles}
\subsection{General case}
Taking $\beta=4$ in (\ref{gnb}) gives the symplectic ensembles and generating function becomes,
$$
G_{N}^{(4)}(f)=C_{N}^{(4)}\int_{[a,b]^{N}}\prod_{1\leq j<k\leq N}\left(x_{j}-x_{k}\right)^{4}\prod_{j=1}^{N}
w(x_{j})\left[1+f(x_{j})\right]dx_{j},
$$
where
$$
C_{N}^{(4)}=\frac{1}{\int_{[a,b]^N}\prod_{1\leq j<k\leq N}\left(x_{j}-x_{k}\right)^{4}\prod_{j=1}^{N}w(x_{j})dx_{j}}
$$
is a constant depending on $N$.

We follow closely the computations of Dieng \cite{Dieng}, and Tracy and Widom \cite{Tracy1998}.

From Theorem \ref{det} and Theorem \ref{pf} and some linear algebra, we get
$$
\left[G_{N}^{(4)}(f)\right]^{2}=\widehat{C_{N}^{(4)}}\det\left(\int_{a}^{b}\left[\pi_{j}(x)\pi_{k}'(x)-\pi_{j}'(x)\pi_{k}(x)\right]
w(x)(1+f(x))dx\right)_{j,k=0}^{2N-1},
$$
where $\widehat{C_{N}^{(4)}}$ is a $N$ dependent constant, and $\pi_j(x)$ is any polynomial of degree $j$. Let
\be
\psi_{j}(x)=\pi_{j}(x)\sqrt{w(x)},\label{psij4}
\ee
and following \cite{Dieng, Tracy1998}, we see that,
$$
\left[G_{N}^{(4)}(f)\right]^{2}=\det\left(I+\left(M^{(4)}\right)^{-1}L^{(4)}\right),
$$
where $M^{(4)}, L^{(4)}$ are matrices given by
$$
M^{(4)}=\left(\int_{a}^{b}\left(\psi_{j}(x)\psi_{k}'(x)-\psi_{j}'(x)\psi_{k}(x)\right)dx\right)_{j,k=0}^{2N-1},
$$
$$
L^{(4)}=\left(\int_{a}^{b}\left(\psi_{j}(x)\psi_{k}'(x)-\psi_{j}'(x)\psi_{k}(x)\right)f(x)dx\right)_{j,k=0}^{2N-1}.
$$
With the notation,
$$
\left(M^{(4)}\right)^{-1}=:(\mu_{jk})_{j,k=0}^{2N-1},
$$
we obtain, finally,
\be
\left[G_{N}^{(4)}(f)\right]^{2}=\det\left(I+K_{N}^{(4)}f\right),\label{gn42}
\ee
where $K_{N}^{(4)}$ is the integral operator
$$
K_{N}^{(4)}=\begin{pmatrix}
\sum_{j,k=0}^{2N-1}\mu_{jk}\psi_{j}'\otimes\psi_{k}&-\sum_{j,k=0}^{2N-1}\mu_{jk}\psi_{j}'\otimes\psi_{k}'\\
\sum_{j,k=0}^{2N-1}\mu_{jk}\psi_{j}\otimes\psi_{k}&-\sum_{j,k=0}^{2N-1}\mu_{jk}\psi_{j}\otimes\psi_{k}'
\end{pmatrix}
=:\begin{pmatrix}
K_{N,4}^{(1,1)}&K_{N,4}^{(1,2)}\\
K_{N,4}^{(2,1)}&K_{N,4}^{(2,2)}
\end{pmatrix},
$$
In the next theorem, we obtain further relations on $K_{N,4}^{(i,j)}, i,j=1,2$, which ultimately expresses $K_N^{(4)}$ in terms of
$K_{N,4}^{(2,2)}.$

\begin{theorem}
$$
K_{N,4}^{(2,1)}=K_{N,4}^{(2,2)}\varepsilon,\;\;K_{N,4}^{(1,1)}=DK_{N,4}^{(2,2)}\varepsilon,\;\;K_{N,4}^{(1,2)}=DK_{N,4}^{(2,2)},
\;\;K_{N,4}^{(2,2)}\varepsilon D=K_{N,4}^{(2,2)}. \label{re}
$$
\end{theorem}

\begin{proof}
First of all,
\bea
K_{N,4}^{(2,2)}\varepsilon
&=&-\sum_{j,k=0}^{2N-1}(\mu_{jk}\psi_{j}\otimes\psi_{k}')\varepsilon\nonumber\\
&=&\sum_{j,k=0}^{2N-1}\mu_{jk}\psi_{j}\otimes(\varepsilon\psi_{k}')\nonumber\\
&=&\sum_{j,k=0}^{2N-1}\mu_{jk}\psi_{j}\otimes\psi_{k}\nonumber\\
&=&K_{N,4}^{(2,1)}.\nonumber
\eea
Secondly, for any integrable function $g(x)$ defined on $[a,b]$, we have
\bea
D K_{N,4}^{(2,1)}g(x)
&=&\frac{d}{dx}\int_{a}^{b}K_{N,4}^{(2,1)}(x,y)g(y)dy\nonumber\\
&=&\int_{a}^{b}\frac{\partial}{\partial x}K_{N,4}^{(2,1)}(x,y)g(y)dy\nonumber\\
&=&\int_{a}^{b}K_{N,4}^{(1,1)}(x,y)g(y)dy\nonumber\\
&=&K_{N,4}^{(1,1)}g(x),\nonumber
\eea
i.e.,
$$
D K_{N,4}^{(2,1)}=K_{N,4}^{(1,1)}.
$$
It follows that
$$
K_{N,4}^{(1,1)}=DK_{N,4}^{(2,2)}\varepsilon.
$$
Similarly,
\bea
D K_{N,4}^{(2,2)}g(x)
&=&\frac{d}{dx}\int_{a}^{b}K_{N,4}^{(2,2)}(x,y)g(y)dy=\int_{a}^{b}\frac{\partial}{\partial x}K_{N,4}^{(2,2)}(x,y)g(y)dy\nonumber\\
&=&\int_{a}^{b}K_{N,4}^{(1,2)}(x,y)g(y)dy\nonumber\\
&=&K_{N,4}^{(1,2)}g(x),\nonumber
\eea
i.e.,
$$
K_{N,4}^{(1,2)}=D K_{N,4}^{(2,2)}.
$$
Finally, we find
\bea
K_{N,4}^{(2,2)}\varepsilon D g(x)
&=&K_{N,4}^{(2,1)} D g(x)=\int_{a}^{b}K_{N,4}^{(2,1)}(x,y)g'(y)dy\nonumber\\
&=&K_{N,4}^{(2,1)}(x,y)g(y)|_{a}^{b}-\int_{a}^{b}g(y)\frac{\partial}{\partial y}K_{N,4}^{(2,1)}(x,y)dy.\nonumber
\eea
Note that
\bea
K_{N,4}^{(2,1)}(x,y):
&=&\sum_{j,k=0}^{2N-1}\psi_{j}(x)\mu_{jk}\psi_{k}(y)\nonumber\\
&=&\sum_{j,k=0}^{2N-1}\psi_{j}(x)\mu_{jk}\pi_{k}(y)\sqrt{w(y)}\nonumber
\eea
and
$$
w(a)=w(b)=0,
$$
hence
$$
K_{N,4}^{(2,1)}(x,a)=K_{N,4}^{(2,1)}(x,b)=0.
$$
Continuing,
\bea
K_{N,4}^{(2,2)}\varepsilon D g(x)
&=&-\int_{a}^{b}g(y)\frac{\partial}{\partial y}K_{N,4}^{(2,1)}(x,y)dy=\int_{a}^{b}K_{N,4}^{(2,2)}(x,y)g(y)dy\nonumber\\
&=&K_{N,4}^{(2,2)}g(x),\nonumber
\eea
i.e.,
$$
K_{N,4}^{(2,2)}\varepsilon D=K_{N,4}^{(2,2)}.
$$
The proof is complete.
\end{proof}
According to Theorem \ref{re}, $K_{N}^{(4)}$ can be written as
\bea
K_{N}^{(4)}
&=&\begin{pmatrix}
D K_{N,4}^{(2,2)}\varepsilon&D K_{N,4}^{(2,2)}\\
K_{N,4}^{(2,2)}\varepsilon&K_{N,4}^{(2,2)}
\end{pmatrix}=
\begin{pmatrix}
D&0\\
0&I
\end{pmatrix}
\begin{pmatrix}
K_{N,4}^{(2,2)}\varepsilon&K_{N,4}^{(2,2)}\\
K_{N,4}^{(2,2)}\varepsilon&K_{N,4}^{(2,2)}
\end{pmatrix}\nonumber\\
&=&:\widetilde{A}\widetilde{B},\nonumber
\eea
where
$$
\widetilde{A}=\begin{pmatrix}
D&0\\
0&I
\end{pmatrix},\;\;\;\;
\widetilde{B}=\begin{pmatrix}
K_{N,4}^{(2,2)}\varepsilon&K_{N,4}^{(2,2)}\\
K_{N,4}^{(2,2)}\varepsilon&K_{N,4}^{(2,2)}
\end{pmatrix}.
$$
From (\ref{gn42}) and using Theorem \ref{hs}, we have
$$
\left[G_{N}^{(4)}(f)\right]^{2}=\det\bigg(I+\Big(\widetilde{A}\widetilde{B}\Big)f\bigg)
=\det\bigg(I+\widetilde{A}\Big(\widetilde{B}f\Big)\bigg)=\det\bigg(I+\Big(\widetilde{B}f\Big)\widetilde{A}\bigg)=
\det\bigg(I+\widetilde{B}\Big(f\widetilde{A}\Big)\bigg).
$$
Since
\bea
\widetilde{B}\left(f\widetilde{A}\right)
&=&
\begin{pmatrix}
K_{N,4}^{(2,1)}&K_{N,4}^{(2,2)}\\
K_{N,4}^{(2,1)}&K_{N,4}^{(2,2)}
\end{pmatrix}
\left(f
\begin{pmatrix}
D&0\\
0&I
\end{pmatrix}\right)
\nonumber\\
&=&
\begin{pmatrix}
K_{N,4}^{(2,2)}\varepsilon&K_{N,4}^{(2,2)}\\
K_{N,4}^{(2,2)}\varepsilon&K_{N,4}^{(2,2)}
\end{pmatrix}
\begin{pmatrix}
f D&0\\
0&f
\end{pmatrix}
\nonumber\\
&=&\begin{pmatrix}
K_{N,4}^{(2,2)}\varepsilon f D&K_{N,4}^{(2,2)}f\\
K_{N,4}^{(2,2)}\varepsilon f D&K_{N,4}^{(2,2)}f
\end{pmatrix},\nonumber
\eea
then
$$
\left[G_{N}^{(4)}(f)\right]^{2}=\det
\begin{pmatrix}
I+K_{N,4}^{(2,2)}\varepsilon f D&K_{N,4}^{(2,2)}f\\
K_{N,4}^{(2,2)}\varepsilon f D&I+K_{N,4}^{(2,2)}f
\end{pmatrix}.
$$

The computation below reduces the above into a determinant of scalar operators. We subtract row 1 from row 2,
$$
\left[G_{N}^{(4)}(f)\right]^{2}=\det
\begin{pmatrix}
I+K_{N,4}^{(2,2)}\varepsilon f D&K_{N,4}^{(2,2)}f\\
-I&I
\end{pmatrix}.
$$
Next, add column 2 to column 1,
\bea
\left[G_{N}^{(4)}(f)\right]^{2}
&=&\det
\begin{pmatrix}
I+K_{N,4}^{(2,2)}\varepsilon f D+K_{N,4}^{(2,2)}f&K_{N,4}^{(2,2)}f\\
0&I
\end{pmatrix}\nonumber\\
&=&\det\left(I+K_{N,4}^{(2,2)}\varepsilon f D+K_{N,4}^{(2,2)}f\right).\nonumber
\eea
$\mathbf{Remark}.$ The above result agrees with \cite{Dieng} for the GUE case if we take $f=-\mu\:\chi_{J}$, where $\chi_{J}$ is the characteristic function of the interval $J$.

Now we use the commutator $[D,f]:=Df-fD$ to obtain a better suited result for our purpose. For a given function {\it smooth} $g(x)$, we have
\bea
[D,f]g(x)
&=&D f g(x)-f D g(x)\nonumber\\
&=&(f(x)g(x))'-f(x)g'(x)\nonumber\\
&=&f'(x)g(x),\nonumber
\eea
this is,
$$
[D,f]=f'.
$$
It follows that
\be
fD=Df-[D,f]=Df-f'.\label{fd}
\ee
Taking this into account, we have the following theorem.
\begin{theorem}
$$
\left[G_{N}^{(4)}(f)\right]^{2}
=\det\left(I+2K_{N,4}^{(2,2)}f-K_{N,4}^{(2,2)}\varepsilon f'\right),\label{gn4s}
$$
where the kernel of $K_{N,4}^{(2,2)}$ reads,
$$
K_{N,4}^{(2,2)}(x,y)=-\sum_{j,k=0}^{2N-1}\psi_{j}(x)\mu_{jk}\psi_{k}'(y).
$$
\end{theorem}

\subsection{GSE}
In the case of the Gaussian weight $w(x)=\mathrm{e}^{-x^{2}},\;\;x\in \mathbb{R}$,
we again follow the discussions \cite{Dieng, Tracy1998},
and choose a special $\psi_{j}$ to simplify  $M^{(4)}$ as much as possible. To proceed, let
%$$
%\int_{\mathbb{R}}\varphi_{j}(x)\varphi_{k}(x)dx=\delta_{jk}.
%$$
\be
\psi_{2j+1}(x):=\frac{1}{\sqrt{2}}\varphi_{2j+1}(x),\;\; \psi_{2j}(x):=-\frac{1}{\sqrt{2}}\varepsilon\varphi_{2j+1}(x),\;\;j=0,1,2,\ldots,
\label{psi2}
\ee
where $\varphi_{j}(x)$ is given by
\be
\varphi_{j}(x)=\frac{H_{j}(x)}{c_{j}}\mathrm{e}^{-\frac{x^{2}}{2}},\;\;c_{j}=\pi^{\frac{1}{4}}2^{\frac{j}{2}}\sqrt{\Gamma(j+1)},\label{gue}
\ee
and $H_{j}(x), j=0,1,\ldots$ are the usual Hermite polynomials with the orthogonality condition
$$
\int_{-\infty}^{\infty}H_{m}(x)H_{n}(x)\mathrm{e}^{-x^{2}}dx=c_{n}^{2}\:\delta_{mn}.
$$
We show in the next lemma that, this definition satisfies (\ref{psij4}), i.e.,
$\psi_{j}(x)=\pi_{j}(x)\mathrm{e}^{-\frac{x^{2}}{2}}, j=0,1,2,\ldots$, where
$\pi_j(x)$ is a polynomial of degree $j.$
\begin{lemma}
$\psi_{j}(x)\mathrm{e}^{\frac{x^{2}}{2}},\:j=0,1,2,\ldots$ is a polynomial of degree $j$.
\end{lemma}

\begin{proof}
If the index  is odd, then it is clear that $\psi_{2j+1}(x)\mathrm{e}^{\frac{x^{2}}{2}}$ is a polynomial of degree $2j+1$.
For even index,
\bea
\psi_{2j}(x)\mathrm{e}^{\frac{x^{2}}{2}}
&=&-\frac{1}{\sqrt{2}}\left(\varepsilon\varphi_{2j+1}(x)\right)\mathrm{e}^{\frac{x^{2}}{2}}\nonumber\\
&=&-\frac{1}{\sqrt{2}}\mathrm{e}^{\frac{x^{2}}{2}}\varepsilon\left(\frac{1}{c_{2j+1}}H_{2j+1}(x)\mathrm{e}^{-\frac{x^{2}}{2}}\right)\nonumber\\
&=&-\frac{1}{2\sqrt{2}c_{2j+1}}\mathrm{e}^{\frac{x^{2}}{2}}\left[\int_{-\infty}^{x}H_{2j+1}(y)\mathrm{e}^{-\frac{y^{2}}{2}}dy
-\int_{x}^{\infty}H_{2j+1}(y)\mathrm{e}^{-\frac{y^{2}}{2}}dy\right],\;\;j=0,1,2,\cdots.\nonumber
\eea
Since $H_{2j+1}(y)\mathrm{e}^{-\frac{y^{2}}{2}}, j=0,1,\ldots$ is an odd function, we have
$$
0=\int_{-\infty}^{\infty}H_{2j+1}(y)\mathrm{e}^{-\frac{y^{2}}{2}}dy=
\int_{-\infty}^{x}H_{2j+1}(y)\mathrm{e}^{-\frac{y^{2}}{2}}dy+\int_{x}^{\infty}H_{2j+1}(y)\mathrm{e}^{-\frac{y^{2}}{2}}dy.
$$
Hence
$$
\int_{x}^{\infty}H_{2j+1}(y)\mathrm{e}^{-\frac{y^{2}}{2}}dy=-\int_{-\infty}^{x}H_{2j+1}(y)\mathrm{e}^{-\frac{y^{2}}{2}}dy,
$$
and we find,
$$
\psi_{2j}(x)\mathrm{e}^{\frac{x^{2}}{2}}=-\frac{1}{\sqrt{2}c_{2j+1}}\mathrm{e}^{\frac{x^{2}}{2}}\int_{-\infty}^{x}H_{2j+1}(y)
\mathrm{e}^{-\frac{y^{2}}{2}}dy.
$$
From mathematical induction, it follows that
\bea
\int_{-\infty}^{x}y^{k}\mathrm{e}^{-\frac{y^{2}}{2}}dy
&=&-\mathrm{e}^{-\frac{x^{2}}{2}}\left[x^{k-1}+(k-1)x^{k-3}+(k-1)(k-3)x^{k-5}\right.\nonumber\\
&+&\left.\cdots+(k-1)(k-3)\cdots2\right],\;\;k=1,3,5,\cdots.\nonumber
\eea
Since $H_{2j+1}(y)$ is a linear combination of $y, y^{3},\ldots, y^{2j+1}$,
 we see that $\int_{-\infty}^{x}H_{2j+1}(y)\mathrm{e}^{-\frac{y^{2}}{2}}dy$ is equal to $\mathrm{e}^{-\frac{x^{2}}{2}}$ multiplying
  a polynomial of degree $2j$. It follows that $\psi_{2j}(x)\mathrm{e}^{\frac{x^{2}}{2}}$ is a polynomial of degree $2j$. The proof is complete.
\end{proof}

Using (\ref{psi2}) to compute
$$
M^{(4)}
:=\left(\int_{-\infty}^{\infty}\left(\psi_{j}(x)\psi_{k}'(x)-\psi_{j}'(x)\psi_{k}(x)\right)dx\right)_{j,k=0}^{2N-1},
$$
we obtain the following lemma.
\begin{lemma}$
\mathrm{\mathbf{(Dieng, Tracy-Widom)}}$
$$
M^{(4)}=\begin{pmatrix}
0&1&0&0&\cdots&0&0\\
-1&0&0&0&\cdots&0&0\\
0&0&0&1&\cdots&0&0\\
0&0&-1&0&\cdots&0&0\\
\vdots&\vdots&\vdots&\vdots&&\vdots&\vdots\\
0&0&0&0&\cdots&0&1\\
0&0&0&0&\cdots&-1&0
\end{pmatrix}_{2N\times 2N}.
$$
\end{lemma}

It is clear that $\left(M^{(4)}\right)^{-1}=-M^{(4)}$, so $\mu_{2j,2j+1}=-1, \mu_{2j+1,2j}=1$, and $\mu_{jk}=0$ for other cases. Hence
\bea
K_{N,4}^{(2,2)}(x,y)
&=&-\sum_{j,k=0}^{2N-1}\psi_{j}(x)\mu_{jk}\psi_{k}'(y)\nonumber\\
&=&\sum_{j=0}^{N-1}\psi_{2j}(x)\psi_{2j+1}'(y)-\sum_{j=0}^{N-1}\psi_{2j+1}(x)\psi_{2j}'(y)\nonumber\\
&=&\frac{1}{2}\left[\sum_{j=0}^{N-1}\varphi_{2j+1}(x)\varphi_{2j+1}(y)-\sum_{j=0}^{N-1}\varepsilon\varphi_{2j+1}(x)\varphi_{2j+1}'(y)\right].
\label{sn}
\eea
%where we just used the definition (\ref{psi2}) in the last step.
Recall that the Hermite polynomials $H_{j}$ satisfy the differentiation formulas \cite{Lebedev}
\be
H_{j}'(x)=2xH_{j}(x)-H_{j+1}(x),\;\;j=0,1,2,\ldots, \label{rh}
\ee
\be
H_{j}'(x)=2jH_{j-1}(x),\;\;j=0,1,2,\ldots. \label{lh}
\ee
Using the fact that $H_{j}(x)=c_{j}\:\varphi_{j}(x)\:{\rm e}^{\frac{x^{2}}{2}}$, (\ref{rh}) becomes
\be
\varphi_{j}'(x)=x\varphi_{j}(x)-\sqrt{2(j+1)}\varphi_{j+1}(x),\;\;j=0,1,2,\ldots, \label{rphi}
\ee
and similarly, (\ref{lh}) becomes,
\be
\varphi_{j}'(x)=-x\varphi_{j}(x)+\sqrt{2j}\varphi_{j-1}(x),\;\;j=0,1,2,\ldots. \label{lphi}
\ee
Combining (\ref{rphi}) and (\ref{lphi}), to eliminate $\varphi_j(x),$ we obtain
\be
\varphi_{j}'(x)=\sqrt{\frac{j}{2}}\varphi_{j-1}(x)-\sqrt{\frac{j}{2}+\frac{1}{2}}\varphi_{j+1}(x),\;\;j=0,1,2,\ldots. \label{phid}
\ee
Using (\ref{phid}) to replace $\varphi_{2j+1}'(y)$, we find,
\bea
&&\sum_{j=0}^{N-1}\varepsilon\varphi_{2j+1}(x)\varphi_{2j+1}'(y)\nonumber\\
&=&\sum_{j=0}^{N-1}\varepsilon\varphi_{2j+1}(x)\left[\sqrt{j+\frac{1}{2}}\varphi_{2j}(y)-\sqrt{j+1}\varphi_{2j+2}(y)\right]\nonumber\\
&=&\sum_{j=0}^{N-1}\sqrt{j+\frac{1}{2}}\varepsilon\varphi_{2j+1}(x)\varphi_{2j}(y)
-\sum_{j=0}^{N-1}\sqrt{j+1}\varepsilon\varphi_{2j+1}(x)\varphi_{2j+2}(y)\nonumber\\
&=&\sum_{j=0}^{N}\sqrt{j+\frac{1}{2}}\varepsilon\varphi_{2j+1}(x)\varphi_{2j}(y)
-\sum_{j=0}^{N}\sqrt{j}\varepsilon\varphi_{2j-1}(x)\varphi_{2j}(y)-\sqrt{N+\frac{1}{2}}\varepsilon\varphi_{2N+1}(x)\varphi_{2N}(y)\nonumber\\
&=&-\sum_{j=0}^{N}\left[\sqrt{j}\varepsilon\varphi_{2j-1}(x)-\sqrt{j+\frac{1}{2}}\varepsilon\varphi_{2j+1}(x)\right]\varphi_{2j}(y)
-\sqrt{N+\frac{1}{2}}\varepsilon\varphi_{2N+1}(x)\varphi_{2N}(y).\label{phi2j}
\eea
To proceed further, using Lemma \ref{de},
%note that $\varepsilon\:D=I$ acting on functions in the $\mathcal{S}(\mathbb{R})$,
together with (\ref{phid}) we find
\bea
\varphi_{2j}(x)
&=&\varepsilon D\varphi_{2j}(x)=\varepsilon\:\varphi_{2j}'(x)\nonumber\\
%&=&\varepsilon\left[\sqrt{j}\varphi_{2j-1}(y)-\sqrt{j+\frac{1}{2}}\varphi_{2j+1}(y)\right]\nonumber\\
&=&\sqrt{j}\;\varepsilon\varphi_{2j-1}(x)-\sqrt{j+\frac{1}{2}}\;\varepsilon\varphi_{2j+1}(x).\label{phi2jy}
\eea
Substituting (\ref{phi2jy}) into (\ref{phi2j}), it follows that
$$
\sum_{j=0}^{N-1}\varepsilon\varphi_{2j+1}(x)\varphi_{2j+1}'(y)=-\sum_{j=0}^{N}\varphi_{2j}(x)\varphi_{2j}(y)
-\sqrt{N+\frac{1}{2}}\varepsilon\varphi_{2N+1}(x)\varphi_{2N}(y).
$$
Hence, (\ref{sn}) becomes,
\bea
K_{N,4}^{(2,2)}(x,y)
&=&\frac{1}{2}\left[\sum_{j=0}^{N-1}\varphi_{2j+1}(x)\varphi_{2j+1}(y)+\sum_{j=0}^{N}\varphi_{2j}(x)\varphi_{2j}(y)
+\sqrt{N+\frac{1}{2}}\varepsilon\varphi_{2N+1}(x)\varphi_{2N}(y)\right]\nonumber\\
&=&\frac{1}{2}\left[\sum_{j=0}^{2N}\varphi_{j}(x)\varphi_{j}(y)+\sqrt{N+\frac{1}{2}}\varepsilon\varphi_{2N+1}(x)\varphi_{2N}(y)\right]\nonumber\\
&=&\frac{1}{2}\left[S_N(x,y)+\sqrt{N+\frac{1}{2}}\varepsilon\varphi_{2N+1}(x)\varphi_{2N}(y)\right],\nonumber
\eea
where
$$
S_{N}(x,y):=\sum_{j=0}^{2N}\varphi_{j}(x)\varphi_{j}(y)=\sqrt{N+\frac{1}{2}}\:\frac{\varphi_{2N+1}(x)\varphi_{2N}(y)-\varphi_{2N+1}(y)\varphi_{2N}(x)}{x-y},
$$
and here the last equality comes from the Christoffel-Darboux formula.

By Theorem \ref{gn4s}, we have the following theorem.
\begin{theorem}
$$
\left[G_{N}^{(4)}(f)\right]^{2}
=\det\left(I+S_{N}f-\frac{1}{2}S_{N}\varepsilon f'+\sqrt{N+\frac{1}{2}}\left(\varepsilon\varphi_{2N+1}\otimes\varphi_{2N}f\right)
+\frac{1}{2}\sqrt{N+\frac{1}{2}}\left(\varepsilon\varphi_{2N+1}\otimes\varepsilon\varphi_{2N}\right) f'\right).
$$
\end{theorem}

\subsection{Large $N$ behavior of the GSE moment generating function}
To proceed with the large $N$ investigation, write,
$\left[G_{N}^{(4)}(f)\right]^{2}$,
as
$$
\left[G_{N}^{(4)}(f)\right]^{2}=:\det(I+T),
$$
where
\be
T:=S_{N}f-\frac{1}{2}S_{N}\varepsilon f'+\sqrt{N+\frac{1}{2}}\left(\varepsilon\varphi_{2N+1}\otimes\varphi_{2N}f\right)
+\frac{1}{2}\sqrt{N+\frac{1}{2}}\left(\varepsilon\varphi_{2N+1}\otimes\varepsilon\varphi_{2N}\right) f'.\label{ct}
\ee
We find,
$$
\log\det(I+T)=\mathrm{Tr}\log(I+T)=\mathrm{Tr}\:T-\frac{1}{2}\mathrm{Tr}\:T^{2}+\frac{1}{3}\mathrm{Tr}\:T^{3}-\cdots.
$$
This is obtained by the trace-log expansion of $\log\det(I+\nu\:T)$, where $0<\nu<1,$ follow by a continuation to $\nu=1$. We assume that similar continuation also holds in other cases.

\begin{theorem}
$$
\lim_{N\rightarrow\infty}\frac{1}{\sqrt{4N}}S_{N}\left(\frac{x}{\sqrt{4N}},\frac{y}{\sqrt{4N}}\right)=\frac{\sin(x-y)}{\pi(x-y)},
$$

$$
\lim_{N\rightarrow\infty}\frac{1}{\sqrt{4N}}S_{N}\left(\frac{x}{\sqrt{4N}},\frac{x}{\sqrt{4N}}\right)=\frac{1}{\pi}.\label{sn4xx}
$$
\end{theorem}

The next theorem characterizes the large $N$ asymptotic of various ``scaled" quantities.
\begin{theorem}
$$
\varphi_{2N}\left(\frac{x}{\sqrt{4N}}\right)=(-1)^{N}\pi^{-\frac{1}{2}}N^{-\frac{1}{4}}\cos x+O\left(N^{-\frac{3}{4}}\right),\;\;N\rightarrow\infty,
$$

$$
\varepsilon\varphi_{2N}\left(\frac{x}{\sqrt{4N}}\right)=2^{-1}(-1)^{N}\pi^{-\frac{1}{2}}N^{-\frac{3}{4}}\sin x+O\left(N^{-\frac{5}{4}}\right),\;\;N\rightarrow\infty,
$$

$$
\varepsilon\varphi_{2N+1}\left(\frac{x}{\sqrt{4N}}\right)=-2^{-\frac{1}{2}}N^{-\frac{1}{4}}+O\left(N^{-\frac{3}{4}}\right),\;\;N\rightarrow\infty.\label{gse}
$$
\end{theorem}

\begin{proof}
From the asymptotic expansion of Hermite polynomials \cite{Szego},
\be
H_{n}(x)\:\mathrm{e}^{-\frac{x^{2}}{2}}=\lambda_{n}\left[\cos\left(\sqrt{2n+1}x-\frac{n\pi}{2}\right)+O\left(n^{-\frac{1}{2}}\right)\right],\;\;n\rightarrow\infty,
\label{asy}
\ee
where
$$
\lambda_{2n}=\frac{\Gamma(2n+1)}{\Gamma\left(n+1\right)}\qquad \mathrm{and} \qquad \lambda_{2n+1}=\frac{\Gamma(2n+3)}{\Gamma(n+2)}
\left(4n+3\right)^{-\frac{1}{2}}.
$$

We have
\bea
\varphi_{2N}\left(\frac{x}{\sqrt{4N}}\right)
&=&\frac{1}{\pi^{\frac{1}{4}}2^{N}\sqrt{\Gamma(2N+1)}}H_{2N}\left(\frac{x}{\sqrt{4N}}\right)\mathrm{e}^{-\frac{x^{2}}{8N}}\nonumber\\
&=&\frac{1}{\pi^{\frac{1}{4}}2^{N}\sqrt{\Gamma(2N+1)}}\cdot\frac{\Gamma(2N+1)}{\Gamma(N+1)}\left[\cos\left(\sqrt{\frac{4N+1}{4N}}x-N\pi\right)+O\left(N^{-\frac{1}{2}}\right)\right]
\nonumber\\
&=&\frac{1}{\pi^{\frac{1}{4}}2^{N}}\cdot\frac{\sqrt{\Gamma(2N+1)}}{\Gamma(N+1)}\left[\cos(x-N\pi)+O\left(N^{-\frac{1}{2}}\right)\right]
\nonumber\\
&=&(-1)^{N}\pi^{-\frac{1}{2}}N^{-\frac{1}{4}}\cos x+O\left(N^{-\frac{3}{4}}\right),\;\;N\rightarrow\infty,\nonumber
\eea
where we have used the Stirling's formula (\ref{stirling}).

A straightforward computation gives,
\bea
\varepsilon\varphi_{2N}(x)
&=&\frac{1}{2}\left(\int_{-\infty}^{x}\varphi_{2N}(y)dy-\int_{x}^{\infty}\varphi_{2N}(y)dy\right)\nonumber\\
&=&\frac{1}{2}\left(\int_{-\infty}^{0}\varphi_{2N}(y)dy+\int_{0}^{x}\varphi_{2N}(y)dy-\int_{x}^{0}\varphi_{2N}(y)dy-\int_{0}^{\infty}\varphi_{2N}(y)dy\right)\nonumber\\
&=&\frac{1}{2}\left(2\int_{0}^{x}\varphi_{2N}(y)dy+\int_{-\infty}^{0}\varphi_{2N}(y)dy-\int_{0}^{\infty}\varphi_{2N}(y)dy\right)\nonumber\\
&=&\frac{1}{\pi^{\frac{1}{4}}2^{N+1}\sqrt{\Gamma(2N+1)}}\left(2\int_{0}^{x}H_{2N}(y)\mathrm{e}^{-\frac{y^{2}}{2}}dy+\int_{-\infty}^{0}H_{2N}(y)\mathrm{e}^{-\frac{y^{2}}{2}}dy
-\int_{0}^{\infty}H_{2N}(y)\mathrm{e}^{-\frac{y^{2}}{2}}dy\right)\nonumber\\
&=&\frac{1}{\pi^{\frac{1}{4}}2^{N}\sqrt{\Gamma(2N+1)}}\int_{0}^{x}H_{2N}(y)\mathrm{e}^{-\frac{y^{2}}{2}}dy.\nonumber
\eea
So we find, for large $N$,
\bea
\varepsilon\varphi_{2N}\left(\frac{x}{\sqrt{4N}}\right)
&=&\frac{1}{\pi^{\frac{1}{4}}2^{N}\sqrt{\Gamma(2N+1)}}\int_{0}^{\frac{x}{\sqrt{4N}}}H_{2N}(y)\mathrm{e}^{-\frac{y^{2}}{2}}dy\nonumber\\
&=&\frac{1}{\pi^{\frac{1}{4}}2^{N}\sqrt{\Gamma(2N+1)}}\cdot\frac{1}{\sqrt{4N}}\int_{0}^{x}H_{2N}\left(\frac{y}{\sqrt{4N}}\right)\mathrm{e}^{-\frac{y^{2}}{8N}}dy\nonumber\\
&=&\frac{1}{\pi^{\frac{1}{4}}2^{N}\sqrt{\Gamma(2N+1)}}\cdot\frac{1}{\sqrt{4N}}\cdot\frac{\Gamma(2N+1)}{\Gamma(N+1)}
\int_{0}^{x}\left[(-1)^{N}\cos y+O\left(N^{-\frac{1}{2}}\right)\right]dy\nonumber\\
&=&\frac{\sqrt{\Gamma(2N+1)}}{\pi^{\frac{1}{4}}2^{N}\sqrt{4N}\:\Gamma(N+1)}\left[(-1)^{N}\sin x+O\left(N^{-\frac{1}{2}}\right)\right]\nonumber\\
&=&2^{-1}(-1)^{N}\pi^{-\frac{1}{2}}N^{-\frac{3}{4}}\sin x+O\left(N^{-\frac{5}{4}}\right),\;\;N\rightarrow\infty,\nonumber
\eea
where we have used the Stirling's formula (\ref{stirling}) in the last step.

Finally, a straightforward computation gives,
\bea
\varepsilon\varphi_{2N+1}(x)
&=&\frac{1}{2}\left[\int_{-\infty}^{x}\varphi_{2N+1}(y)dy-\int_{x}^{\infty}\varphi_{2N+1}(y)dy\right]\nonumber\\
&=&\frac{1}{\pi^{\frac{1}{4}}2^{N+\frac{3}{2}}\sqrt{\Gamma(2N+2)}}\left[\int_{-\infty}^{x}H_{2N+1}(y)\mathrm{e}^{-\frac{y^{2}}{2}}dy
-\int_{x}^{\infty}H_{2N+1}(y)\mathrm{e}^{-\frac{y^{2}}{2}}dy\right]\nonumber\\
&=&\frac{1}{\pi^{\frac{1}{4}}2^{N+\frac{3}{2}}\sqrt{\Gamma(2N+2)}}\left[\int_{-\infty}^{0}H_{2N+1}(y)\mathrm{e}^{-\frac{y^{2}}{2}}dy
+\int_{0}^{x}H_{2N+1}(y)\mathrm{e}^{-\frac{y^{2}}{2}}dy\right.\nonumber\\
&-&\left.\int_{x}^{0}H_{2N+1}(y)\mathrm{e}^{-\frac{y^{2}}{2}}dy-\int_{0}^{\infty}H_{2N+1}(y)\mathrm{e}^{-\frac{y^{2}}{2}}dy\right]\nonumber\\
&=&\frac{1}{\pi^{\frac{1}{4}}2^{N+\frac{3}{2}}\sqrt{\Gamma(2N+2)}}\left[2\int_{0}^{x}H_{2N+1}(y)\mathrm{e}^{-\frac{y^{2}}{2}}dy
-2\int_{0}^{\infty}H_{2N+1}(y)\mathrm{e}^{-\frac{y^{2}}{2}}dy\right]\nonumber\\
&=&\frac{1}{\pi^{\frac{1}{4}}2^{N+\frac{1}{2}}\sqrt{\Gamma(2N+2)}}\left[\int_{0}^{x}H_{2N+1}(y)\mathrm{e}^{-\frac{y^{2}}{2}}dy
-\int_{0}^{\infty}H_{2N+1}(y)\mathrm{e}^{-\frac{y^{2}}{2}}dy\right],\nonumber
\eea
continuing, we see that,
\bea
\varepsilon\varphi_{2N+1}\left(\frac{x}{\sqrt{4N}}\right)
&=&\frac{1}{\pi^{\frac{1}{4}}2^{N+\frac{1}{2}}\sqrt{\Gamma(2N+2)}}\int_{0}^{\frac{x}{\sqrt{4N}}}H_{2N+1}(y)\mathrm{e}^{-\frac{y^{2}}{2}}dy\nonumber\\
&-&\frac{1}{\pi^{\frac{1}{4}}2^{N+\frac{1}{2}}\sqrt{\Gamma(2N+2)}}\int_{0}^{\infty}H_{2N+1}(y)\mathrm{e}^{-\frac{y^{2}}{2}}dy\nonumber\\
&=&\frac{1}{\pi^{\frac{1}{4}}2^{N+\frac{1}{2}}\sqrt{\Gamma(2N+2)}}\cdot\frac{1}{\sqrt{4N}}\int_{0}^{x}H_{2N+1}\left(\frac{y}{\sqrt{4N}}\right)\mathrm{e}^{-\frac{y^{2}}{8N}}dy
\nonumber\\
&-&\frac{1}{\pi^{\frac{1}{4}}2^{N+\frac{1}{2}}\sqrt{\Gamma(2N+2)}}\int_{0}^{\infty}H_{2N+1}(y)\mathrm{e}^{-\frac{y^{2}}{2}}dy.\nonumber
\eea
By (\ref{asy}), we have
$$
H_{2N+1}\left(\frac{y}{\sqrt{4N}}\right)\mathrm{e}^{-\frac{y^{2}}{8N}}=\frac{\Gamma(2N+3)}{\Gamma(N+2)}(4N+3)^{-\frac{1}{2}}
\left[(-1)^{N}\sin y+O\left(N^{-\frac{1}{2}}\right)\right],\;\;N\rightarrow\infty.
$$
It follows that
$$
\int_{0}^{x}H_{2N+1}\left(\frac{y}{\sqrt{4N}}\right)\mathrm{e}^{-\frac{y^{2}}{8N}}dy=\frac{\Gamma(2N+3)}{\Gamma(N+2)}(4N+3)^{-\frac{1}{2}}
\left[(-1)^{N}(1-\cos x)+O\left(N^{-\frac{1}{2}}\right)\right],\;\;N\rightarrow\infty.
$$
On the other hand, from \cite{Gradshteyn},
$$
\int_{0}^{\infty}H_{2N+1}(y)\mathrm{e}^{-\frac{y^{2}}{2}}dy=\frac{2\:(-1)^{N}\Gamma(2N+2)}{\Gamma(N+1)}
\:_{2}F_{1}\left(-N,1;\frac{3}{2};2\right).
$$
It follows from Lemma \ref{2f1}, that,
$$
\int_{0}^{\infty}H_{2N+1}(y)\mathrm{e}^{-\frac{y^{2}}{2}}dy=\frac{2\:(-1)^{N}\Gamma(2N+2)}{\Gamma(N+1)}\left[(-1)^{N}\sqrt{\frac{\pi}{8N}}
+O\left(N^{-1}\right)\right],\;\;N\rightarrow\infty.
$$
By the Stirling's formula (\ref{stirling}), we obtain
\bea
\varepsilon\varphi_{2N+1}\left(\frac{x}{\sqrt{4N}}\right)
&=&\left[2^{-1}(-1)^{N}\pi^{-\frac{1}{2}}N^{-\frac{3}{4}}(1-\cos x)+O\left(N^{-\frac{5}{4}}\right)\right]
-\left[2^{-\frac{1}{2}}N^{-\frac{1}{4}}+O\left(N^{-\frac{3}{4}}\right)\right]\nonumber\\
&=&-2^{-\frac{1}{2}}N^{-\frac{1}{4}}+O\left(N^{-\frac{3}{4}}\right),\;\;N\rightarrow\infty.\nonumber
\eea
\end{proof}

We are now in a position to compute $\mathrm{Tr}\:T$ and $\mathrm{Tr}\:T^{2}$ as $N\rightarrow\infty$, using Theorem \ref{sn4xx} and
Theorem \ref{gse}. The estimates provided by Theorem \ref{gse} are instrumental in the large $N$ computations that follows.

In what follows, we replace $f(x)$ by $f\left(\sqrt{4N}x\right)$, and note that
$$
f'(\sqrt{4N}x)=\frac{1}{\sqrt{4N}}\:\frac{d}{dx}f(\sqrt{4N}x).
$$
The $f'$ that appears in the trace will be accordingly interpreted.

We first consider $\mathrm{Tr}\:T$, which reads,
$$
\mathrm{Tr}\:T=\mathrm{Tr}\:S_{N}f-\mathrm{Tr}\: \frac{1}{2}S_{N}\varepsilon f'+\mathrm{Tr}\:\sqrt{N+\frac{1}{2}}
\left(\varepsilon\varphi_{2N+1}\otimes\varphi_{2N}f\right)
+\mathrm{Tr}\:\frac{1}{2}\sqrt{N+\frac{1}{2}}(\varepsilon\varphi_{2N+1}\otimes\varepsilon\varphi_{2N})f'.
$$
 $\mathrm{Tr}\:T$ has four parts.
 \\
First of all, we find
\bea
\mathrm{Tr}\:S_{N}f
&=&\int_{-\infty}^{\infty}S_{N}(x,x)f\left(\sqrt{4N}x\right)dx\nonumber\\
&=&\int_{-\infty}^{\infty}\frac{1}{\sqrt{4N}}S_{N}\left(\frac{x}{\sqrt{4N}},\frac{x}{\sqrt{4N}}\right)f(x)dx\nonumber\\
&\rightarrow&\frac{1}{\pi}\int_{-\infty}^{\infty}f(x)dx,\;\;N\to\infty.\nonumber
\eea
The second term reads,
\bea
\mathrm{Tr}\:\frac{1}{2}S_{N}\varepsilon f'
&=&\frac{1}{2}\int_{-\infty}^{\infty}\int_{-\infty}^{\infty}S_{N}(x,y)\varepsilon(y,x)f'\left(\sqrt{4N}x\right)dx dy\nonumber\\
&=&\frac{1}{4}\int_{-\infty}^{\infty}\left[\int_{x}^{\infty}S_{N}(x,y)dy-\int_{-\infty}^{x}S_{N}(x,y)dy\right]f'\left(\sqrt{4N}x\right)dx.
\nonumber
\eea
To proceed further, let
$$
u=\sqrt{4N}x,\;\;v=\sqrt{4N}y,
$$
it follows that,
\bea
\mathrm{Tr}\:\frac{1}{2}S_{N}\varepsilon f'
&=&\frac{1}{4\sqrt{4N}}\int_{-\infty}^{\infty}\left[\int_{u}^{\infty}\frac{1}{\sqrt{4N}}S_{N}\left(\frac{u}{\sqrt{4N}},\frac{v}{\sqrt{4N}}\right)dv
-\int_{-\infty}^{u}\frac{1}{\sqrt{4N}}S_{N}\left(\frac{u}{\sqrt{4N}},\frac{v}{\sqrt{4N}}\right)dv\right]f'(u)du\nonumber\\
&\rightarrow&\frac{1}{4\sqrt{4N}\pi}\int_{-\infty}^{\infty}\left[\int_{u}^{\infty}\frac{\sin(u-v)}{u-v}dv-\int_{-\infty}^{u}\frac{\sin(u-v)}{u-v}dv\right]f'(u)du,
\;\;N\to\infty.\nonumber
\eea
Now let $t=u-v$, it follows that,
\bea
\mathrm{Tr}\:\frac{1}{2}S_{N}\varepsilon f'
&\rightarrow&\frac{1}{4\sqrt{4N}\pi}\int_{-\infty}^{\infty}\left[\int_{-\infty}^{0}\frac{\sin t}{t}dt
-\int_{0}^{\infty}\frac{\sin t}{t}dt\right]f'(u)du\nonumber\\
&=&\frac{1}{4\sqrt{4N}\pi}\int_{-\infty}^{\infty}\left(\frac{\pi}{2}-\frac{\pi}{2}\right)f'(u)du\nonumber\\
&=&0,\;\;N\rightarrow\infty.\nonumber
\eea
For the third term, we have,
\bea
\mathrm{Tr}\:\sqrt{N+\frac{1}{2}}\left(\varepsilon\varphi_{2N+1}\otimes\varphi_{2N}f\right)
&=&\sqrt{N+\frac{1}{2}}\int_{-\infty}^{\infty}\varepsilon\varphi_{2N+1}(x)\varphi_{2N}(x)f\left(\sqrt{4N}x\right)dx\nonumber\\
&=&\sqrt{N+\frac{1}{2}}\cdot\frac{1}{\sqrt{4N}}\int_{-\infty}^{\infty}\varepsilon\varphi_{2N+1}\left(\frac{x}{\sqrt{4N}}\right)
\varphi_{2N}\left(\frac{x}{\sqrt{4N}}\right)f(x)dx
\nonumber\\
&=&-\frac{(-1)^{N}}{2\sqrt{2\pi N}}\int_{-\infty}^{\infty}\cos x\: f(x)dx+O\left(N^{-1}\right),\;\;N\to\infty,\nonumber
\eea
where use has been made of Theorem \ref{gse}.

Finally to the fourth term, and take note of Theorem \ref{gse},
\bea
\mathrm{Tr}\:\frac{1}{2}\sqrt{N+\frac{1}{2}}(\varepsilon\varphi_{2N+1}\otimes\varepsilon\varphi_{2N})f'
&=&\frac{1}{2}\sqrt{N+\frac{1}{2}}\int_{-\infty}^{\infty}\varepsilon\varphi_{2N+1}(x)\varepsilon\varphi_{2N}(x)
f'\left(\sqrt{4N}x\right)dx\nonumber\\
&=&\frac{1}{2}\sqrt{N+\frac{1}{2}}\cdot\frac{1}{\sqrt{4N}}\int_{-\infty}^{\infty}\varepsilon\varphi_{2N+1}\left(\frac{x}{\sqrt{4N}}\right)
\varepsilon\varphi_{2N}\left(\frac{x}{\sqrt{4N}}\right)f'(x)dx\nonumber\\
&=&O\left(N^{-1}\right),\;\;N\to\infty.\nonumber
\eea

Therefore, the large $N$ expansion of ${\mathrm{Tr}}T$, reads,
$$
\mathrm{Tr}\:T=\frac{1}{\pi}\int_{-\infty}^{\infty}f(x)dx
-\frac{(-1)^{N}}{2\sqrt{2\pi N}}\int_{-\infty}^{\infty}\cos x\: f(x)dx+O\left(N^{-1}\right),\;\;N\rightarrow\infty.
$$

Working out $\mathrm{Tr}\:T^{2}$, with $T$ given by, (\ref{ct}), there are 10 traces:
\bea
\mathrm{Tr}\:T^{2}
&=&\mathrm{Tr}\:S_{N}f S_{N}f-\mathrm{Tr}\:S_{N}f S_{N}\varepsilon f'+\mathrm{Tr}\:\sqrt{4N+2}S_{N}f(\varepsilon\varphi_{2N+1}
\otimes\varphi_{2N}f)\nonumber\\
&+&\mathrm{Tr}\:\sqrt{N+\frac{1}{2}}S_{N}f(\varepsilon\varphi_{2N+1}\otimes\varepsilon\varphi_{2N})f'
+\mathrm{Tr}\:\frac{1}{4}S_{N}\varepsilon f'S_{N}\varepsilon f'\nonumber\\
&-&\mathrm{Tr}\:\sqrt{N+\frac{1}{2}}S_{N}\varepsilon f'(\varepsilon\varphi_{2N+1}\otimes\varphi_{2N}f)
-\mathrm{Tr}\:\frac{1}{2}\sqrt{N+\frac{1}{2}}S_{N}\varepsilon f'(\varepsilon\varphi_{2N+1}\otimes\varepsilon\varphi_{2N})f'\nonumber\\
&+&\mathrm{Tr}\:\left(N+\frac{1}{2}\right)(\varepsilon\varphi_{2N+1}\otimes\varphi_{2N}f)(\varepsilon\varphi_{2N+1}\otimes\varphi_{2N}f)
\nonumber\\
&+&\mathrm{Tr}\:\left(N+\frac{1}{2}\right)(\varepsilon\varphi_{2N+1}\otimes\varphi_{2N}f)(\varepsilon\varphi_{2N+1}\otimes\varepsilon
\varphi_{2N})f'\nonumber\\
&+&\mathrm{Tr}\:\frac{1}{4}\left(N+\frac{1}{2}\right)(\varepsilon\varphi_{2N+1}\otimes\varepsilon\varphi_{2N})f'(\varepsilon\varphi_{2N+1}
\otimes\varepsilon\varphi_{2N})f'.\label{t2}
\eea
In the following, we calculate the trace on the right side of (\ref{t2}), term by term.
\\
The first term:
\bea
\mathrm{Tr}\:S_{N}f S_{N}f
&=&\int_{-\infty}^{\infty}\int_{-\infty}^{\infty}S_{N}(x,y)f\left(\sqrt{4N}y\right)S_{N}(y,x)f\left(\sqrt{4N}x\right)dy dx\nonumber\\
&=&\int_{-\infty}^{\infty}\int_{-\infty}^{\infty}\left[\frac{1}{\sqrt{4N}}S_{N}\left(\frac{x}{\sqrt{4N}},\frac{y}{\sqrt{4N}}\right)\right]^{2}
f(x)f(y)dx dy\nonumber\\
&\rightarrow&\frac{1}{\pi^{2}}\int_{-\infty}^{\infty}\int_{-\infty}^{\infty}\left[\frac{\sin(x-y)}{x-y}\right]^{2}f(x)f(y)dx dy,
\;\;N\to\infty.\nonumber
\eea
The second term reads,
\bea
\mathrm{Tr}\:S_{N}f S_{N}\varepsilon f'
&=&\int_{-\infty}^{\infty}\int_{-\infty}^{\infty}\int_{-\infty}^{\infty}S_{N}(x,y)f\left(\sqrt{4N}y\right)S_{N}(y,z)\varepsilon(z,x)
f'\left(\sqrt{4N}x\right)dx dy dz\nonumber\\
&=&\frac{1}{2}\int_{-\infty}^{\infty}\int_{-\infty}^{\infty}S_{N}(x,y)\left[\int_{x}^{\infty}S_{N}(y,z)dz-\int_{-\infty}^{x}S_{N}(y,z)dz\right]
f'\left(\sqrt{4N}x\right)f\left(\sqrt{4N}y\right)dx dy.\nonumber
\eea
A change of variables,
$$
u=\sqrt{4N}x,\;\;v=\sqrt{4N}y,\;\;w=\sqrt{4N}z,
$$
give
\bea
\mathrm{Tr}\:S_{N}f S_{N}\varepsilon f'
&=&\frac{1}{2\sqrt{4N}}\int_{-\infty}^{\infty}\int_{-\infty}^{\infty}\frac{1}{\sqrt{4N}}S_{N}
\left(\frac{u}{\sqrt{4N}},\frac{v}{\sqrt{4N}}
\right)\nonumber\\
&&\left[\int_{u}^{\infty}\frac{1}{\sqrt{4N}}S_{N}\left(\frac{v}{\sqrt{4N}},\frac{w}{\sqrt{4N}}\right)dw
-\int_{-\infty}^{u}\frac{1}{\sqrt{4N}}S_{N}\left(\frac{v}{\sqrt{4N}},\frac{w}{\sqrt{4N}}\right)dw\right]f'(u)f(v)du dv\nonumber\\
&\rightarrow&\frac{1}{2\sqrt{4N}\pi^{2}}\int_{-\infty}^{\infty}\int_{-\infty}^{\infty}\frac{\sin(u-v)}{u-v}
\left[\int_{u}^{\infty}\frac{\sin(v-w)}{v-w}dw
-\int_{-\infty}^{u}\frac{\sin(v-w)}{v-w}dw\right]\nonumber\\
&&f'(u)f(v)du dv,\;\;N\to\infty.\nonumber
\eea
To proceed further, let $t=v-w$, we see that,
\bea
&&\int_{u}^{\infty}\frac{\sin(v-w)}{v-w}dw-\int_{-\infty}^{u}\frac{\sin(v-w)}{v-w}dw
=\int_{-\infty}^{v-u}\frac{\sin t}{t}dt-\int_{v-u}^{\infty}\frac{\sin t}{t}dt
=2\int_{0}^{v-u}\frac{\sin t}{t}dt\nonumber\\
&=&2\:\mathrm{Si}(v-u),\nonumber
\eea
where $\mathrm{Si}(x)$ is the sine integral
$$
\mathrm{Si}(x):=\int_{0}^{x}\frac{\sin t}{t}dt.
$$
Since $\mathrm{Si}(-x)=-\mathrm{Si}(x),$
it follows that
\bea
\mathrm{Tr}\:S_{N}f S_{N}\varepsilon f'
&\rightarrow&\frac{1}{\sqrt{4N}\pi^{2}}\int_{-\infty}^{\infty}\int_{-\infty}^{\infty}\frac{\sin(u-v)}{u-v}\:\mathrm{Si}(v-u)f'(u)f(v)dudv\nonumber\\
&=&-\frac{1}{2\pi^{2}\sqrt{N}}\int_{-\infty}^{\infty}\int_{-\infty}^{\infty}\frac{\sin(u-v)}{u-v}\:\mathrm{Si}(u-v)f'(u)f(v)du dv,\;N\to\infty.\nonumber
\eea
The third term:
\bea
&&\mathrm{Tr}\:\sqrt{4N+2}S_{N}f(\varepsilon\varphi_{2N+1}\otimes\varphi_{2N}f)\nonumber\\
&=&\sqrt{4N+2}\int_{-\infty}^{\infty}\int_{-\infty}^{\infty}S_{N}(x,y)f\left(\sqrt{4N}y\right)\varepsilon\varphi_{2N+1}(y)
\varphi_{2N}(x)f\left(\sqrt{4N}x\right)dxdy\nonumber\\
&=&\sqrt{4N+2}\cdot\frac{1}{\sqrt{4N}}\int_{-\infty}^{\infty}\int_{-\infty}^{\infty}\frac{1}{\sqrt{4N}}S_{N}
\left(\frac{x}{\sqrt{4N}},\frac{y}{\sqrt{4N}}\right)
f(y)\varepsilon\varphi_{2N+1}\left(\frac{y}{\sqrt{4N}}\right)\varphi_{2N}\left(\frac{x}{\sqrt{4N}}\right)f(x)dxdy\nonumber\\
&=&-\frac{(-1)^{N}}{\pi\sqrt{2\pi N}}\int_{-\infty}^{\infty}\int_{-\infty}^{\infty}\frac{\sin(x-y)}{x-y}\cos x\:f(x)f(y)dx dy
+O\left(N^{-1}\right),\;N\to\infty.\nonumber
\eea
The fourth term:
\bea
&&\mathrm{Tr}\:\sqrt{N+\frac{1}{2}}S_{N}f(\varepsilon\varphi_{2N+1}\otimes\varepsilon\varphi_{2N})f'\nonumber\\
&=&\sqrt{N+\frac{1}{2}}\int_{-\infty}^{\infty}\int_{-\infty}^{\infty}S_{N}(x,y)f\left(\sqrt{4N}y\right)\varepsilon\varphi_{2N+1}(y)
\varepsilon\varphi_{2N}(x)f'\left(\sqrt{4N}x\right)dx dy\nonumber\\
&=&\sqrt{N+\frac{1}{2}}\cdot\frac{1}{\sqrt{4N}}\int_{-\infty}^{\infty}\int_{-\infty}^{\infty}\frac{1}{\sqrt{4N}}S_{N}
\left(\frac{x}{\sqrt{4N}},\frac{y}{\sqrt{4N}}\right)f(y)
\varepsilon\varphi_{2N+1}\left(\frac{y}{\sqrt{4N}}\right)\varepsilon\varphi_{2N}\left(\frac{x}{\sqrt{4N}}\right)f'(x)dx dy\nonumber\\
&=&O\left(N^{-1}\right),\;N\to\infty.\nonumber
\eea
The fifth term, with the change of variables,
$$
u=\sqrt{4N}x,\;\;v=\sqrt{4N}y,\;\;w=\sqrt{4N}z,\;\;\tau=\sqrt{4N}t,
$$
we see that,
\bea
&&\mathrm{Tr}\:\frac{1}{4}S_{N}\varepsilon f'S_{N}\varepsilon f'\nonumber\\
&=&\frac{1}{64N}\int_{-\infty}^{\infty}\int_{-\infty}^{\infty}\left[\int_{w}^{\infty}\frac{1}{\sqrt{4N}}S_{N}\left(\frac{u}{\sqrt{4N}},\frac{v}{\sqrt{4N}}\right)dv
-\int_{-\infty}^{w}\frac{1}{\sqrt{4N}}S_{N}\left(\frac{u}{\sqrt{4N}},\frac{v}{\sqrt{4N}}\right)dv\right]\nonumber\\
&&\left[\int_{u}^{\infty}\frac{1}{\sqrt{4N}}S_{N}\left(\frac{w}{\sqrt{4N}},\frac{\tau}{\sqrt{4N}}\right)d\tau
-\int_{-\infty}^{u}\frac{1}{\sqrt{4N}}S_{N}\left(\frac{w}{\sqrt{4N}},\frac{\tau}{\sqrt{4N}}\right)d\tau\right]f'(u)f'(w)dudw\nonumber\\
&=&O\left(N^{-1}\right),\;N\to\infty.\nonumber
\eea
The sixth term,
\bea
&&\mathrm{Tr}\:\sqrt{N+\frac{1}{2}}S_{N}\varepsilon f'(\varepsilon\varphi_{2N+1}\otimes\varphi_{2N}f)\nonumber\\
&=&\frac{1}{8N}\sqrt{N+\frac{1}{2}}\int_{-\infty}^{\infty}\int_{-\infty}^{\infty}\left[\int_{w}^{\infty}\frac{1}{\sqrt{4N}}S_{N}\left(\frac{u}{\sqrt{4N}},\frac{v}{\sqrt{4N}}\right)dv
-\int_{-\infty}^{w}\frac{1}{\sqrt{4N}}S_{N}\left(\frac{u}{\sqrt{4N}},\frac{v}{\sqrt{4N}}\right)dv\right]\nonumber\\
&&\varphi_{2N}\left(\frac{u}{\sqrt{4N}}\right)\varepsilon\varphi_{2N+1}\left(\frac{w}{\sqrt{4N}}\right)f(u)f'(w)du dw\nonumber\\
&=&O\left(N^{-1}\right),\;N\to\infty.\nonumber
\eea
The seventh term, becomes,
\bea
&&\mathrm{Tr}\:\frac{1}{2}\sqrt{N+\frac{1}{2}}S_{N}\varepsilon f'(\varepsilon\varphi_{2N+1}\otimes\varepsilon\varphi_{2N})f'\nonumber\\
%&=&\frac{1}{2}\sqrt{N+\frac{1}{2}}\int_{-\infty}^{\infty}\int_{-\infty}^{\infty}\int_{-\infty}^{\infty}S_{N}(x,y)\varepsilon(y,z)f'\left(\sqrt{4N}z\right)\varepsilon\varphi_{2N+1}(z)
%\varepsilon\varphi_{2N}(x)f'\left(\sqrt{4N}x\right)dxdydz\nonumber\\
%&=&\frac{1}{4}\sqrt{N+\frac{1}{2}}\int_{-\infty}^{\infty}\int_{-\infty}^{\infty}\left(\int_{z}^{\infty}S_{N}(x,y)dy-\int_{-\infty}^{z}S_{N}(x,y)dy\right)\varepsilon\varphi_{2N+1}(z)
%\varepsilon\varphi_{2N}(x)\nonumber\\
%&\cdot&f'\left(\sqrt{4N}z\right)f'\left(\sqrt{4N}x\right)dxdz.\nonumber
%\eea
%Let
%$$
%u=\sqrt{4N}x,\;\;v=\sqrt{4N}y,\;\;w=\sqrt{4N}z,
%$$
%then
%\bea
%&&\mathrm{Tr}\:\frac{1}{2}\sqrt{N+\frac{1}{2}}S_{N}\varepsilon f'(\varepsilon\varphi_{2N+1}\otimes\varepsilon\varphi_{2N})f'\nonumber\\
&=&\frac{1}{16N}\sqrt{N+\frac{1}{2}}\int_{-\infty}^{\infty}\int_{-\infty}^{\infty}\left[\int_{w}^{\infty}\frac{1}{\sqrt{4N}}S_{N}\left(\frac{u}{\sqrt{4N}},\frac{v}{\sqrt{4N}}\right)dv
-\int_{-\infty}^{w}\frac{1}{\sqrt{4N}}S_{N}\left(\frac{u}{\sqrt{4N}},\frac{v}{\sqrt{4N}}\right)dv\right]\nonumber\\
&&\varepsilon\varphi_{2N}\left(\frac{u}{\sqrt{4N}}\right)\varepsilon\varphi_{2N+1}\left(\frac{w}{\sqrt{4N}}\right)f'(u)f'(w)dudw\nonumber\\
&=&O\left(N^{-\frac{3}{2}}\right),\;N\to\infty.\nonumber
\eea
The eighth term can be computed in a similar manner, and we have,
% we see $\mathrm{Tr}\:\left(N+\frac{1}{2}\right)(\varepsilon\varphi_{2N+1}\otimes\varphi_{2N}f)(\varepsilon\varphi_{2N+1}\otimes\varphi_{2N}f)$,
%\bea
%&&\mathrm{Tr}\:\left(N+\frac{1}{2}\right)(\varepsilon\varphi_{2N+1}\otimes\varphi_{2N}f)(\varepsilon\varphi_{2N+1}\otimes\varphi_{2N}f)\nonumber\\
%&=&\left(N+\frac{1}{2}\right)\int_{-\infty}^{\infty}\int_{-\infty}^{\infty}\varepsilon\varphi_{2N+1}(x)\varepsilon\varphi_{2N+1}(y)\varphi_{2N}(x)\varphi_{2N}(y)
%f\left(\sqrt{4N}x\right)f\left(\sqrt{4N}y\right)dx dy.\nonumber
%\eea
%Let
%$$
%u=\sqrt{4N}x,\;\;v=\sqrt{4N}y,
%$$
%then
\bea
&&\mathrm{Tr}\:\left(N+\frac{1}{2}\right)(\varepsilon\varphi_{2N+1}\otimes\varphi_{2N}f)(\varepsilon\varphi_{2N+1}\otimes\varphi_{2N}f)\nonumber\\
&=&\frac{N+\frac{1}{2}}{4N}\int_{-\infty}^{\infty}\int_{-\infty}^{\infty}\varepsilon\varphi_{2N+1}\left(\frac{u}{\sqrt{4N}}\right)
\varepsilon\varphi_{2N+1}\left(\frac{v}{\sqrt{4N}}\right)\varphi_{2N}\left(\frac{u}{\sqrt{4N}}\right)\varphi_{2N}\left(\frac{v}{\sqrt{4N}}\right)
f(u)f(v)du dv\nonumber\\
&=&O\left(N^{-1}\right),\;N\to\infty.\nonumber
\eea
Proceeding in a similar manner with the ninth term, we have,
%Ninthly, we see $\mathrm{Tr}\:\left(N+\frac{1}{2}\right)(\varepsilon\varphi_{2N+1}\otimes\varphi_{2N}f)(\varepsilon\varphi_{2N+1}\otimes\varepsilon\varphi_{2N})f'$,
%\bea
%&&\mathrm{Tr}\:\left(N+\frac{1}{2}\right)(\varepsilon\varphi_{2N+1}\otimes\varphi_{2N}f)(\varepsilon\varphi_{2N+1}\otimes\varepsilon\varphi_{2N})f'\nonumber\\
%&=&\left(N+\frac{1}{2}\right)\int_{-\infty}^{\infty}\int_{-\infty}^{\infty}\varepsilon\varphi_{2N+1}(x)\varepsilon\varphi_{2N+1}(y)\varepsilon\varphi_{2N}(x)\varphi_{2N}(y)
%f'\left(\sqrt{4N}x\right)f\left(\sqrt{4N}y\right)dx dy.\nonumber
%\eea
%Let
%$$
%u=\sqrt{4N}x,\;\;v=\sqrt{4N}y,
%$$
%then
\bea
&&\mathrm{Tr}\:\left(N+\frac{1}{2}\right)(\varepsilon\varphi_{2N+1}\otimes\varphi_{2N}f)(\varepsilon\varphi_{2N+1}\otimes\varepsilon\varphi_{2N})f'\nonumber\\
&=&\frac{N+\frac{1}{2}}{4N}\int_{-\infty}^{\infty}\int_{-\infty}^{\infty}\varepsilon\varphi_{2N+1}\left(\frac{u}{\sqrt{4N}}\right)\varepsilon\varphi_{2N+1}
\left(\frac{v}{\sqrt{4N}}\right)\varepsilon\varphi_{2N}\left(\frac{u}{\sqrt{4N}}\right)\varphi_{2N}\left(\frac{v}{\sqrt{4N}}\right)f'(u)f(v)du dv\nonumber\\
&=&O\left(N^{-\frac{3}{2}}\right),\;N\to\infty.\nonumber
\eea
The tenth and last term in ${\rm Tr}T^2$, becomes
%Lastly, we see $\mathrm{Tr}\:\frac{1}{4}\left(N+\frac{1}{2}\right)(\varepsilon\varphi_{2N+1}\otimes\varepsilon\varphi_{2N})f'(\varepsilon\varphi_{2N+1}\otimes\varepsilon\varphi_{2N})f'$,
%\bea
%&&\mathrm{Tr}\:\frac{1}{4}\left(N+\frac{1}{2}\right)(\varepsilon\varphi_{2N+1}\otimes\varepsilon\varphi_{2N})f'(\varepsilon\varphi_{2N+1}\otimes\varepsilon\varphi_{2N})f'\nonumber\\
%&=&\frac{1}{4}\left(N+\frac{1}{2}\right)\int_{-\infty}^{\infty}\int_{-\infty}^{\infty}\varepsilon\varphi_{2N+1}(x)\varepsilon\varphi_{2N+1}(y)\varepsilon\varphi_{2N}(x)
%\varepsilon\varphi_{2N}(y)f'\left(\sqrt{4N}x\right)f'\left(\sqrt{4N}y\right)dx dy.\nonumber
%\eea
%Let
%$$
%u=\sqrt{4N}x,\;\;v=\sqrt{4N}y,
%$$
%then
\bea
&&\mathrm{Tr}\:\frac{1}{4}\left(N+\frac{1}{2}\right)(\varepsilon\varphi_{2N+1}\otimes\varepsilon\varphi_{2N})f'(\varepsilon\varphi_{2N+1}\otimes\varepsilon\varphi_{2N})f'\nonumber\\
&=&\frac{N+\frac{1}{2}}{16N}\int_{-\infty}^{\infty}\int_{-\infty}^{\infty}\varepsilon\varphi_{2N+1}\left(\frac{u}{\sqrt{4N}}\right)
\varepsilon\varphi_{2N+1}\left(\frac{v}{\sqrt{4N}}\right)\varepsilon\varphi_{2N}\left(\frac{u}{\sqrt{4N}}\right)
\varepsilon\varphi_{2N}\left(\frac{v}{\sqrt{4N}}\right)f'(u)f'(v)du dv\nonumber\\
&=&O\left(N^{-2}\right),\;N\to\infty.\nonumber
\eea
Hence, the large $N$ behavior of (\ref{t2}), reads,
\bea
\mathrm{Tr}\:T^{2}
&=&\frac{1}{\pi^{2}}\int_{-\infty}^{\infty}\int_{-\infty}^{\infty}\left[\frac{\sin(x-y)}{x-y}\right]^{2}f(x)f(y)dx dy\nonumber\\
&+&\frac{1}{2\pi^{2}\sqrt{N}}\int_{-\infty}^{\infty}\int_{-\infty}^{\infty}\frac{\sin(x-y)}{x-y}\:\mathrm{Si}(x-y)f'(x)f(y)dx dy\nonumber\\
&-&\frac{(-1)^{N}}{\pi\sqrt{2\pi N}}\int_{-\infty}^{\infty}\int_{-\infty}^{\infty}\frac{\sin(x-y)}{x-y}\cos x\:f(x)f(y)dx dy
+O\left(N^{-1}\right),\;N\to\infty.\nonumber
\eea
We are now in a position to compute the mean and variance of the (scaled) linear statistics $\sum_{j=1}^{N}F\left(\sqrt{4N}x_{j}\right)$,
which are obtained as the coefficients of $\lambda$ and $\lambda^{2},$ of $\log\det(I+T).$

Since
$$
f\left(\sqrt{4N}x\right)\approx-\lambda F\left(\sqrt{4N}x\right)+\frac{\lambda^{2}}{2}F^{2}\left(\sqrt{4N}x\right),
$$
we replace $f$ with $-\lambda F+\frac{\lambda^{2}}{2}F^{2}$ in the expression of $\mathrm{Tr}\:T$ and $\mathrm{Tr}\:T^{2}$. A minor rearrangement
gives,
\bea
&&\log\det\left(I+T\right)\nonumber\\
&=&-\lambda\Bigg\{\frac{1}{\pi}\int_{-\infty}^{\infty}F(x)dx-\frac{(-1)^{N}}{2\sqrt{2\pi N}}\int_{-\infty}^{\infty}\cos x\: F(x)dx
+O\left(N^{-1}\right)\Bigg\}\nonumber\\
&+&\frac{\lambda^{2}}{2}\Bigg\{\frac{1}{\pi}\int_{-\infty}^{\infty}F^{2}(x)dx
-\frac{1}{\pi^{2}}\int_{-\infty}^{\infty}\int_{-\infty}^{\infty}\left[\frac{\sin(x-y)}{x-y}\right]^{2}F(x)F(y)dx dy\nonumber\\
&-&\frac{1}{2\pi^{2}\sqrt{N}}\int_{-\infty}^{\infty}\int_{-\infty}^{\infty}\frac{\sin(x-y)}{x-y}\:\mathrm{Si}(x-y)F'(x)F(y)dx dy\nonumber\\
&-&\frac{(-1)^{N}}{2\sqrt{2\pi N}}\left[\int_{-\infty}^{\infty}\cos x\: F^{2}(x)dx-
\frac{2}{\pi}\int_{-\infty}^{\infty}\int_{-\infty}^{\infty}\frac{\sin(x-y)}{x-y}\cos x\:F(x)F(y)dx dy\right]+O\left(N^{-1}\right)\Bigg\},\;\;N\rightarrow\infty.\nonumber
\eea
Denote by $\mu_{N}^{(GSE)}$ and $\mathcal{V}_{N}^{(GSE)}$ the mean and variance of the linear statistics
$\sum_{j=1}^{N}F\left(\sqrt{4N}x_{j}\right)$, then we have obtained,
the large $N$ corrections of these quantities.
\begin{theorem}
As $N\rightarrow\infty$,
$$
\mu_{N}^{(GSE)}=\frac{1}{2}\:\mu_{N}^{(GUE)}-\frac{(-1)^{N}}{4\sqrt{2\pi N}}\int_{-\infty}^{\infty}\cos x\: F(x)dx+O\left(N^{-1}\right),
$$
\bea
\mathcal{V}_{N}^{(GSE)}
&=&\frac{1}{2}\:\mathcal{V}_{N}^{(GUE)}-\frac{1}{4\pi^{2}\sqrt{N}}\int_{-\infty}^{\infty}\int_{-\infty}^{\infty}\frac{\sin(x-y)}{x-y}\:\mathrm{Si}(x-y)F'(x)F(y)dx dy\nonumber\\
&-&\frac{(-1)^{N}}{4\sqrt{2\pi N}}\left[\int_{-\infty}^{\infty}\cos x\: F^{2}(x)dx-
\frac{2}{\pi}\int_{-\infty}^{\infty}\int_{-\infty}^{\infty}\frac{\sin(x-y)}{x-y}\cos x\:F(x)F(y)dx dy\right]+O\left(N^{-1}\right),\nonumber
\eea
where $\mu_{N}^{(GUE)}$ and $\mathcal{V}_{N}^{(GUE)}$ for $N\rightarrow\infty$ are given in (\ref{guem}) and (\ref{guev}) respectively.
\end{theorem}

\subsection{LSE}
We study the case with the Laguerre background, namely, the weight,\\
$w(x)=x^{\alpha}\mathrm{e}^{-x},\;\;\alpha>0,\;\;x\in \mathbb{R}^+.$
%The parameter $\alpha$ maybe continued down to $\alpha>-1.$
The idea is to choose special $\psi_{j}$ so that  $M^{(4)}$ takes on the simplest possible form. To this end, let
\be
\psi_{2j+1}(x):=\frac{1}{\sqrt{2}}\varphi_{2j+1}^{(\alpha-1)}(x),\;\;\;j=0,1,2,\ldots,\label{psia1}
\ee
\be
\psi_{2j}(x):=-\frac{1}{\sqrt{2}}\varepsilon\widetilde{\varphi}_{2j+1}^{(\alpha-1)}(x),\;\;\;j=0,1,2,\ldots,\label{psia2}
\ee
where $\varphi_{j}^{(\alpha-1)}(x)$ and $\widetilde{\varphi}_{j}^{(\alpha-1)}(x)$ are given by
\be
\varphi_{j}^{(\alpha-1)}(x)=\frac{L_{j}^{(\alpha-1)}(x)}{c_{j}^{(\alpha-1)}}
x^{\frac{\alpha}{2}}\mathrm{e}^{-\frac{x}{2}},\;\;j=0,1,2,\ldots,\label{phij}
\ee
$$
\widetilde{\varphi}_{j}^{(\alpha-1)}(x)=
\frac{L_{j}^{(\alpha-1)}(x)}{c_{j}^{(\alpha-1)}}x^{\frac{\alpha}{2}-1}\mathrm{e}^{-\frac{x}{2}},
\;\;j=0,1,2,\ldots.
$$
Here $L_{j}^{(\alpha)}(x),\:j=0,1,2,\ldots, $ are the Laguerre polynomials, with the orthogonality condition,
$$
\int_{0}^{\infty}L_{j}^{(\alpha)}(x)L_{k}^{(\alpha)}(x)x^{\alpha}\mathrm{e}^{-x}dx=\left(c_{j}^{(\alpha)}\right)^{2}\delta_{jk},
\;\;\;\;c_{j}^{(\alpha)}=\sqrt{\frac{\Gamma(j+\alpha+1)}{\Gamma(j+1)}}.
$$
It is easy to see that
$$
\int_{0}^{\infty}\varphi_{j}^{(\alpha-1)}(x)\widetilde{\varphi}_{k}^{(\alpha-1)}(x)dx=\delta_{jk},\;\;j,k=0,1,2,\ldots.
$$
We now prove that (\ref{psia1}) and (\ref{psia2}) satisfy (\ref{psij4}), i.e.,
$\psi_{j}(x)=\pi_{j}(x)x^{\frac{\alpha}{2}}\mathrm{e}^{-\frac{x}{2}},\: j=0,1,2,\ldots$, where $\pi_{j}(x)$ is a polynomial of degree $j$.

\begin{theorem}
$\psi_{j}(x)=\pi_{j}(x)x^{\frac{\alpha}{2}}\mathrm{e}^{-\frac{x}{2}},\: j=0,1,2,\ldots$, where $\pi_{j}(x)$ is a polynomial of degree $j$.
\end{theorem}

\begin{proof}
We prove this by considering two cases, $j$ odd, and $j$ even. If $j=2n+1$, then by (\ref{psia1}),
$$
\pi_{2n+1}(x)=\frac{1}{\sqrt{2}c_{2n+1}^{(\alpha-1)}}L_{2n+1}^{(\alpha-1)}(x).
$$
%Obviously $\pi_{2n+1}(x)$ is a polynomial of $2n+1$.

Let $j=2n$, then
$$
\psi_{2n}(x)=-\frac{1}{\sqrt{2}}\varepsilon\widetilde{\varphi}_{2n+1}(x)
=-\frac{1}{\sqrt{2}c_{2n+1}^{(\alpha-1)}}\int_{0}^{x}L_{2n+1}^{(\alpha-1)}(y)y^{\frac{\alpha}{2}-1}\mathrm{e}^{-\frac{y}{2}}dy,
$$
and we have used the fact $\int_{0}^{\infty}L_{2n+1}^{(\alpha-1)}(y)y^{\frac{\alpha}{2}-1}\mathrm{e}^{-\frac{y}{2}}dy=0,\: n=0,1,2,\ldots.$

Let us rewrite the above as
\be
\int_{0}^{x}L_{2n+1}^{(\alpha-1)}(y)y^{\frac{\alpha}{2}-1}\mathrm{e}^{-\frac{y}{2}}dy
=\widehat{\pi}_{2n}(x)x^{\frac{\alpha}{2}}\mathrm{e}^{-\frac{x}{2}},\;\;n=0,1,2,\ldots,\label{eq2}
\ee
where
$$
\widehat{\pi}_{2n}(x):=-\sqrt{2}c_{2n+1}^{(\alpha-1)}\pi_{2n}(x).
$$
Take a derivative on both sides,
$$
L_{2n+1}^{(\alpha-1)}(x)x^{\frac{\alpha}{2}-1}\mathrm{e}^{-\frac{x}{2}}=\left(x\widehat{\pi}_{2n}'(x)+\frac{\alpha}{2}\widehat{\pi}_{2n}(x)
-\frac{1}{2}x\widehat{\pi}_{2n}(x)\right)x^{\frac{\alpha}{2}-1}\mathrm{e}^{-\frac{x}{2}},
$$
%Multiply $x^{1-\frac{\alpha}{2}}\mathrm{e}^{\frac{x}{2}}$ on both sides,
which becomes,
\be
L_{2n+1}^{(\alpha-1)}(x)=x\widehat{\pi}_{2n}'(x)+\frac{\alpha}{2}\widehat{\pi}_{2n}(x)-\frac{1}{2}x\widehat{\pi}_{2n}(x).\label{pe}
\ee
Note that equation (\ref{eq2}) is equivalent to (\ref{pe}). Now we seek to solve (\ref{pe}).
%Actually, if $\widehat{\pi}_{2n}(x)$ is the solution of (\ref{eq2}), certainly it solves (\ref{pe}) from the above analysis. In turn, if $\widehat{\pi}_{2n}(x)$ is a solution of %(\ref{pe}), multiply $x^{\frac{\alpha}{2}-1}\mathrm{e}^{-\frac{x}{2}}$ on both sides, we get equation (\ref{eq3}), then integrate from 0 to $x$, we obtain
%$$
%\int_{0}^{x}L_{2n+1}^{(\alpha-1)}(y)y^{\frac{\alpha}{2}-1}\mathrm{e}^{-\frac{y}{2}}dy
%=\widehat{\pi}_{2n}(x)x^{\frac{\alpha}{2}}\mathrm{e}^{-\frac{x}{2}}+C,
%$$
%where $C$ is a constant. Take $x=0$ on both sides, we have $C=0$, hence $\widehat{\pi}_{2n}(x)$ satisfies equation (\ref{eq2}).

Suppose
$$
\widehat{\pi}_{2n}(x):=a_{0}(n)+a_{1}(n)x+a_{2}(n)x^{2}+\cdots+a_{2n-1}(n)x^{2n-1}+a_{2n}(n)x^{2n},
$$
we see that the right side of (\ref{pe}) is equal to
\bea
&&x\widehat{\pi}_{2n}'(x)+\frac{\alpha}{2}\widehat{\pi}_{2n}(x)-\frac{1}{2}x\widehat{\pi}_{2n}(x)\nonumber\\
&=&\frac{\alpha}{2}a_{0}(n)+\left[\left(1+\frac{\alpha}{2}\right)a_{1}(n)-\frac{1}{2}a_{0}(n)\right]x+\left[\left(2+\frac{\alpha}{2}\right)a_{2}(n)
-\frac{1}{2}a_{1}(n)\right]x^{2}+\cdots\nonumber\\
&+&\left[\left(2n-1+\frac{\alpha}{2}\right)a_{2n-1}(n)-\frac{1}{2}a_{2n-2}(n)\right]x^{2n-1}+\left[\left(2n+\frac{\alpha}{2}\right)a_{2n}(n)
-\frac{1}{2}a_{2n-1}(n)\right]x^{2n}\nonumber\\
&-&\frac{1}{2}a_{2n}(n)x^{2n+1}.\label{rhs}
\eea
On the other hand, the left side of (\ref{pe}) is equal to
\bea
L_{2n+1}^{(\alpha-1)}(x)
&=&{2n+\alpha \choose 2n+1}-{2n+\alpha \choose 2n}x+\frac{1}{2!}{2n+\alpha \choose 2n-1}x^{2}-\frac{1}{3!}{2n+\alpha \choose 2n-2}x^{3}\nonumber\\
&+&\cdots+\frac{1}{(2n)!}{2n+\alpha \choose 1}x^{2n}-\frac{1}{(2n+1)!}{2n+\alpha \choose 0}x^{2n+1}.\label{lhs}
\eea
Compare the coefficients of (\ref{rhs}) and (\ref{lhs}), we have the equations
\be
\begin{cases}
\frac{\alpha}{2}a_{0}(n)={2n+\alpha \choose 2n+1}\\
\left(1+\frac{\alpha}{2}\right)a_{1}(n)-\frac{1}{2}a_{0}(n)=-{2n+\alpha \choose 2n}\\
\left(2+\frac{\alpha}{2}\right)a_{2}(n)-\frac{1}{2}a_{1}(n)=\frac{1}{2!}{2n+\alpha \choose 2n-1}\\
\left(3+\frac{\alpha}{2}\right)a_{3}(n)-\frac{1}{2}a_{2}(n)=-\frac{1}{3!}{2n+\alpha \choose 2n-2}\\
\qquad\qquad\qquad\vdots\\
\left(2n-1+\frac{\alpha}{2}\right)a_{2n-1}(n)-\frac{1}{2}a_{2n-2}(n)=-\frac{1}{(2n-1)!}{2n+\alpha \choose 2}\\
\left(2n+\frac{\alpha}{2}\right)a_{2n}(n)-\frac{1}{2}a_{2n-1}(n)=-\frac{1}{(2n)!}{2n+\alpha \choose 1}\\
-\frac{1}{2}a_{2n}(n)=-\frac{1}{(2n+1)!}{2n+\alpha \choose 0}.
\end{cases}\label{eq}
\ee
By solving the first $2n+1$ equations in (\ref{eq}), we find,
$$
\begin{cases}
\frac{1}{2}a_{0}(n)=\frac{1}{\alpha}{2n+\alpha \choose 2n+1}\\
\frac{1}{2}a_{1}(n)=\frac{1}{\alpha(\alpha+2)}{2n+\alpha \choose 2n+1}-\frac{1}{\alpha+2}{2n+\alpha \choose 2n}\\
\frac{1}{2}a_{2}(n)=\frac{1}{\alpha(\alpha+2)(\alpha+4)}{2n+\alpha \choose 2n+1}-\frac{1}{(\alpha+2)(\alpha+4)}{2n+\alpha \choose 2n}
+\frac{1}{2!}\frac{1}{\alpha+4}{2n+\alpha \choose 2n-1}\\
\frac{1}{2}a_{3}(n)=\frac{1}{\alpha(\alpha+2)(\alpha+4)(\alpha+6)}{2n+\alpha \choose 2n+1}-\frac{1}{(\alpha+2)(\alpha+4)(\alpha+6)}{2n+\alpha \choose 2n}
+\frac{1}{2!}\frac{1}{(\alpha+4)(\alpha+6)}{2n+\alpha \choose 2n-1}-\frac{1}{3!}\frac{1}{\alpha+6}{2n+\alpha \choose 2n-2}\\
\qquad\qquad\qquad\vdots\\
\frac{1}{2}a_{2n}(n)=\frac{1}{\alpha(\alpha+2)(\alpha+4)\cdots(\alpha+4n)}{2n+\alpha \choose 2n+1}-\frac{1}{(\alpha+2)(\alpha+4)\cdots(\alpha+4n)}{2n+\alpha \choose 2n}+\frac{1}{2!}\frac{1}{(\alpha+4)(\alpha+6)\cdots(\alpha+4n)}{2n+\alpha \choose 2n-1}\\
\qquad\quad-\frac{1}{3!}\frac{1}{(\alpha+6)\cdots(\alpha+4n)}{2n+\alpha \choose 2n-2}+\cdots-\frac{1}{(2n-1)!}\frac{1}{(\alpha+4n-2)(\alpha+4n)}{2n+\alpha \choose 2}+\frac{1}{(2n)!}\frac{1}{\alpha+4n}{2n+\alpha \choose 1}.
\end{cases}
$$
The last equation of (\ref{eq}), simplifies to,
$$
a_{2n}(n)=\frac{2}{(2n+1)!}.
$$
By Lemma \ref{bi}, we see that linear system (\ref{eq}) is solvable.
Hence $\widehat{\pi}_{2n}(x)$ is a polynomial of degree $2n.$
%, so that $\pi_{2n}(x)$ is a polynomial of degree $2n$.
The proof is complete.
\end{proof}

With (\ref{psia1}) and (\ref{psia2}), we compute
$M^{(4)}:=\left(\int_{0}^{\infty}\left(\psi_{j}(x)\psi_{k}'(x)-\psi_{j}'(x)\psi_{k}(x)\right)dx\right)_{j,k=0}^{2N-1}$, resulting
in the following theorem.
\begin{theorem}
$$
M^{(4)}=\begin{pmatrix}
0&1&0&0&\cdots&0&0\\
-1&0&0&0&\cdots&0&0\\
0&0&0&1&\cdots&0&0\\
0&0&-1&0&\cdots&0&0\\
\vdots&\vdots&\vdots&\vdots&&\vdots&\vdots\\
0&0&0&0&\cdots&0&1\\
0&0&0&0&\cdots&-1&0
\end{pmatrix}_{2N\times 2N}.
$$
\end{theorem}

\begin{proof}
Let $m_{jk}$ be the $(j,k)$-entry of $M^{(4)}$, i.e.,
$$
m_{jk}:=\int_{0}^{\infty}\left(\psi_{j}(x)\psi_{k}'(x)-\psi_{j}'(x)\psi_{k}(x)\right)dx,\;\; j,k=0,1,\ldots,2N-1.
$$
We compute $m_{j,k}$ by considering four cases: $(j,k)=({\rm even},{\rm odd}),\;\;({\rm odd},{\rm even}),
\;\;({\rm even},{\rm even}),\;\;({\rm odd},{\rm odd}).$

For the $({\rm even},{\rm odd})$ case,
\bea
m_{2j, 2k+1}
&=&\int_{0}^{\infty}\left(\psi_{2j}(x)\psi_{2k+1}'(x)-\psi_{2j}'(x)\psi_{2k+1}(x)\right)dx\nonumber\\
&=&\int_{0}^{\infty}\psi_{2j}(x)\psi_{2k+1}'(x)dx-\int_{0}^{\infty}\psi_{2j}'(x)\psi_{2k+1}(x)dx\nonumber\\
&=&\left(\psi_{2j}(x)\psi_{2k+1}(x)|_{0}^{\infty}-\int_{0}^{\infty}\psi_{2j}'(x)\psi_{2k+1}(x)dx\right)
-\int_{0}^{\infty}\psi_{2j}'(x)\psi_{2k+1}(x)dx\nonumber\\
&=&-2\int_{0}^{\infty}\psi_{2j}'(x)\psi_{2k+1}(x)dx\nonumber\\
&=&-2\int_{0}^{\infty}\left(-\frac{1}{\sqrt{2}}\widetilde{\varphi}_{2j+1}^{(\alpha-1)}(x)\right)
\left(\frac{1}{\sqrt{2}}\varphi_{2k+1}^{(\alpha-1)}(x)\right)dx\nonumber\\
&=&\int_{0}^{\infty}\widetilde{\varphi}_{2j+1}^{(\alpha-1)}(x)\varphi_{2k+1}^{(\alpha-1)}(x)dx\nonumber\\
&=&\delta_{jk}.\nonumber
\eea
For the $({\rm odd},{\rm even})$ case, since $M^{(4)}$ is antisymmetric, we have
\bea
m_{2j+1,2k}
&=&-m_{2k, 2j+1}\nonumber\\
&=&-\delta_{j,k}.\nonumber
\eea
For the $({\rm even},{\rm even})$ case and $j>k$,
\bea
m_{2j,2k}
&=&\int_{0}^{\infty}\left(\psi_{2j}(x)\psi_{2k}'(x)-\psi_{2j}'(x)\psi_{2k}(x)\right)dx\nonumber\\
&=&\int_{0}^{\infty}\psi_{2j}(x)\psi_{2k}'(x)dx-\int_{0}^{\infty}\psi_{2j}'(x)\psi_{2k}(x)dx\nonumber\\
&=&\psi_{2j}(x)\psi_{2k}(x)|_{0}^{\infty}-2\int_{0}^{\infty}\psi_{2j}'(x)\psi_{2k}(x)dx\nonumber\\
&=&-2\int_{0}^{\infty}\psi_{2j}'(x)\psi_{2k}(x)dx\nonumber\\
&=&-2\int_{0}^{\infty}\left(-\frac{1}{\sqrt{2}c_{2j+1}^{(\alpha-1)}}L_{2j+1}^{(\alpha-1)}(x)x^{\frac{\alpha}{2}-1}\mathrm{e}^{-\frac{x}{2}}\right)
\left(\pi_{2k}(x)x^{\frac{\alpha}{2}}\mathrm{e}^{-\frac{x}{2}}\right)dx\nonumber\\
&=&\frac{\sqrt{2}}{c_{2j+1}^{(\alpha-1)}}\int_{0}^{\infty}L_{2j+1}^{(\alpha-1)}(x)\pi_{2k}(x)x^{\alpha-1}\mathrm{e}^{-x}dx\nonumber\\
&=&0,\nonumber
\eea
since $\pi_{2k}(x)$ is a polynomial of degree $2k$ which is less than $2j+1$.

For the $({\rm odd},{\rm odd})$ case and $j>k$,
\bea
m_{2j+1,2k+1}
&=&\int_{0}^{\infty}\left(\psi_{2j+1}(x)\psi_{2k+1}'(x)-\psi_{2j+1}'(x)\psi_{2k+1}(x)\right)dx\nonumber\\
&=&\int_{0}^{\infty}\psi_{2j+1}(x)\psi_{2k+1}'(x)dx-\int_{0}^{\infty}\psi_{2j+1}'(x)\psi_{2k+1}(x)dx\nonumber\\
&=&\int_{0}^{\infty}\psi_{2j+1}(x)\psi_{2k+1}'(x)dx-\left(\psi_{2j+1}(x)\psi_{2k+1}(x)|_{0}^{\infty}
-\int_{0}^{\infty}\psi_{2j+1}(x)\psi_{2k+1}'(x)dx\right)\nonumber\\
&=&2\int_{0}^{\infty}\psi_{2j+1}(x)\psi_{2k+1}'(x)dx\nonumber\\
&=&\frac{1}{c_{2j+1}^{(\alpha-1)}c_{2k+1}^{(\alpha-1)}}\int_{0}^{\infty}L_{2j+1}^{(\alpha-1)}(x)x^{\frac{\alpha}{2}}\mathrm{e}^{-\frac{x}{2}}
\left(x\left(L_{2k+1}^{(\alpha-1)}(x)\right)'+\frac{\alpha}{2}L_{2k+1}^{(\alpha-1)}(x)-\frac{1}{2}x L_{2k+1}^{(\alpha-1)}(x)\right)x^{\frac{\alpha}{2}-1}
\mathrm{e}^{-\frac{x}{2}}dx\nonumber\\
&=&\frac{1}{c_{2j+1}^{(\alpha-1)}c_{2k+1}^{(\alpha-1)}}\int_{0}^{\infty}L_{2j+1}^{(\alpha-1)}(x)\left(x\left(L_{2k+1}^{(\alpha-1)}(x)\right)'
+\frac{\alpha}{2}L_{2k+1}^{(\alpha-1)}(x)-\frac{1}{2}x L_{2k+1}^{(\alpha-1)}(x)\right)x^{\alpha-1}\mathrm{e}^{-x}dx\nonumber\\
&=&0,\nonumber
\eea
since $x\left(L_{2k+1}^{(\alpha-1)}(x)\right)'+\frac{\alpha}{2}L_{2k+1}^{(\alpha-1)}(x)-\frac{1}{2}x L_{2k+1}^{(\alpha-1)}(x)$ is a polynomial of degree $2k+2$ which is less than $2j+1$.

If $j<k$, due to the fact that $M^{(4)}$ is antisymmetric,
\bea
m_{2j,2k}
&=&-m_{2k,2j}\nonumber\\
&=&0,\nonumber
\eea
\bea
m_{2j+1,2k+1}
&=&-m_{2k+1,2j+1}\nonumber\\
&=&0.\nonumber
\eea
Thus
$$
m_{2j,2k}=m_{2j+1,2k+1}=0,\;\;j,k=0,1,\ldots,N-1.
$$
This is just the desired form of $M^{(4)}$.
\end{proof}

It's clear that $\left(M^{(4)}\right)^{-1}=-M^{(4)}$, so $\mu_{2j,2j+1}=-1, \mu_{2j+1,2j}=1$, and $\mu_{jk}=0$ for other cases. The rest of this subsection is devoted to the determination
of $K_{N,4}^{(2,2)}(x,y)$,
\bea
K_{N,4}^{(2,2)}(x,y)
&=&-\sum_{j,k=0}^{2N-1}\psi_{j}(x)\mu_{jk}\psi_{k}'(y)\nonumber\\
&=&\sum_{j=0}^{N-1}\psi_{2j}(x)\psi_{2j+1}'(y)-\sum_{j=0}^{N-1}\psi_{2j+1}(x)\psi_{2j}'(y)\nonumber\\
&=&\frac{1}{2}\sum_{j=0}^{N-1}\varphi_{2j+1}^{(\alpha-1)}(x)\widetilde{\varphi}_{2j+1}^{(\alpha-1)}(y)
-\frac{1}{2}\sum_{j=0}^{N-1}\varepsilon\widetilde{\varphi}_{2j+1}^{(\alpha-1)}(x)\left[\varphi_{2j+1}^{(\alpha-1)}(y)\right]'.\nonumber
\eea
From (\ref{phij}), we see that,
\be
\left[\varphi_{j}^{(\alpha-1)}(x)\right]'=\frac{1}{c_{j}^{(\alpha-1)}}x^{\frac{\alpha}{2}-1}\mathrm{e}^{-\frac{x}{2}}\left[x\left(L_{j}^{(\alpha-1)}(x)\right)'
+\frac{\alpha-x}{2}
L_{j}^{(\alpha-1)}(x)\right].\label{sign}
\ee
Recall that the Laguerre polynomials $L_{j}^{(\alpha-1)}(x),$ satisfy the differentiation formulas \cite{Gradshteyn},
\be
x\left(L_{j}^{(\alpha-1)}(x)\right)'=j L_{j}^{(\alpha-1)}(x)-(j+\alpha-1)L_{j-1}^{(\alpha-1)}(x),\;\;j=0,1,2,\ldots,\label{ldf1}
\ee
\be
x\left(L_{j}^{(\alpha-1)}(x)\right)'=(j+1)L_{j+1}^{(\alpha-1)}(x)+(x-j-\alpha)L_{j}^{(\alpha-1)}(x),\;\;j=0,1,2,\ldots.\label{ldf2}
\ee
Summing (\ref{ldf1}) and  (\ref{ldf2}), and divide by 2, gives,
$$
x\left(L_{j}^{(\alpha-1)}(x)\right)'=\frac{j+1}{2}L_{j+1}^{(\alpha-1)}(x)-\frac{j+\alpha-1}{2}L_{j-1}^{(\alpha-1)}(x)
+\frac{x-\alpha}{2}L_{j}^{(\alpha-1)}(x),\;\;j=0,1,2,\ldots,
$$
or
$$
x\left(L_{j}^{(\alpha-1)}(x)\right)'+\frac{\alpha-x}{2}L_{j}^{(\alpha-1)}(x)=\frac{j+1}{2}L_{j+1}^{(\alpha-1)}(x)-\frac{j+\alpha-1}{2}L_{j-1}^{(\alpha-1)}(x)
,\;\;j=0,1,2,\ldots.
$$
Hence (\ref{sign}) becomes,
\bea
\left[\varphi_{j}^{(\alpha-1)}(x)\right]'
&=&\frac{1}{c_{j}^{(\alpha-1)}}x^{\frac{\alpha}{2}-1}\mathrm{e}^{-\frac{x}{2}}\left[\frac{j+1}{2}L_{j+1}^{(\alpha-1)}(x)-\frac{j+\alpha-1}{2}
L_{j-1}^{(\alpha-1)}(x)\right]\nonumber\\
&=&\frac{1}{2}\sqrt{(j+1)(j+\alpha)}\:\widetilde{\varphi}_{j+1}^{(\alpha-1)}(x)-\frac{1}{2}\sqrt{j(j+\alpha-1)}\:\widetilde{\varphi}_{j-1}^{(\alpha-1)}(x).\label{phijpx}
\eea
We see that $\left[\varphi_{j}^{(\alpha-1)}\right]'$ is a linear combination of $\widetilde{\varphi}_{j+1}^{(\alpha-1)}$ and $\widetilde{\varphi}_{j-1}^{(\alpha-1)}$, just like the GSE case studied in the last section. Replacing $j$ by $2j+1$ in (\ref{phijpx}), to find,
$$
\left[\varphi_{2j+1}^{(\alpha-1)}(y)\right]'=\sqrt{(j+1)\left(j+\frac{\alpha+1}{2}\right)}\widetilde{\varphi}_{2j+2}^{(\alpha-1)}(y)-\sqrt{\left(j+\frac{1}{2}\right)
\left(j+\frac{\alpha}{2}\right)}\widetilde{\varphi}_{2j}^{(\alpha-1)}(y).
$$
A straightforward computation shows that,
\bea
&&\sum_{j=0}^{N-1}\varepsilon\widetilde{\varphi}_{2j+1}^{(\alpha-1)}(x)\left[\varphi_{2j+1}^{(\alpha-1)}(y)\right]'\nonumber\\
&=&\sum_{j=0}^{N-1}\varepsilon\widetilde{\varphi}_{2j+1}^{(\alpha-1)}(x)\sqrt{(j+1)\left(j+\frac{\alpha+1}{2}\right)}\widetilde{\varphi}_{2j+2}^{(\alpha-1)}(y)
-\sum_{j=0}^{N-1}\varepsilon\widetilde{\varphi}_{2j+1}^{(\alpha-1)}(x)\sqrt{\left(j+\frac{1}{2}\right)\left(j+\frac{\alpha}{2}\right)}
\widetilde{\varphi}_{2j}^{(\alpha-1)}(y)\nonumber\\
&=&\sum_{j=0}^{N}\sqrt{j\left(j+\frac{\alpha-1}{2}\right)}\varepsilon\widetilde{\varphi}_{2j-1}^{(\alpha-1)}(x)\widetilde{\varphi}_{2j}^{(\alpha-1)}(y)
-\sum_{j=0}^{N}\sqrt{\left(j+\frac{1}{2}\right)\left(j+\frac{\alpha}{2}\right)}\varepsilon\widetilde{\varphi}_{2j+1}^{(\alpha-1)}(x)\widetilde{\varphi}_{2j}^{(\alpha-1)}(y)
\nonumber\\
&+&\sqrt{\left(N+\frac{1}{2}\right)\left(N+\frac{\alpha}{2}\right)}\varepsilon\widetilde{\varphi}_{2N+1}^{(\alpha-1)}(x)\widetilde{\varphi}_{2N}^{(\alpha-1)}(y)\nonumber\\
&=&\sum_{j=0}^{N}\left[\sqrt{j\left(j+\frac{\alpha-1}{2}\right)}\varepsilon\widetilde{\varphi}_{2j-1}^{(\alpha-1)}(x)
-\sqrt{\left(j+\frac{1}{2}\right)\left(j+\frac{\alpha}{2}\right)}\varepsilon\widetilde{\varphi}_{2j+1}^{(\alpha-1)}(x)\right]
\widetilde{\varphi}_{2j}^{(\alpha-1)}(y)\nonumber\\
&+&\sqrt{\left(N+\frac{1}{2}\right)\left(N+\frac{\alpha}{2}\right)}\varepsilon\widetilde{\varphi}_{2N+1}^{(\alpha-1)}(x)\widetilde{\varphi}_{2N}^{(\alpha-1)}(y).\nonumber
\eea
From (\ref{phijpx}) and Theorem \ref{de},
\bea
\varphi_{2j}^{(\alpha-1)}(x)
&=&\varepsilon\:D\:\varphi_{2j}^{(\alpha-1)}(x)\nonumber\\
&=&\varepsilon\left[\varphi_{2j}^{(\alpha-1)}(x)\right]'\nonumber\\
&=&\varepsilon\left[\frac{1}{2}\sqrt{(2j+1)(2j+\alpha)}\widetilde{\varphi}_{2j+1}^{(\alpha-1)}(x)-\frac{1}{2}\sqrt{2j(2j+\alpha-1)}
\widetilde{\varphi}_{2j-1}^{(\alpha-1)}(x)\right]\nonumber\\
&=&\sqrt{\left(j+\frac{1}{2}\right)\left(j+\frac{\alpha}{2}\right)}\varepsilon\widetilde{\varphi}_{2j+1}^{(\alpha-1)}(x)
-\sqrt{j\left(j+\frac{\alpha-1}{2}\right)}\varepsilon\widetilde{\varphi}_{2j-1}^{(\alpha-1)}(x).\nonumber
\eea
Hence, the sum, $\sum_{j=0}^{N-1}\varepsilon\widetilde{\varphi}_{2j+1}^{(\alpha-1)}(x)\left[\varphi_{2j+1}^{(\alpha-1)}(y)\right]'$ simplifies immediately, and leads to,
$$
\sum_{j=0}^{N-1}\varepsilon\widetilde{\varphi}_{2j+1}^{(\alpha-1)}(x)\left[\varphi_{2j+1}^{(\alpha-1)}(y)\right]'
=-\sum_{j=0}^{N}\varphi_{2j}^{(\alpha-1)}(x)\widetilde{\varphi}_{2j}^{(\alpha-1)}(y)
+\sqrt{\left(N+\frac{1}{2}\right)\left(N+\frac{\alpha}{2}\right)}\:\varepsilon\widetilde{\varphi}_{2N+1}^{(\alpha-1)}(x)\widetilde{\varphi}_{2N}^{(\alpha-1)}(y).
$$
It follows that,
\bea
K_{N,4}^{(2,2)}(x,y)
&=&\frac{1}{2}\sum_{j=0}^{N-1}\varphi_{2j+1}^{(\alpha-1)}(x)\widetilde{\varphi}_{2j+1}^{(\alpha-1)}(y)+\frac{1}{2}\sum_{j=0}^{N}\varphi_{2j}^{(\alpha-1)}(x)
\widetilde{\varphi}_{2j}^{(\alpha-1)}(y)\nonumber\\
&-&\frac{1}{2}\sqrt{\left(N+\frac{1}{2}\right)\left(N+\frac{\alpha}{2}\right)}\:\varepsilon\widetilde{\varphi}_{2N+1}^{(\alpha-1)}(x)\widetilde{\varphi}_{2N}^{(\alpha-1)}(y)\nonumber\\
&=&\frac{1}{2}\sum_{j=0}^{2N}\varphi_{j}^{(\alpha-1)}(x)\widetilde{\varphi}_{j}^{(\alpha-1)}(y)-\frac{1}{2}\sqrt{\left(N+\frac{1}{2}\right)\left(N+\frac{\alpha}{2}\right)}
\:\varepsilon\widetilde{\varphi}_{2N+1}^{(\alpha-1)}(x)\widetilde{\varphi}_{2N}^{(\alpha-1)}(y)\nonumber\\
&=&\frac{1}{2}S_{N}(x,y)-\frac{1}{2}\sqrt{\left(N+\frac{1}{2}\right)\left(N+\frac{\alpha}{2}\right)}
\:\varepsilon\widetilde{\varphi}_{2N+1}^{(\alpha-1)}(x)\widetilde{\varphi}_{2N}^{(\alpha-1)}(y),\nonumber
\eea
where
$$
S_{N}(x,y):=\sum_{j=0}^{2N}\varphi_{j}^{(\alpha-1)}(x)\widetilde{\varphi}_{j}^{(\alpha-1)}(y)
=-\sqrt{(2N+1)(2N+\alpha)}\:\frac{\varphi_{2N+1}^{(\alpha-1)}(x)\widetilde{\varphi}_{2N}^{(\alpha-1)}(y)
-\widetilde{\varphi}_{2N+1}^{(\alpha-1)}(y)\varphi_{2N}^{(\alpha-1)}(x)}{x-y}.
$$
Here we used the Christoffel-Darboux formula in the last equality.

By Theorem \ref{gn4s}, we have the following theorem.
\begin{theorem}
\bea
\left[G_{N}^{(4)}(f)\right]^{2}
&=&\det\Bigg(I+S_{N}f-\frac{1}{2}S_{N}\varepsilon f'-\sqrt{\left(N+\frac{1}{2}\right)\left(N+\frac{\alpha}{2}\right)}
\left(\varepsilon\widetilde{\varphi}_{2N+1}^{(\alpha-1)}\otimes\widetilde{\varphi}_{2N}^{(\alpha-1)}f\right)\nonumber\\
&-&\frac{1}{2}\sqrt{\left(N+\frac{1}{2}\right)\left(N+\frac{\alpha}{2}\right)}
\left(\varepsilon\widetilde{\varphi}_{2N+1}^{(\alpha-1)}\otimes\varepsilon\widetilde{\varphi}_{2N}^{(\alpha-1)}\right)f'\Bigg).\nonumber
\eea
\end{theorem}

\subsection{Large $N$ behavior of the LSE moment generating function}
Now consider the scaling limit of $\left[G_{N}^{(4)}(f)\right]^{2}$, write
$$
\left[G_{N}^{(4)}(f)\right]^{2}=:\det(I+T),
$$
where
\bea
T:&=&S_{N}f-\frac{1}{2}S_{N}\varepsilon f'-\sqrt{\left(N+\frac{1}{2}\right)\left(N+\frac{\alpha}{2}\right)}
\left(\varepsilon\widetilde{\varphi}_{2N+1}^{(\alpha-1)}\otimes\widetilde{\varphi}_{2N}^{(\alpha-1)}f\right)\nonumber\\
&-&\frac{1}{2}\sqrt{\left(N+\frac{1}{2}\right)\left(N+\frac{\alpha}{2}\right)}
\left(\varepsilon\widetilde{\varphi}_{2N+1}^{(\alpha-1)}\otimes\varepsilon\widetilde{\varphi}_{2N}^{(\alpha-1)}\right)f'.\nonumber
\eea

\begin{theorem}
$$
\lim_{N\rightarrow\infty}\frac{y}{4N}S_{N}\left(\frac{x^{2}}{8N},\frac{y^{2}}{8N}\right)=B^{(\alpha-1)}(x,y),
$$

$$
\lim_{N\rightarrow\infty}\frac{x}{4N}S_{N}\left(\frac{x^{2}}{8N},\frac{x^{2}}{8N}\right)=B^{(\alpha-1)}(x,x),\label{ba4xx}
$$
where $B^{(\alpha-1)}(x,y)$ and $B^{(\alpha-1)}(x,x)$ are given by (\ref{bxy}) and (\ref{bxx}) with $\alpha$ replaced by $\alpha-1$.
\end{theorem}

\begin{theorem}
$$
\widetilde{\varphi}_{2N}^{(\alpha-1)}\left(\frac{x^{2}}{8N}\right)= 2\:(2N)^{\frac{1}{2}}\:\frac{J_{\alpha-1}(x)}{x}+O\left(N^{-\frac{3}{2}}\right),\;\;N\rightarrow\infty,
$$

$$
\varepsilon\widetilde{\varphi}_{2N}^{(\alpha-1)}\left(\frac{x^{2}}{8N}\right)= (2N)^{-\frac{1}{2}}\left[\int_{0}^{x}J_{\alpha-1}(y)dy-1\right]+O\left(N^{-\frac{5}{2}}\right),
\;\;N\rightarrow\infty,
$$

$$
\varepsilon\widetilde{\varphi}_{2N+1}^{(\alpha-1)}\left(\frac{x^{2}}{8N}\right)= (2N)^{-\frac{1}{2}}\int_{0}^{x}J_{\alpha-1}(y)dy+O\left(N^{-\frac{5}{2}}\right),
\;\;N\rightarrow\infty.\label{lse}
$$
\end{theorem}

\begin{proof}
Recall the asymptotic formula of the Laguerre polynomials \cite{Szego},
$$
L_{N}^{(\alpha)}(x)x^{\frac{\alpha}{2}}\mathrm{e}^{-\frac{x}{2}}=\frac{\Gamma(N+\alpha+1)}{\Gamma(N+1)}\left(N+\frac{\alpha+1}{2}\right)^{-\frac{\alpha}{2}}
J_{\alpha}\left(\sqrt{(4N+2\alpha+2)x}\right)+x^{\frac{5}{4}}O\left(N^{\frac{\alpha}{2}-\frac{3}{4}}\right),\;\;N\rightarrow\infty.
$$
We find,
\bea
\widetilde{\varphi}_{2N}^{(\alpha-1)}\left(\frac{x^{2}}{8N}\right)
&=&\sqrt{\frac{\Gamma(2N+1)}{\Gamma(2N+\alpha)}}L_{2N}^{(\alpha-1)}\left(\frac{x^{2}}{8N}\right)\left(\frac{x^{2}}{8N}\right)^{\frac{\alpha}{2}-1}\mathrm{e}^{-\frac{x^{2}}{16N}}
\nonumber\\
&=&\sqrt{\frac{\Gamma(2N+1)}{\Gamma(2N+\alpha)}}\cdot\left(\frac{x^{2}}{8N}\right)^{-\frac{1}{2}}\left[\frac{\Gamma(2N+\alpha)}{\Gamma(2N+1)}
\left(2N+\frac{\alpha}{2}\right)^{-\frac{\alpha-1}{2}}J_{\alpha-1}(x)+O\left(N^{\frac{\alpha}{2}-\frac{5}{2}}\right)\right]\nonumber\\
&=&2\:(2N)^{\frac{1}{2}}\:\frac{J_{\alpha-1}(x)}{x}+O\left(N^{-\frac{3}{2}}\right),\;\;N\rightarrow\infty,\nonumber
\eea
where we have used the formula,
$$
\frac{\Gamma(n+a)}{\Gamma(n+b)}= n^{a-b}\left[1+O\left(n^{-1}\right)\right],\;\;n\rightarrow\infty.
$$
Proceeding to $\varepsilon\widetilde{\varphi}_{2N}^{(\alpha-1)}(x)$, we have,
\bea
&&\varepsilon\widetilde{\varphi}_{2N}^{(\alpha-1)}(x)\nonumber\\
&=&\frac{1}{2}\left[\int_{0}^{x}\widetilde{\varphi}_{2N}^{(\alpha-1)}(y)dy-\int_{x}^{\infty}\widetilde{\varphi}_{2N}^{(\alpha-1)}(y)dy\right]\nonumber\\
&=&\frac{1}{2}\sqrt{\frac{\Gamma(2N+1)}{\Gamma(2N+\alpha)}}\left[\int_{0}^{x}L_{2N}^{(\alpha-1)}(y)y^{\frac{\alpha}{2}-1}\mathrm{e}^{-\frac{y}{2}}dy
-\int_{x}^{\infty}L_{2N}^{(\alpha-1)}(y)y^{\frac{\alpha}{2}-1}\mathrm{e}^{-\frac{y}{2}}dy\right]\nonumber\\
&=&\frac{1}{2}\sqrt{\frac{\Gamma(2N+1)}{\Gamma(2N+\alpha)}}\left[2\int_{0}^{x}L_{2N}^{(\alpha-1)}(y)y^{\frac{\alpha}{2}-1}\mathrm{e}^{-\frac{y}{2}}dy
-\int_{0}^{\infty}L_{2N}^{(\alpha-1)}(y)y^{\frac{\alpha}{2}-1}\mathrm{e}^{-\frac{y}{2}}dy\right]\nonumber\\
&=&\frac{1}{2}\sqrt{\frac{\Gamma(2N+1)}{\Gamma(2N+\alpha)}}\left[2\int_{0}^{x}L_{2N}^{(\alpha-1)}(y)y^{\frac{\alpha}{2}-1}\mathrm{e}^{-\frac{y}{2}}dy
-\frac{2^{\frac{\alpha}{2}}\Gamma\left(N+\frac{\alpha}{2}\right)}{\Gamma(N+1)}\right],\nonumber
\eea
where we have used the fact that
$$
\int_{0}^{\infty}L_{2N}^{(\alpha-1)}(y)y^{\frac{\alpha}{2}-1}\mathrm{e}^{-\frac{y}{2}}dy=\frac{2^{\frac{\alpha}{2}}\Gamma\left(N+\frac{\alpha}{2}\right)}{\Gamma(N+1)}.
$$
Continuing,
\bea
&&\varepsilon\widetilde{\varphi}_{2N}^{(\alpha-1)}\left(\frac{x^{2}}{8N}\right)\nonumber\\
&=&\frac{1}{2}\sqrt{\frac{\Gamma(2N+1)}{\Gamma(2N+\alpha)}}\left[\frac{1}{2N}\int_{0}^{x}L_{2N}^{(\alpha-1)}\left(\frac{y^{2}}{8N}\right)
\left(\frac{y^{2}}{8N}\right)^{\frac{\alpha}{2}-1}\mathrm{e}^{-\frac{y^{2}}{16N}}ydy
-\frac{2^{\frac{\alpha}{2}}\Gamma\left(N+\frac{\alpha}{2}\right)}{\Gamma(N+1)}\right]\nonumber\\
&=&(2N)^{-\frac{1}{2}}\left(2N+\frac{\alpha}{2}\right)^{-\frac{\alpha-1}{2}}\sqrt{\frac{\Gamma(2N+\alpha)}{\Gamma(2N+1)}}
\int_{0}^{x}J_{\alpha-1}(y)dy-2^{\frac{\alpha}{2}-1}\sqrt{\frac{\Gamma(2N+1)}{\Gamma(2N+\alpha)}}\cdot\frac{\Gamma\left(N+\frac{\alpha}{2}\right)}{\Gamma(N+1)}
+O\left(N^{-\frac{5}{2}}\right)\nonumber\\
&=&(2N)^{-\frac{1}{2}}\left[\int_{0}^{x}J_{\alpha-1}(y)dy-1\right]+O\left(N^{-\frac{5}{2}}\right),\;\;N\rightarrow\infty.\nonumber
\eea
Similarly,
\bea
&&\varepsilon\widetilde{\varphi}_{2N+1}^{(\alpha-1)}(x)\nonumber\\
&=&\frac{1}{2}\left[\int_{0}^{x}\widetilde{\varphi}_{2N+1}^{(\alpha-1)}(y)dy-\int_{x}^{\infty}\widetilde{\varphi}_{2N+1}^{(\alpha-1)}(y)dy\right]\nonumber\\
&=&\frac{1}{2}\sqrt{\frac{\Gamma(2N+2)}{\Gamma(2N+\alpha+1)}}\left[\int_{0}^{x}L_{2N+1}^{(\alpha-1)}(y)y^{\frac{\alpha}{2}-1}\mathrm{e}^{-\frac{y}{2}}dy
-\int_{x}^{\infty}L_{2N+1}^{(\alpha-1)}(y)y^{\frac{\alpha}{2}-1}\mathrm{e}^{-\frac{y}{2}}dy\right]\nonumber\\
&=&\frac{1}{2}\sqrt{\frac{\Gamma(2N+2)}{\Gamma(2N+\alpha+1)}}\left[2\int_{0}^{x}L_{2N+1}^{(\alpha-1)}(y)y^{\frac{\alpha}{2}-1}\mathrm{e}^{-\frac{y}{2}}dy
-\int_{0}^{\infty}L_{2N+1}^{(\alpha-1)}(y)y^{\frac{\alpha}{2}-1}\mathrm{e}^{-\frac{y}{2}}dy\right]\nonumber\\
&=&\sqrt{\frac{\Gamma(2N+2)}{\Gamma(2N+\alpha+1)}}\int_{0}^{x}L_{2N+1}^{(\alpha-1)}(y)y^{\frac{\alpha}{2}-1}\mathrm{e}^{-\frac{y}{2}}dy,\nonumber
\eea
where we have used the fact that
$$
\int_{0}^{\infty}L_{2N+1}^{(\alpha-1)}(y)y^{\frac{\alpha}{2}-1}\mathrm{e}^{-\frac{y}{2}}dy=0.
$$
It follows that
\bea
&&\varepsilon\widetilde{\varphi}_{2N+1}^{(\alpha-1)}\left(\frac{x^{2}}{8N}\right)\nonumber\\
&=&\sqrt{\frac{\Gamma(2N+2)}{\Gamma(2N+\alpha+1)}}\int_{0}^{\frac{x^{2}}{8N}}L_{2N+1}^{(\alpha-1)}(y)y^{\frac{\alpha}{2}-1}\mathrm{e}^{-\frac{y}{2}}dy\nonumber\\
&=&\frac{1}{4N}\sqrt{\frac{\Gamma(2N+2)}{\Gamma(2N+\alpha+1)}}\int_{0}^{x}L_{2N+1}^{(\alpha-1)}\left(\frac{y^{2}}{8N}\right)\left(\frac{y^{2}}{8N}\right)^{\frac{\alpha}{2}-1}
\mathrm{e}^{-\frac{y^{2}}{16N}}ydy\nonumber\\
&=&(2N)^{-\frac{1}{2}}\left(2N+1+\frac{\alpha}{2}\right)^{-\frac{\alpha-1}{2}}\sqrt{\frac{\Gamma(2N+\alpha+1)}{\Gamma(2N+2)}}\int_{0}^{x}J_{\alpha-1}(y)dy
+O\left(N^{-\frac{5}{2}}\right)\nonumber\\
&=&(2N)^{-\frac{1}{2}}\int_{0}^{x}J_{\alpha-1}(y)dy+O\left(N^{-\frac{5}{2}}\right),\;\;N\rightarrow\infty.\nonumber
\eea
\end{proof}

We now use Theorem \ref{ba4xx} and Theorem \ref{lse} to compute $\mathrm{Tr}\:T$ and $\mathrm{Tr}\:T^{2}$ as $N\rightarrow\infty$. In the computations below, we replace $f(x)$ by $f\left(\sqrt{8Nx}\right)$ and $f'(x)$ by
$$
f'\left(\sqrt{8Nx}\right)=\frac{d}{\sqrt{8N}d\sqrt{x}}f\left(\sqrt{8Nx}\right).
$$

Consider $\mathrm{Tr}\:T$, which reads,
\bea
\mathrm{Tr}\:T&=&\mathrm{Tr}\:S_{N}f-\mathrm{Tr}\:\frac{1}{2}S_{N}\varepsilon f'-\mathrm{Tr}\:\sqrt{\left(N+\frac{1}{2}\right)\left(N+\frac{\alpha}{2}\right)}
\left(\varepsilon\widetilde{\varphi}_{2N+1}^{(\alpha-1)}\otimes\widetilde{\varphi}_{2N}^{(\alpha-1)}f\right)\nonumber\\
&-&\mathrm{Tr}\:\frac{1}{2}\sqrt{\left(N+\frac{1}{2}\right)\left(N+\frac{\alpha}{2}\right)}
\left(\varepsilon\widetilde{\varphi}_{2N+1}^{(\alpha-1)}\otimes\varepsilon\widetilde{\varphi}_{2N}^{(\alpha-1)}\right)f'.\nonumber
\eea
So we compute $\mathrm{Tr}\:T$ by calculating the four terms in the right side. The first term,
\bea
\mathrm{Tr}\:S_{N}f
&=&\int_{0}^{\infty}S_{N}(x,x)f\left(\sqrt{8Nx}\right)dx\nonumber\\
&=&\int_{0}^{\infty}\frac{x}{4N}S_{N}\left(\frac{x^{2}}{8N},\frac{x^{2}}{8N}\right)f(x)dx\nonumber\\
&\rightarrow&\int_{0}^{\infty}B^{(\alpha-1)}(x,x)f(x)dx,\;\;N\rightarrow\infty.\nonumber
\eea

The second term,
\bea
\mathrm{Tr}\:\frac{1}{2}S_{N}\varepsilon f'
&=&\frac{1}{2}\int_{0}^{\infty}\int_{0}^{\infty}S_{N}(x,y)\varepsilon(y,x)f'\left(\sqrt{8Nx}\right)dx dy\nonumber\\
&=&\frac{1}{4}\int_{0}^{\infty}\left[\int_{x}^{\infty}S_{N}(x,y)dy-\int_{0}^{x}S_{N}(x,y)dy\right]f'\left(\sqrt{8Nx}\right)dx.\nonumber
\eea
Let
$$
u=\sqrt{8Nx},\;\;v=\sqrt{8Ny},
$$
then
\bea
\mathrm{Tr}\:\frac{1}{2}S_{N}\varepsilon f'
&=&\frac{1}{16N}\int_{0}^{\infty}\left[\int_{u}^{\infty}\frac{1}{4N}S_{N}\left(\frac{u^{2}}{8N},\frac{v^{2}}{8N}\right)vdv
-\int_{0}^{u}\frac{1}{4N}S_{N}\left(\frac{u^{2}}{8N},\frac{v^{2}}{8N}\right)vdv\right]u\:f'(u)du\nonumber\\
&\rightarrow&\frac{1}{16N}\int_{0}^{\infty}\left[\int_{u}^{\infty}B^{(\alpha-1)}(u,v)dv
-\int_{0}^{u}B^{(\alpha-1)}(u,v)dv\right]u\:f'(u)du,\;\;N\rightarrow\infty.\nonumber
\eea

The third term,
\bea
&&\mathrm{Tr}\:\sqrt{\left(N+\frac{1}{2}\right)\left(N+\frac{\alpha}{2}\right)}
\left(\varepsilon\widetilde{\varphi}_{2N+1}^{(\alpha-1)}\otimes\widetilde{\varphi}_{2N}^{(\alpha-1)}f\right)\nonumber\\
&=&\sqrt{\left(N+\frac{1}{2}\right)\left(N+\frac{\alpha}{2}\right)}
\int_{0}^{\infty}\varepsilon\widetilde{\varphi}_{2N+1}^{(\alpha-1)}(x)\widetilde{\varphi}_{2N}^{(\alpha-1)}(x)f\left(\sqrt{8Nx}\right)dx\nonumber\\
&=&\frac{\sqrt{\left(N+\frac{1}{2}\right)\left(N+\frac{\alpha}{2}\right)}}{4N}
\int_{0}^{\infty}\varepsilon\widetilde{\varphi}_{2N+1}^{(\alpha-1)}\left(\frac{x^{2}}{8N}\right)\widetilde{\varphi}_{2N}^{(\alpha-1)}\left(\frac{x^{2}}{8N}\right)x f(x)dx\nonumber\\
&=&\frac{1}{2}\int_{0}^{\infty}\left[\int_{0}^{x}J_{\alpha-1}(y)dy\right]J_{\alpha-1}(x)f(x)dx+O\left(N^{-2}\right),\;\;N\rightarrow\infty.\nonumber
\eea

The fourth term,
\bea
&&\mathrm{Tr}\:\frac{1}{2}\sqrt{\left(N+\frac{1}{2}\right)\left(N+\frac{\alpha}{2}\right)}
\left(\varepsilon\widetilde{\varphi}_{2N+1}^{(\alpha-1)}\otimes\varepsilon\widetilde{\varphi}_{2N}^{(\alpha-1)}\right)f'\nonumber\\
&=&\frac{1}{2}\sqrt{\left(N+\frac{1}{2}\right)\left(N+\frac{\alpha}{2}\right)}
\int_{0}^{\infty}\varepsilon\widetilde{\varphi}_{2N+1}^{(\alpha-1)}(x)\varepsilon\widetilde{\varphi}_{2N}^{(\alpha-1)}(x)f'\left(\sqrt{8Nx}\right)dx\nonumber\\
&=&\frac{\sqrt{\left(N+\frac{1}{2}\right)\left(N+\frac{\alpha}{2}\right)}}{8N}
\int_{0}^{\infty}\varepsilon\widetilde{\varphi}_{2N+1}^{(\alpha-1)}\left(\frac{x^{2}}{8N}\right)\varepsilon\widetilde{\varphi}_{2N}^{(\alpha-1)}\left(\frac{x^{2}}{8N}\right)
x\:f'(x)dx\nonumber\\
&=&\frac{1}{16N}\int_{0}^{\infty}\left[\int_{0}^{x}J_{\alpha-1}(y)dy\right]\left[\int_{0}^{x}J_{\alpha-1}(y)dy-1\right]x\:f'(x)dx+O\left(N^{-3}\right),\;\;N\rightarrow\infty.\nonumber
\eea

Therefore,
\bea
\mathrm{Tr}\:T
&=&\int_{0}^{\infty}B^{(\alpha-1)}(x,x)f(x)dx-\frac{1}{2}\int_{0}^{\infty}\left[\int_{0}^{x}J_{\alpha-1}(y)dy\right]J_{\alpha-1}(x)f(x)dx\nonumber\\
&-&\frac{1}{16N}\int_{0}^{\infty}\left[\int_{x}^{\infty}B^{(\alpha-1)}(x,y)dy-\int_{0}^{x}B^{(\alpha-1)}(x,y)dy\right]x\:f'(x)dx\nonumber\\
&-&\frac{1}{16N}\int_{0}^{\infty}\left[\int_{0}^{x}J_{\alpha-1}(y)dy\right]\left[\int_{0}^{x}J_{\alpha-1}(y)dy-1\right]x\:f'(x)dx+O\left(N^{-2}\right),\;\;N\rightarrow\infty.\nonumber
\eea

Next, we compute $\mathrm{Tr}\:T^{2}$, where there are 10 traces,
\bea
\mathrm{Tr}\:T^{2}
&=&\mathrm{Tr}\:S_{N}f S_{N}f-\mathrm{Tr}\:S_{N}f S_{N}\varepsilon f'-\mathrm{Tr}\:2\sqrt{\left(N+\frac{1}{2}\right)\left(N+\frac{\alpha}{2}\right)}
S_{N}f\left(\varepsilon\widetilde{\varphi}_{2N+1}^{(\alpha-1)}\otimes\widetilde{\varphi}_{2N}^{(\alpha-1)}f\right)\nonumber\\
&-&\mathrm{Tr}\:\sqrt{\left(N+\frac{1}{2}\right)\left(N+\frac{\alpha}{2}\right)}S_{N}f\left(\varepsilon\widetilde{\varphi}_{2N+1}^{(\alpha-1)}
\otimes\varepsilon\widetilde{\varphi}_{2N}^{(\alpha-1)}\right)f'+\mathrm{Tr}\:\frac{1}{4}S_{N}\varepsilon f'S_{N}\varepsilon f'\nonumber\\
&+&\mathrm{Tr}\:\sqrt{\left(N+\frac{1}{2}\right)\left(N+\frac{\alpha}{2}\right)}
S_{N}\varepsilon f'\left(\varepsilon\widetilde{\varphi}_{2N+1}^{(\alpha-1)}\otimes\widetilde{\varphi}_{2N}^{(\alpha-1)}f\right)\nonumber\\
&+&\mathrm{Tr}\:\frac{1}{2}\sqrt{\left(N+\frac{1}{2}\right)\left(N+\frac{\alpha}{2}\right)}
S_{N}\varepsilon f'\left(\varepsilon\widetilde{\varphi}_{2N+1}^{(\alpha-1)}\otimes\varepsilon\widetilde{\varphi}_{2N}^{(\alpha-1)}\right)f'\nonumber\\
&+&\mathrm{Tr}\:\left(N+\frac{1}{2}\right)\left(N+\frac{\alpha}{2}\right)
\left(\varepsilon\widetilde{\varphi}_{2N+1}^{(\alpha-1)}\otimes\widetilde{\varphi}_{2N}^{(\alpha-1)}f\right)
\left(\varepsilon\widetilde{\varphi}_{2N+1}^{(\alpha-1)}\otimes\widetilde{\varphi}_{2N}^{(\alpha-1)}f\right)\nonumber\\
&+&\mathrm{Tr}\:\left(N+\frac{1}{2}\right)\left(N+\frac{\alpha}{2}\right)
\left(\varepsilon\widetilde{\varphi}_{2N+1}^{(\alpha-1)}\otimes\widetilde{\varphi}_{2N}^{(\alpha-1)}f\right)
\left(\varepsilon\widetilde{\varphi}_{2N+1}^{(\alpha-1)}\otimes\varepsilon\widetilde{\varphi}_{2N}^{(\alpha-1)}\right)f'\nonumber\\
&+&\mathrm{Tr}\:\frac{1}{4}\left(N+\frac{1}{2}\right)\left(N+\frac{\alpha}{2}\right)
\left(\varepsilon\widetilde{\varphi}_{2N+1}^{(\alpha-1)}\otimes\varepsilon\widetilde{\varphi}_{2N}^{(\alpha-1)}\right)f'
\left(\varepsilon\widetilde{\varphi}_{2N+1}^{(\alpha-1)}\otimes\varepsilon\widetilde{\varphi}_{2N}^{(\alpha-1)}\right)f'.\nonumber
\eea
In the following, we need to calculate the traces on the right side term by term. The first term,
\bea
\mathrm{Tr}\:S_{N}f S_{N}f
&=&\int_{0}^{\infty}\int_{0}^{\infty}S_{N}(x,y)f\left(\sqrt{8Ny}\right)S_{N}(y,x)f\left(\sqrt{8Nx}\right)dxdy\nonumber\\
&=&\int_{0}^{\infty}\int_{0}^{\infty}\frac{y}{4N}S_{N}\left(\frac{x^{2}}{8N},\frac{y^{2}}{8N}\right)\frac{x}{4N}S_{N}\left(\frac{y^{2}}{8N},\frac{x^{2}}{8N}\right) f(x)f(y)dxdy\nonumber\\
&\rightarrow&\int_{0}^{\infty}\int_{0}^{\infty}B^{(\alpha-1)}(x,y)B^{(\alpha-1)}(y,x)f(x)f(y)dxdy,\;\;N\rightarrow\infty.\nonumber
\eea

The second term,
\bea
&&\mathrm{Tr}\:S_{N}f S_{N}\varepsilon f'\nonumber\\
&=&\frac{1}{8N}\int_{0}^{\infty}\int_{0}^{\infty}\frac{v}{4N}S_{N}\left(\frac{u^{2}}{8N},\frac{v^{2}}{8N}\right)
\left[\int_{u}^{\infty}\frac{w}{4N}S_{N}\left(\frac{v^{2}}{8N},\frac{w^{2}}{8N}\right)dw
-\int_{0}^{u}\frac{w}{4N}S_{N}\left(\frac{v^{2}}{8N},\frac{w^{2}}{8N}\right)dw\right]\nonumber\\
&&u\:f'(u)f(v)du dv\nonumber\\
&\rightarrow&\frac{1}{8N}\int_{0}^{\infty}\int_{0}^{\infty}B^{(\alpha-1)}(u,v)
\left[\int_{u}^{\infty}B^{(\alpha-1)}(v,w)dw-\int_{0}^{u}B^{(\alpha-1)}(v,w)dw\right]u\:f'(u)f(v)du dv,\;\;N\rightarrow\infty.\nonumber
\eea

The third term,
\bea
&&\mathrm{Tr}\:2\sqrt{\left(N+\frac{1}{2}\right)\left(N+\frac{\alpha}{2}\right)}
S_{N}f\left(\varepsilon\widetilde{\varphi}_{2N+1}^{(\alpha-1)}\otimes\widetilde{\varphi}_{2N}^{(\alpha-1)}f\right)\nonumber\\
&=&\frac{1}{2N}\sqrt{\left(N+\frac{1}{2}\right)\left(N+\frac{\alpha}{2}\right)}\int_{0}^{\infty}\int_{0}^{\infty}
\frac{1}{4N}S_{N}\left(\frac{u^{2}}{8N},\frac{v^{2}}{8N}\right)v\:\widetilde{\varphi}_{2N}^{(\alpha-1)}\left(\frac{u^{2}}{8N}\right)
\varepsilon\widetilde{\varphi}_{2N+1}^{(\alpha-1)}\left(\frac{v^{2}}{8N}\right)u f(u)f(v)dudv\nonumber\\
&=&\int_{0}^{\infty}\int_{0}^{\infty}B^{(\alpha-1)}(u,v)J_{\alpha-1}(u)\left[\int_{0}^{v}J_{\alpha-1}(t)dt\right]f(u)f(v)dudv+O\left(N^{-2}\right),\;\;N\rightarrow\infty.\nonumber
\eea

The fourth term,
\bea
&&\mathrm{Tr}\:\sqrt{\left(N+\frac{1}{2}\right)\left(N+\frac{\alpha}{2}\right)}S_{N}f\left(\varepsilon\widetilde{\varphi}_{2N+1}^{(\alpha-1)}
\otimes\varepsilon\widetilde{\varphi}_{2N}^{(\alpha-1)}\right)f'\nonumber\\
&=&\frac{1}{4N}\sqrt{\left(N+\frac{1}{2}\right)\left(N+\frac{\alpha}{2}\right)}\int_{0}^{\infty}\int_{0}^{\infty}\frac{1}{4N}S_{N}\left(\frac{u^{2}}{8N},\frac{v^{2}}{8N}\right)v\:f(v)
\varepsilon\widetilde{\varphi}_{2N+1}^{(\alpha-1)}\left(\frac{v^{2}}{8N}\right)\varepsilon\widetilde{\varphi}_{2N}^{(\alpha-1)}\left(\frac{u^{2}}{8N}\right)u\:f'(u)dudv\nonumber\\
&=&\frac{1}{8N}\int_{0}^{\infty}\int_{0}^{\infty}B^{(\alpha-1)}(u,v)\left[\int_{0}^{u}J_{\alpha-1}(t)dt-1\right]\left[\int_{0}^{v}J_{\alpha-1}(t)dt\right]u\:f'(u)f(v)dudv
+O\left(N^{-3}\right),\;\;N\rightarrow\infty.\nonumber
\eea

The fifth term,
\bea
&&\mathrm{Tr}\:\frac{1}{4}S_{N}\varepsilon f'S_{N}\varepsilon f'\nonumber\\
&=&\frac{1}{4}\int_{0}^{\infty}\int_{0}^{\infty}\int_{0}^{\infty}\int_{0}^{\infty}S_{N}(x,y)\varepsilon(y,z)f'\left(\sqrt{8Nz}\right)S_{N}(z,t)\varepsilon(t,x) f'\left(\sqrt{8Nx}\right)dxdydzdt\nonumber\\
&=&\frac{1}{16}\int_{0}^{\infty}\int_{0}^{\infty}\left[\int_{z}^{\infty}S_{N}(x,y)dy-\int_{0}^{z}S_{N}(x,y)dy\right]
\left[\int_{x}^{\infty}S_{N}(z,t)dt-\int_{0}^{x}S_{N}(z,t)dt\right]\nonumber\\
&&f'\left(\sqrt{8Nx}\right)f'\left(\sqrt{8Nz}\right)dxdz\nonumber\\
&=&\frac{1}{256N^{2}}\int_{0}^{\infty}\int_{0}^{\infty}\left[\int_{w}^{\infty}\frac{1}{4N}S_{N}\left(\frac{u^{2}}{8N},\frac{v^{2}}{8N}\right)vdv
-\int_{0}^{w}\frac{1}{4N}S_{N}\left(\frac{u^{2}}{8N},\frac{v^{2}}{8N}\right)vdv\right]\nonumber\\
&&\left[\int_{u}^{\infty}\frac{1}{4N}S_{N}\left(\frac{w^{2}}{8N},\frac{\tau^{2}}{8N}\right)\tau d\tau
-\int_{0}^{u}\frac{1}{4N}S_{N}\left(\frac{w^{2}}{8N},\frac{\tau^{2}}{8N}\right)\tau d\tau\right]u\:w\:f'(u)f'(w)dudw\nonumber\\
&=&O\left(\frac{1}{N^{2}}\right),\;\;N\rightarrow\infty.\nonumber
\eea

The sixth term,
\bea
&&\mathrm{Tr}\:\sqrt{\left(N+\frac{1}{2}\right)\left(N+\frac{\alpha}{2}\right)}
S_{N}\varepsilon f'\left(\varepsilon\widetilde{\varphi}_{2N+1}^{(\alpha-1)}\otimes\widetilde{\varphi}_{2N}^{(\alpha-1)}f\right)\nonumber\\
&=&\frac{1}{32N^{2}}\sqrt{\left(N+\frac{1}{2}\right)\left(N+\frac{\alpha}{2}\right)}\int_{0}^{\infty}\int_{0}^{\infty}
\left[\int_{w}^{\infty}\frac{v}{4N}S_{N}\left(\frac{u^{2}}{8N},\frac{v^{2}}{8N}\right)dv-\int_{0}^{w}\frac{v}{4N}S_{N}\left(\frac{u^{2}}{8N},\frac{v^{2}}{8N}\right)dv\right]
\nonumber\\
&&f'(w)\varepsilon\widetilde{\varphi}_{2N+1}^{(\alpha-1)}\left(\frac{w^{2}}{8N}\right)\widetilde{\varphi}_{2N}^{(\alpha-1)}\left(\frac{u^{2}}{8N}\right)u\:w\:f(u)dudw\nonumber\\
&=&\frac{1}{16N}\int_{0}^{\infty}\int_{0}^{\infty}
\left[\int_{w}^{\infty}B^{(\alpha-1)}(u,v)dv-\int_{0}^{w}B^{(\alpha-1)}(u,v)dv\right]\left[\int_{0}^{w}J_{\alpha-1}(t)dt\right]J_{\alpha-1}(u)w\:f'(w)f(u)dudw\nonumber\\
&+&O\left(N^{-3}\right),\;\;N\rightarrow\infty.\nonumber
\eea

The seventh term,
\bea
&&\mathrm{Tr}\:\frac{1}{2}\sqrt{\left(N+\frac{1}{2}\right)\left(N+\frac{\alpha}{2}\right)}
S_{N}\varepsilon f'\left(\varepsilon\widetilde{\varphi}_{2N+1}^{(\alpha-1)}\otimes\varepsilon\widetilde{\varphi}_{2N}^{(\alpha-1)}\right)f'\nonumber\\
&=&\frac{1}{64N^{2}}\sqrt{\left(N+\frac{1}{2}\right)\left(N+\frac{\alpha}{2}\right)}\int_{0}^{\infty}\int_{0}^{\infty}\left[\int_{w}^{\infty}
\frac{v}{4N}S_{N}\left(\frac{u^{2}}{8N},\frac{v^{2}}{8N}\right)dv-\int_{0}^{w}\frac{v}{4N}S_{N}\left(\frac{u^{2}}{8N},\frac{v^{2}}{8N}\right)dv\right]
\nonumber\\
&&f'(w)\varepsilon\widetilde{\varphi}_{2N+1}^{(\alpha-1)}\left(\frac{w^{2}}{8N}\right)\varepsilon\widetilde{\varphi}_{2N}^{(\alpha-1)}\left(\frac{u^{2}}{8N}\right)
f'(u)uwdudw\nonumber\\
&=&O\left(\frac{1}{N^{2}}\right),\;\;N\rightarrow\infty.\nonumber
\eea

The eighth term,
\bea
&&\mathrm{Tr}\:\left(N+\frac{1}{2}\right)\left(N+\frac{\alpha}{2}\right)
\left(\varepsilon\widetilde{\varphi}_{2N+1}^{(\alpha-1)}\otimes\widetilde{\varphi}_{2N}^{(\alpha-1)}f\right)
\left(\varepsilon\widetilde{\varphi}_{2N+1}^{(\alpha-1)}\otimes\widetilde{\varphi}_{2N}^{(\alpha-1)}f\right)\nonumber\\
&=&\frac{\left(N+\frac{1}{2}\right)\left(N+\frac{\alpha}{2}\right)}{16N^{2}}\int_{0}^{\infty}\int_{0}^{\infty}
\varepsilon\widetilde{\varphi}_{2N+1}^{(\alpha-1)}\left(\frac{u^{2}}{8N}\right)\varepsilon\widetilde{\varphi}_{2N+1}^{(\alpha-1)}\left(\frac{v^{2}}{8N}\right)
\widetilde{\varphi}_{2N}^{(\alpha-1)}\left(\frac{u^{2}}{8N}\right)\widetilde{\varphi}_{2N}^{(\alpha-1)}\left(\frac{v^{2}}{8N}\right)
uvf(u)f(v)dudv\nonumber\\
&=&\frac{1}{4}\int_{0}^{\infty}\int_{0}^{\infty}\left[\int_{0}^{u}J_{\alpha-1}(t)dt\right]\left[\int_{0}^{v}J_{\alpha-1}(t)dt\right]J_{\alpha-1}(u)J_{\alpha-1}(v)f(u)f(v)dudv
+O\left(N^{-2}\right),\;\;N\rightarrow\infty.\nonumber
\eea

The ninth term,
\bea
&&\mathrm{Tr}\:\left(N+\frac{1}{2}\right)\left(N+\frac{\alpha}{2}\right)
\left(\varepsilon\widetilde{\varphi}_{2N+1}^{(\alpha-1)}\otimes\widetilde{\varphi}_{2N}^{(\alpha-1)}f\right)
\left(\varepsilon\widetilde{\varphi}_{2N+1}^{(\alpha-1)}\otimes\varepsilon\widetilde{\varphi}_{2N}^{(\alpha-1)}\right)f'\nonumber\\
&=&\frac{\left(N+\frac{1}{2}\right)\left(N+\frac{\alpha}{2}\right)}{16N^{2}}\int_{0}^{\infty}\int_{0}^{\infty}
\varepsilon\widetilde{\varphi}_{2N+1}^{(\alpha-1)}\left(\frac{u^{2}}{8N}\right)\varepsilon\widetilde{\varphi}_{2N+1}^{(\alpha-1)}\left(\frac{v^{2}}{8N}\right)
\varepsilon\widetilde{\varphi}_{2N}^{(\alpha-1)}\left(\frac{u^{2}}{8N}\right)\widetilde{\varphi}_{2N}^{(\alpha-1)}\left(\frac{v^{2}}{8N}\right)\nonumber\\
&&uv\: f'(u)f(v)dudv\nonumber\\
&=&\frac{1}{16N}\int_{0}^{\infty}\int_{0}^{\infty}\left[\int_{0}^{u}J_{\alpha-1}(t)dt\right]\left[\int_{0}^{v}J_{\alpha-1}(t)dt\right]
\left[\int_{0}^{u}J_{\alpha-1}(t)dt-1\right]J_{\alpha-1}(v)u\:f'(u)f(v)dudv\nonumber\\
&+&O\left(N^{-3}\right),\;\;N\rightarrow\infty.\nonumber
\eea

The tenth term,
\bea
&&\mathrm{Tr}\:\frac{1}{4}\left(N+\frac{1}{2}\right)\left(N+\frac{\alpha}{2}\right)
\left(\varepsilon\widetilde{\varphi}_{2N+1}^{(\alpha-1)}\otimes\varepsilon\widetilde{\varphi}_{2N}^{(\alpha-1)}\right)f'
\left(\varepsilon\widetilde{\varphi}_{2N+1}^{(\alpha-1)}\otimes\varepsilon\widetilde{\varphi}_{2N}^{(\alpha-1)}\right)f'\nonumber\\
&=&\frac{1}{64N^{2}}\left(N+\frac{1}{2}\right)\left(N+\frac{\alpha}{2}\right)\int_{0}^{\infty}\int_{0}^{\infty}
\varepsilon\widetilde{\varphi}_{2N+1}^{(\alpha-1)}\left(\frac{u^{2}}{8N}\right)\varepsilon\widetilde{\varphi}_{2N+1}^{(\alpha-1)}\left(\frac{v^{2}}{8N}\right)
\varepsilon\widetilde{\varphi}_{2N}^{(\alpha-1)}\left(\frac{u^{2}}{8N}\right)\varepsilon\widetilde{\varphi}_{2N}^{(\alpha-1)}\left(\frac{v^{2}}{8N}\right)\nonumber\\
&&uv\:f'(u)f'(v)dudv\nonumber\\
&=&O\left(\frac{1}{N^{2}}\right),\;\;N\rightarrow\infty.\nonumber
\eea

Therefore,
\bea
&&\mathrm{Tr}\:T^{2}\nonumber\\
&=&\int_{0}^{\infty}\int_{0}^{\infty}B^{(\alpha-1)}(x,y)B^{(\alpha-1)}(y,x)f(x)f(y)dxdy\nonumber\\
&-&\int_{0}^{\infty}\int_{0}^{\infty}B^{(\alpha-1)}(x,y)J_{\alpha-1}(x)\left[\int_{0}^{y}J_{\alpha-1}(z)dz\right]f(x)f(y)dxdy\nonumber\\
&+&\frac{1}{4}\int_{0}^{\infty}\int_{0}^{\infty}\left[\int_{0}^{x}J_{\alpha-1}(z)dz\right]\left[\int_{0}^{y}J_{\alpha-1}(z)dz\right]J_{\alpha-1}(x)J_{\alpha-1}(y)f(x)f(y)dxdy\nonumber\\
&-&\frac{1}{8N}\int_{0}^{\infty}\int_{0}^{\infty}B^{(\alpha-1)}(x,y)\left[\int_{x}^{\infty}B^{(\alpha-1)}(y,z)dz-\int_{0}^{x}B^{(\alpha-1)}(y,z)dz\right]x\:f'(x)f(y)dx dy\nonumber\\
&-&\frac{1}{8N}\int_{0}^{\infty}\int_{0}^{\infty}B^{(\alpha-1)}(x,y)\left[\int_{0}^{x}J_{\alpha-1}(z)dz-1\right]\left[\int_{0}^{y}J_{\alpha-1}(z)dz\right]x\:f'(x)f(y)dxdy\nonumber\\
&+&\frac{1}{16N}\int_{0}^{\infty}\int_{0}^{\infty}
\left[\int_{x}^{\infty}B^{(\alpha-1)}(y,z)dz-\int_{0}^{x}B^{(\alpha-1)}(y,z)dz\right]\left[\int_{0}^{x}J_{\alpha-1}(z)dz\right]J_{\alpha-1}(y)x\:f'(x)f(y)dxdy\nonumber\\
&+&\frac{1}{16N}\int_{0}^{\infty}\int_{0}^{\infty}\left[\int_{0}^{x}J_{\alpha-1}(z)dz\right]\left[\int_{0}^{y}J_{\alpha-1}(z)dz\right]
\left[\int_{0}^{x}J_{\alpha-1}(z)dz-1\right]J_{\alpha-1}(y)x\:f'(x)f(y)dxdy\nonumber\\
&+&O\left(N^{-2}\right),\;\;N\rightarrow\infty.\nonumber
\eea

Now we want to see the mean and variance of the linear statistics $\sum_{j=1}^{N}F\left(\sqrt{8Nx_{j}}\right)$, so we need to obtain the coefficients of $\lambda$ and $\lambda^{2}$,
firstly we know
$$
f\left(\sqrt{8Nx}\right)\approx-\lambda F\left(\sqrt{8Nx}\right)+\frac{\lambda^{2}}{2}F^{2}\left(\sqrt{8Nx}\right),
$$
then we replace $f$ with $-\lambda F+\frac{\lambda^{2}}{2}F^{2}$ in the expression of $\mathrm{Tr}\:T$ and $\mathrm{Tr}\:T^{2}$,
similar to previous\\
discussions, denote by $\mu_{N}^{(LSE,\:\alpha)}$ and $\mathcal{V}_{N}^{(LSE,\:\alpha)}$ the mean and variance of the linear statistics $\sum_{j=1}^{N}F\left(\sqrt{8Nx_{j}}\right)$, we have the following theorem.
\begin{theorem}
As $N\rightarrow\infty$,
\bea
\mu_{N}^{(LSE,\:\alpha)}
&=&\frac{1}{2}\:\mu_{N}^{(LUE,\:\alpha-1)}
-\frac{1}{4}\int_{0}^{\infty}\left[\int_{0}^{x}J_{\alpha-1}(y)dy\right]J_{\alpha-1}(x)F(x)dx\nonumber\\
&-&\frac{1}{32N}\int_{0}^{\infty}\left[\int_{x}^{\infty}B^{(\alpha-1)}(x,y)dy-\int_{0}^{x}B^{(\alpha-1)}(x,y)dy\right]x\:F'(x)dx\nonumber\\
&-&\frac{1}{32N}\int_{0}^{\infty}\left[\int_{0}^{x}J_{\alpha-1}(y)dy\right]\left[\int_{0}^{x}J_{\alpha-1}(y)dy-1\right]x\:F'(x)dx+O\left(N^{-2}\right),\nonumber
\eea
\bea
\mathcal{V}_{N}^{(LSE,\:\alpha)}
&=&\frac{1}{2}\:\mathcal{V}_{N}^{(LUE,\:\alpha-1)}-\frac{1}{4}\int_{0}^{\infty}\left[\int_{0}^{x}J_{\alpha-1}(y)dy\right]J_{\alpha-1}(x)F^{2}(x)dx\nonumber\\
&+&\frac{1}{2}\int_{0}^{\infty}\int_{0}^{\infty}B^{(\alpha-1)}(x,y)J_{\alpha-1}(x)\left[\int_{0}^{y}J_{\alpha-1}(z)dz\right]F(x)F(y)dxdy\nonumber\\
&-&\frac{1}{8}\int_{0}^{\infty}\int_{0}^{\infty}\left[\int_{0}^{x}J_{\alpha-1}(z)dz\right]\left[\int_{0}^{y}J_{\alpha-1}(z)dz\right]J_{\alpha-1}(x)J_{\alpha-1}(y)F(x)F(y)dxdy\nonumber\\
&-&\frac{1}{16N}\int_{0}^{\infty}\left[\int_{x}^{\infty}B^{(\alpha-1)}(x,y)dy-\int_{0}^{x}B^{(\alpha-1)}(x,y)dy\right]x\:F(x)F'(x)dx\nonumber\\
&-&\frac{1}{16N}\int_{0}^{\infty}\left[\int_{0}^{x}J_{\alpha-1}(y)dy\right]\left[\int_{0}^{x}J_{\alpha-1}(y)dy-1\right]x\:F(x)F'(x)dx\nonumber\\
&+&\frac{1}{16N}\int_{0}^{\infty}\int_{0}^{\infty}B^{(\alpha-1)}(x,y)\left[\int_{x}^{\infty}B^{(\alpha-1)}(y,z)dz-\int_{0}^{x}B^{(\alpha-1)}(y,z)dz\right]x\:F'(x)F(y)dx dy\nonumber\\
&+&\frac{1}{16N}\int_{0}^{\infty}\int_{0}^{\infty}B^{(\alpha-1)}(x,y)\left[\int_{0}^{x}J_{\alpha-1}(z)dz-1\right]\left[\int_{0}^{y}J_{\alpha-1}(z)dz\right]x\:F'(x)F(y)dxdy\nonumber\\
&-&\frac{1}{32N}\int_{0}^{\infty}\int_{0}^{\infty}
\left[\int_{x}^{\infty}B^{(\alpha-1)}(y,z)dz-\int_{0}^{x}B^{(\alpha-1)}(y,z)dz\right]\left[\int_{0}^{x}J_{\alpha-1}(z)dz\right]J_{\alpha-1}(y)\nonumber\\
&&x\:F'(x)F(y)dxdy\nonumber\\
&-&\frac{1}{32N}\int_{0}^{\infty}\int_{0}^{\infty}\left[\int_{0}^{x}J_{\alpha-1}(z)dz\right]\left[\int_{0}^{y}J_{\alpha-1}(z)dz\right]
\left[\int_{0}^{x}J_{\alpha-1}(z)dz-1\right]J_{\alpha-1}(y)x\:F'(x)F(y)dxdy\nonumber\\
&+&O\left(N^{-2}\right),\nonumber
\eea
where $\mu_{N}^{(LUE,\:\alpha-1)}$ and $\mathcal{V}_{N}^{(LUE,\:\alpha-1)}$ for $N\rightarrow\infty$ are given in (\ref{luem}) and (\ref{luev}) respectively with $\alpha$ replaced by $\alpha-1$.
\end{theorem}

\section{The orthogonal ensembles}
\subsection{General case}
For the orthogonal ensembles, $\beta=1$, the equation (\ref{gnb}), becomes,
$$
G_{N}^{(1)}(f)=C_{N}^{(1)}\int_{[a,b]^n}\prod_{1\leq j<k\leq N}|x_{j}-x_{k}|
\prod_{j=1}^{N}w(x_{j})\left[1+f(x_{j})\right]dx_{j},
$$
and we assume $N$ is even. Here,
$$
C_{N}^{(1)}=\frac{1}{\int_{(a,b)^N}\prod_{1\leq j<k\leq N}\left|x_{j}-x_{k}\right|\prod_{j=1}^{N}w(x_{j})dx_{j}}
$$
depends on $N$.

We also follow the treatment of \cite{Dieng, Tracy1998}, firstly using Theorem \ref{pf} and some computations, we find
$$
\left[G_{N}^{(1)}(f)\right]^{2}=\widehat{C_{N}^{(1)}}\det\left(\int_{a}^{b}\int_{a}^{b}\varepsilon(x,y)\pi_{j}(x)\pi_{k}(y)w(x)w(y)
(1+f(x))(1+f(y))dxdy\right)_{j,k=0}^{N-1},
$$
where $\widehat{C_{N}^{(1)}}$ is a constant depending on $N$ and $\pi_{j}(x)$ is an arbitrary polynomial of degree $j$. Let
\be
\psi_{j}(x)=\pi_{j}(x)w(x),\label{psijx}
\ee
it follows that
$$
\left[G_{N}^{(1)}(f)\right]^{2}=\det\left(I+\left(M^{(1)}\right)^{-1}L^{(1)}\right),
$$
where
$$
M^{(1)}=\left(\int_{a}^{b}\psi_{j}(x)\varepsilon\psi_{k}(x)dx\right)_{j,k=0}^{N-1},
$$
$$
L^{(1)}=\left(\int_{a}^{b}\left(f(x)\psi_{j}(x)\varepsilon\psi_{k}(x)-f(x)\psi_{k}(x)\varepsilon\psi_{j}(x)
-f(x)\psi_{k}(x)\varepsilon(f\psi_{j})(x)\right)dx\right)_{j,k=0}^{N-1}.
$$
If $\left(M^{(1)}\right)^{-1}=(\mu_{jk})_{j,k=0}^{N-1}$, then
$$
\left[G_{N}^{(1)}(f)\right]^{2}=\det(I+K_{N}^{(1)}f),
$$
where $K_{N}^{(1)}$ is an integral operator
$$
K_{N}^{(1)}=
\begin{pmatrix}
-\sum_{j,k=0}^{N-1}\mu_{jk}\psi_{j}\otimes\varepsilon\psi_{k}&\sum_{j,k=0}^{N-1}\mu_{jk}\psi_{j}\otimes\psi_{k}\\
-\sum_{j,k=0}^{N-1}\mu_{jk}\varepsilon\psi_{j}\otimes\varepsilon\psi_{k}-\varepsilon&\sum_{j,k=0}^{N-1}\mu_{jk}\varepsilon\psi_{j}\otimes\psi_{k}\end{pmatrix}
=:\begin{pmatrix}
K_{N,1}^{(1,1)}&K_{N,1}^{(1,2)}\\
K_{N,1}^{(2,1)}&K_{N,1}^{(2,2)}
\end{pmatrix}.
$$
In the next theorem, we obtain relations on $K_{N,1}^{(i,j)}, i,j=1,2$.

\begin{theorem}
$$
K_{N,1}^{(1,2)}=DK_{N,1}^{(2,2)},\;\;K_{N,1}^{(2,1)}=K_{N,1}^{(2,2)}\varepsilon-\varepsilon,\;\;K_{N,1}^{(1,1)}=D K_{N,1}^{(2,2)}\varepsilon. \label{re1}
$$
\end{theorem}

\begin{proof}
For any integrable function $g(x)$ supported on $[a,b]$, we have
\bea
D K_{N,1}^{(2,2)}g(x)
&=&\frac{d}{dx}\int_{a}^{b}K_{N,1}^{(2,2)}(x,y)g(y)dy\nonumber\\
&=&\int_{a}^{b}\frac{\partial}{\partial x}K_{N,1}^{(2,2)}(x,y)g(y)dy\nonumber\\
&=&\int_{a}^{b}\frac{\partial}{\partial x}\sum_{j,k=0}^{N-1}\varepsilon\psi_{j}(x)\mu_{jk}\psi_{k}(y)g(y)dy\nonumber\\
&=&\int_{a}^{b}\sum_{j,k=0}^{N-1}(D\varepsilon\psi_{j}(x))\mu_{jk}\psi_{k}(y)g(y)dy\nonumber\\
&=&\int_{a}^{b}\sum_{j,k=0}^{N-1}\psi_{j}(x)\mu_{jk}\psi_{k}(y)g(y)dy\nonumber\\
&=&\int_{a}^{b}K_{N,1}^{(1,2)}(x,y)g(y)dy\nonumber\\
&=&K_{N,1}^{(1,2)}g(x),\nonumber
\eea
which implies,
$$
K_{N,1}^{(1,2)}=D K_{N,1}^{(2,2)}.
$$
Note that
\bea
K_{N,1}^{(2,2)}\varepsilon
&=&\sum_{j,k=0}^{N-1}\mu_{jk}\varepsilon\psi_{j}\otimes\psi_{k}\varepsilon\nonumber\\
&=&-\sum_{j,k=0}^{N-1}\mu_{jk}\varepsilon\psi_{j}\otimes\varepsilon\psi_{k} \label{ot}\\
&=&K_{N,1}^{(2,1)}+\varepsilon,\nonumber
\eea
that is,
$$
K_{N,1}^{(2,1)}=K_{N,1}^{(2,2)}\varepsilon-\varepsilon.
$$
Moreover, from (\ref{ot}),
\bea
D K_{N,1}^{(2,2)}\varepsilon
&=&-\sum_{j,k=0}^{N-1}\mu_{jk}(D\varepsilon\psi_{j})\otimes\varepsilon\psi_{k}\nonumber\\
&=&-\sum_{j,k=0}^{N-1}\mu_{jk}\psi_{j}\otimes\varepsilon\psi_{k}\nonumber\\
&=&K_{N,1}^{(1,1)}.\nonumber
\eea
The proof is complete.
\end{proof}

The series of computations presented below takes the determinant into the form of identity plus scalar operators. From Theorem \ref{re1}, $K_{N}^{(1)}$ can be written as
\bea
K_{N}^{(1)}
&=&\begin{pmatrix}
D K_{N,1}^{(2,2)}\varepsilon&D K_{N,1}^{(2,2)}\\
K_{N,1}^{(2,2)}\varepsilon-\varepsilon&K_{N,1}^{(2,2)}
\end{pmatrix}\nonumber\\
&=&
\begin{pmatrix}
D&0\\
0&I
\end{pmatrix}
\begin{pmatrix}
K_{N,1}^{(2,2)}\varepsilon&K_{N,1}^{(2,2)}\\
K_{N,1}^{(2,2)}\varepsilon-\varepsilon&K_{N,1}^{(2,2)}
\end{pmatrix}\nonumber\\
&=&:\widetilde{A}\widetilde{B},\nonumber
\eea
where
$$
\widetilde{A}:=\begin{pmatrix}
D&0\\
0&I
\end{pmatrix},\;\;\;\;
\widetilde{B}:=\begin{pmatrix}
K_{N,1}^{(2,2)}\varepsilon&K_{N,1}^{(2,2)}\\
K_{N,1}^{(2,2)}\varepsilon-\varepsilon&K_{N,1}^{(2,2)}
\end{pmatrix}.
$$
By Theorem \ref{hs}, we have
$$
\left[G_{N}^{(1)}(f)\right]^{2}=\det\Big(I+\widetilde{A}\widetilde{B}f\Big)=\det\Big(I+\widetilde{B}f\widetilde{A}\Big).
$$
Since
\bea
\widetilde{B}f\widetilde{A}
&=&
\begin{pmatrix}
K_{N,1}^{(2,2)}\varepsilon&K_{N,1}^{(2,2)}\\
K_{N,1}^{(2,2)}\varepsilon-\varepsilon&K_{N,1}^{(2,2)}
\end{pmatrix}
f
\begin{pmatrix}
D&0\\
0&I
\end{pmatrix}
\nonumber\\
&=&
\begin{pmatrix}
K_{N,1}^{(2,2)}\varepsilon&K_{N,1}^{(2,2)}\\
K_{N,1}^{(2,2)}\varepsilon-\varepsilon&K_{N,1}^{(2,2)}
\end{pmatrix}
\begin{pmatrix}
f D&0\\
0&f
\end{pmatrix}
\nonumber\\
&=&\begin{pmatrix}
K_{N,1}^{(2,2)}\varepsilon f D&K_{N,1}^{(2,2)}f\\
K_{N,1}^{(2,2)}\varepsilon f D-\varepsilon f D&K_{N,1}^{(2,2)}f
\end{pmatrix},\nonumber
\eea
then
$$
\left[G_{N}^{(1)}(f)\right]^{2}=\det\begin{pmatrix}
I+K_{N,1}^{(2,2)}\varepsilon f D&K_{N,1}^{(2,2)}f\\
K_{N,1}^{(2,2)}\varepsilon f D-\varepsilon f D&I+K_{N,1}^{(2,2)}f
\end{pmatrix}.
$$

The computation below reduces the above into a determinant of scalar operators. We subtract row 1 from row 2,
$$
\left[G_{N}^{(1)}(f)\right]^{2}=\det
\begin{pmatrix}
I+K_{N,1}^{(2,2)}\varepsilon f D&K_{N,1}^{(2,2)}f\\
-I-\varepsilon f D&I
\end{pmatrix}.
$$
Next, add column 2 times $I+\varepsilon f D$ to column 1,
\bea
\left[G_{N}^{(1)}(f)\right]^{2}
&=&\det
\begin{pmatrix}
I+K_{N,1}^{(2,2)}\varepsilon f D+K_{N,1}^{(2,2)}f(I+\varepsilon f D)&K_{N,1}^{(2,2)}f\\
0&I
\end{pmatrix}\nonumber\\
&=&\det\left(I+K_{N,1}^{(2,2)}(\varepsilon f D+f+f\varepsilon f D)\right).\nonumber
\eea
$\mathbf{Remark}.$ The above result agrees with \cite{Dieng} for the GOE case
if we take $f=-\mu\:\chi_{J}$, where $\chi_{J}$ is the characteristic function of the interval $J$.

Using (\ref{fd}), Lemma \ref{de} and note that $f$ is a Schwartz function, then we have the following theorem.
\begin{theorem}
$$
\left[G_{N}^{(1)}(f)\right]^{2}
=\det\left(I+K_{N,1}^{(2,2)}\left(f^{2}+2f\right)-K_{N,1}^{(2,2)}\varepsilon f'-K_{N,1}^{(2,2)}f\varepsilon f'\right), \label{sca1}
$$
where the kernel of $K_{N,1}^{(2,2)}$ reads
$$
K_{N,1}^{(2,2)}(x,y)=\sum_{j,k=0}^{N-1}\mu_{jk}\varepsilon\psi_{j}(x)\psi_{k}(y).
$$
\end{theorem}

\subsection{GOE}
It is convenient in this case to choose
$$
w(x)=\mathrm{e}^{-\frac{x^{2}}{2}},\;\;x\in \mathbb{R},
$$
following Dieng and Tracy-Widom's discussion \cite{Dieng, Tracy1998}. We want to choose $\psi_{j}$ to make $M^{(1)}$ simplest possible. Define
\be
\psi_{2n+1}(x):=\frac{d}{dx}\varphi_{2n}(x),\;\;\psi_{2n}(x):=\varphi_{2n}(x),\;\;n=0,1,2,\ldots,\label{psi3}
\ee
where $\varphi_{j}(x)$ is given by (\ref{gue}),
$$
\varphi_{j}(x)=\frac{H_{j}(x)}{c_{j}}\mathrm{e}^{-\frac{x^{2}}{2}},\;\;c_{j}=\pi^{\frac{1}{4}}2^{\frac{j}{2}}\sqrt{\Gamma(j+1)}.
$$
%$$
%\varphi_{n}(x)=\frac{1}{c_{n}}H_{n}(x)\mathrm{e}^{-\frac{x^{2}}{2}},\;\; c_{n}=2^{n/2}\sqrt{n!}\pi^{1/4}.
%$$

We can check that this definition satisfies (\ref{psijx}), i.e., $\psi_{j}(x)=\pi_{j}(x)\mathrm{e}^{-\frac{x^{2}}{2}}, j=0,1,2,\ldots$, where
$\pi_{j}(x)$ is a polynomial of degree $j$.
The matrix $M^{(1)}:=\left(\int_{-\infty}^{\infty}\psi_{j}(x)\varepsilon\psi_{k}(x)dx\right)_{j,k=0}^{N-1}$ is computed below
from (\ref{psi3}).
\begin{lemma}
$\mathrm{\mathbf{(Dieng, Tracy-Widom)}}$
$$
M^{(1)}=\begin{pmatrix}
0&1&0&0&\cdots&0&0\\
-1&0&0&0&\cdots&0&0\\
0&0&0&1&\cdots&0&0\\
0&0&-1&0&\cdots&0&0\\
\vdots&\vdots&\vdots&\vdots&&\vdots&\vdots\\
0&0&0&0&\cdots&0&1\\
0&0&0&0&\cdots&-1&0
\end{pmatrix}_{N\times N}.
$$
\end{lemma}

It's obvious that $\left(M^{(1)}\right)^{-1}=-M^{(1)}$, so $\mu_{2j,2j+1}=-1, \mu_{2j+1,2j}=1$, and
$\mu_{jk}=0$ for other cases, hence
\bea
K_{N,1}^{(2,2)}(x,y)
&=&\sum_{j,k=0}^{N-1}\mu_{jk}\varepsilon\psi_{j}(x)\psi_{k}(y)\nonumber\\
&=&-\sum_{j=0}^{\frac{N}{2}-1}\varepsilon\psi_{2j}(x)\psi_{2j+1}(y)+\sum_{j=0}^{\frac{N}{2}-1}\varepsilon\psi_{2j+1}(x)\psi_{2j}(y)\nonumber\\
&=&\sum_{j=0}^{\frac{N}{2}-1}\varphi_{2j}(x)\varphi_{2j}(y)-\sum_{j=0}^{\frac{N}{2}-1}\varepsilon\varphi_{2j}(x)\varphi_{2j}'(y).\nonumber
\eea
By (\ref{phid}), we have
\bea
\sum_{j=0}^{\frac{N}{2}-1}\varepsilon\varphi_{2j}(x)\varphi_{2j}'(y)
&=&\sum_{j=0}^{\frac{N}{2}-1}\varepsilon\varphi_{2j}(x)\left(\sqrt{j}\varphi_{2j-1}(y)-\sqrt{j+\frac{1}{2}}\varphi_{2j+1}(y)\right)\nonumber\\
&=&\sum_{j=0}^{\frac{N}{2}-1}\sqrt{j}\varepsilon\varphi_{2j}(x)\varphi_{2j-1}(y)-\sum_{j=0}^{\frac{N}{2}-1}
\sqrt{j+\frac{1}{2}}\varepsilon\varphi_{2j}(x)\varphi_{2j+1}(y)\nonumber\\
&=&\sum_{j=1}^{\frac{N}{2}-1}\sqrt{j}\varepsilon\varphi_{2j}(x)\varphi_{2j-1}(y)-\sum_{j=1}^{\frac{N}{2}}
\sqrt{j-\frac{1}{2}}\varepsilon\varphi_{2j-2}(x)\varphi_{2j-1}(y)\nonumber\\
&=&\sum_{j=1}^{\frac{N}{2}}\sqrt{j}\varepsilon\varphi_{2j}(x)\varphi_{2j-1}(y)-\sqrt{\frac{N}{2}}\varepsilon\varphi_{N}(x)\varphi_{N-1}(y)
-\sum_{j=1}^{\frac{N}{2}}\sqrt{j-\frac{1}{2}}\varepsilon\varphi_{2j-2}(x)\varphi_{2j-1}(y)\nonumber\\
&=&-\sum_{j=1}^{\frac{N}{2}}\left(\sqrt{j-\frac{1}{2}}\varepsilon\varphi_{2j-2}(x)-\sqrt{j}\varepsilon\varphi_{2j}(x)\right)\varphi_{2j-1}(y)
-\sqrt{\frac{N}{2}}\varepsilon\varphi_{N}(x)\varphi_{N-1}(y).\nonumber
\eea
Using (\ref{phid}) again, we find
$$
\varphi_{2j-1}'(x)=\sqrt{j-\frac{1}{2}}\varphi_{2j-2}(x)-\sqrt{j}\varphi_{2j}(x),
$$
then according to Lemma \ref{de},
$$
\varphi_{2j-1}(x)=\varepsilon\varphi_{2j-1}'(x)=\sqrt{j-\frac{1}{2}}\varepsilon\varphi_{2j-2}(x)-\sqrt{j}\varepsilon\varphi_{2j}(x).
$$
Hence
$$
\sum_{j=0}^{\frac{N}{2}-1}\varepsilon\varphi_{2j}(x)\varphi_{2j}'(y)
=-\sum_{j=1}^{\frac{N}{2}}\varphi_{2j-1}(x)\varphi_{2j-1}(y)-\sqrt{\frac{N}{2}}\varepsilon\varphi_{N}(x)\varphi_{N-1}(y).
$$
It follows that
\bea
K_{N,1}^{(2,2)}(x,y)
&=&\sum_{j=0}^{\frac{N}{2}-1}\varphi_{2j}(x)\varphi_{2j}(y)+\sum_{j=1}^{\frac{N}{2}}\varphi_{2j-1}(x)\varphi_{2j-1}(y)
+\sqrt{\frac{N}{2}}\varepsilon\varphi_{N}(x)\varphi_{N-1}(y)\nonumber\\
&=&\sum_{j=0}^{N-1}\varphi_{j}(x)\varphi_{j}(y)+\sqrt{\frac{N}{2}}\varepsilon\varphi_{N}(x)\varphi_{N-1}(y)\nonumber\\
&=&S_{N}(x,y)+\sqrt{\frac{N}{2}}\varepsilon\varphi_{N}(x)\varphi_{N-1}(y),\nonumber
\eea
where
$$
S_{N}(x,y):=\sum_{j=0}^{N-1}\varphi_{j}(x)\varphi_{j}(y)=\sqrt{\frac{N}{2}}\:\frac{\varphi_{N}(x)\varphi_{N-1}(y)-\varphi_{N}(y)\varphi_{N-1}(x)}{x-y}.
$$
Here we have used the Christoffel-Darboux formula in the last equality. Note that\\ $S_{N}(x,y)=K_{N}^{(2)}(x,y)$.

By Theorem \ref{sca1}, we have the following theorem.
\begin{theorem}
\bea
\left[G_{N}^{(1)}(f)\right]^{2}
&=&\det\bigg(I+S_{N}(f^{2}+2f)-S_{N}\varepsilon f'-S_{N}f\varepsilon f'+\sqrt{\frac{N}{2}}\varepsilon\varphi_{N}\otimes\varphi_{N-1}(f^{2}+2f)\nonumber\\
&+&\sqrt{\frac{N}{2}}(\varepsilon\varphi_{N}\otimes\varepsilon\varphi_{N-1})f'+\sqrt{\frac{N}{2}}\left(\varepsilon\varphi_{N}\otimes\varepsilon(\varphi_{N-1}f)\right) f'\bigg).\nonumber
\eea
\end{theorem}

\subsection{Large $N$ behavior of the GOE moment generating function}
Now we consider the scaling limit of $\left[G_{N}^{(1)}(f)\right]^{2}$. Write
$$
\left[G_{N}^{(1)}(f)\right]^{2}=:\det(I+T),
$$
where
\bea
T:
&=&S_{N}(f^{2}+2f)-S_{N}\varepsilon f'-S_{N}f\varepsilon f'+\sqrt{\frac{N}{2}}\varepsilon\varphi_{N}\otimes\varphi_{N-1}(f^{2}+2f)\nonumber\\
&+&\sqrt{\frac{N}{2}}(\varepsilon\varphi_{N}\otimes\varepsilon\varphi_{N-1})f'+\sqrt{\frac{N}{2}}\left(\varepsilon\varphi_{N}\otimes\varepsilon(\varphi_{N-1}f)\right)f'.\nonumber
\eea
In the computations below, we replace $f(x)$ by $f\left(\sqrt{2N}x\right)$ and $f'(x)$ by
$$
f'\left(\sqrt{2N}x\right)=\frac{d}{\sqrt{2N}dx}f\left(\sqrt{2N}x\right).
$$
Similarly to the GSE case, we have the following theorems.
\begin{theorem}
$$
\lim_{N\rightarrow\infty}\frac{1}{\sqrt{2N}}S_{N}\left(\frac{x}{\sqrt{2N}},\frac{y}{\sqrt{2N}}\right)=\frac{\sin(x-y)}{\pi(x-y)},
$$

$$
\lim_{N\rightarrow\infty}\frac{1}{\sqrt{2N}}S_{N}\left(\frac{x}{\sqrt{2N}},\frac{x}{\sqrt{2N}}\right)=\frac{1}{\pi}.\label{sn1}
$$
\end{theorem}

\begin{theorem}
$$
\varphi_{N-1}\left(\frac{x}{\sqrt{2N}}\right)=-(-1)^{\frac{N}{2}}2^{\frac{1}{4}}\pi^{-\frac{1}{2}}N^{-\frac{1}{4}}\sin x+O\left(N^{-\frac{3}{4}}\right),\;\;N\rightarrow\infty,
$$

$$
\varepsilon\varphi_{N}\left(\frac{x}{\sqrt{2N}}\right)=(-1)^{\frac{N}{2}}2^{-\frac{1}{4}}\pi^{-\frac{1}{2}}N^{-\frac{3}{4}}\sin x+O\left(N^{-\frac{5}{4}}\right),
\;\;N\rightarrow\infty,
$$

$$
\varepsilon\varphi_{N-1}\left(\frac{x}{\sqrt{2N}}\right)=-2^{-\frac{1}{4}}N^{-\frac{1}{4}}+O\left(N^{-\frac{3}{4}}\right),\;\;N\rightarrow\infty,
$$

$$
\varepsilon\left(\varphi_{N-1}f\right)\left(\frac{x}{\sqrt{2N}}\right)=-(-1)^{\frac{N}{2}}2^{-\frac{5}{4}}\pi^{-\frac{1}{2}}N^{-\frac{3}{4}}
\left[\int_{-\infty}^{x}\sin y\: f(y)dy-\int_{x}^{\infty}\sin y\: f(y)dy\right]+O\left(N^{-\frac{5}{4}}\right),\;\;N\rightarrow\infty.\label{goe}
$$
\end{theorem}

Using Theorem \ref{sn1} and Theorem \ref{goe} to compute $\mathrm{Tr}\:T$ and $\mathrm{Tr}\:T^{2}$ as $N\rightarrow\infty$.  we have
\bea
\mathrm{Tr}\:T
&=&\frac{1}{\pi}\int_{-\infty}^{\infty}\left[f^{2}(x)+2f(x)\right]dx\nonumber\\
&-&\frac{1}{2\pi\sqrt{2N}}\int_{-\infty}^{\infty}\left[\int_{x}^{\infty}\frac{\sin(x-y)}{x-y}f(y)dy-\int_{-\infty}^{x}\frac{\sin(x-y)}{x-y}f(y)dy\right]f'(x)dx\nonumber\\
&-&\frac{1}{2\pi N}\int_{-\infty}^{\infty}\sin^{2}x\left[f^{2}(x)+2f(x)\right]dx
-\frac{(-1)^{\frac{N}{2}}}{2\sqrt{2\pi}N}\int_{-\infty}^{\infty}\sin x\:f'(x)dx+O\left(N^{-\frac{3}{2}}\right),\;\;N\rightarrow\infty,\nonumber
\eea
and
\bea
\mathrm{Tr}\:T^{2}
&=&\frac{1}{\pi^{2}}\int_{-\infty}^{\infty}\int_{-\infty}^{\infty}\left[\frac{\sin(x-y)}{x-y}\right]^{2}\left[f^{2}(x)+2f(x)\right]\left[f^{2}(y)+2f(y)\right]dxdy\nonumber\\
&+&\frac{2}{\pi^{2}\sqrt{2N}}\int_{-\infty}^{\infty}\int_{-\infty}^{\infty}\frac{\sin(x-y)}{x-y}\:\mathrm{Si}(x-y)f'(x)\left[f^{2}(y)+2f(y)\right]dxdy\nonumber\\
&-&\frac{1}{\pi^{2}\sqrt{2N}}\int_{-\infty}^{\infty}\int_{-\infty}^{\infty}
\left[\int_{x}^{\infty}\frac{\sin(y-z)}{y-z}f(z)dz-\int_{-\infty}^{x}\frac{\sin(y-z)}{y-z}f(z)dz\right]\frac{\sin(x-y)}{x-y}\nonumber\\
&&f'(x)\left[f^{2}(y)+2f(y)\right]dxdy+O\left(N^{-1}\right),\;\;N\rightarrow\infty.\nonumber
\eea

Now we want to find the mean and variance of the linear statistics $\sum_{j=1}^{N}F\left(\sqrt{2N}x_{j}\right)$, since
$$
f\left(\sqrt{2N}x\right)\approx-\lambda F\left(\sqrt{2N}x\right)+\frac{\lambda^{2}}{2}F^{2}\left(\sqrt{2N}x\right),
$$
we replace $f$ with $-\lambda F+\frac{\lambda^{2}}{2}F^{2}$ in the expression of $\mathrm{Tr}\:T$ and $\mathrm{Tr}\:T^{2}$, we have
\bea
&&\log\det\left(I+T\right)\nonumber\\
&=&-\lambda\Bigg\{\frac{2}{\pi}\int_{-\infty}^{\infty}F(x)dx
-\frac{1}{\pi N}\int_{-\infty}^{\infty}\sin^{2}x\:F(x)dx
-\frac{(-1)^{\frac{N}{2}}}{2\sqrt{2\pi}N}\int_{-\infty}^{\infty}\sin x\:F'(x)dx+O\left(N^{-\frac{3}{2}}\right)\Bigg\}\nonumber\\
&+&\frac{\lambda^{2}}{2}\Bigg\{\frac{4}{\pi}\int_{-\infty}^{\infty}F^{2}(x)dx
-\frac{4}{\pi^{2}}\int_{-\infty}^{\infty}\int_{-\infty}^{\infty}\left[\frac{\sin(x-y)}{x-y}\right]^{2}F(x)F(y)dxdy\nonumber\\
&-&\frac{1}{\pi\sqrt{2N}}\int_{-\infty}^{\infty}\left[\int_{x}^{\infty}\frac{\sin(x-y)}{x-y}F(y)dy-\int_{-\infty}^{x}\frac{\sin(x-y)}{x-y}F(y)dy\right]F'(x)dx\nonumber\\
&-&\frac{4}{\pi^{2}\sqrt{2N}}\int_{-\infty}^{\infty}\int_{-\infty}^{\infty}\frac{\sin(x-y)}{x-y}\:\mathrm{Si}(x-y)F'(x)F(y)dxdy+O\left(N^{-1}\right)\Bigg\},
\;\;N\rightarrow\infty.\nonumber
\eea
Denote by $\mu_{N}^{(GOE)}$ and $\mathcal{V}_{N}^{(GOE)}$ the mean and variance of the linear statistics $\sum_{j=1}^{N}F\left(\sqrt{2N}x_{j}\right)$, then we have the following theorem.
\begin{theorem}
As $N\rightarrow\infty$,
$$
\mu_{N}^{(GOE)}=\mu_{N}^{(GUE)}
-\frac{1}{2\pi N}\int_{-\infty}^{\infty}\sin^{2}x\:F(x)dx
-\frac{(-1)^{\frac{N}{2}}}{4\sqrt{2\pi}N}\int_{-\infty}^{\infty}\sin x\:F'(x)dx+O\left(N^{-\frac{3}{2}}\right),
$$
\bea
\mathcal{V}_{N}^{(GOE)}
&=&2\mathcal{V}_{N}^{(GUE)}
-\frac{1}{2\pi\sqrt{2N}}\int_{-\infty}^{\infty}\left[\int_{x}^{\infty}\frac{\sin(x-y)}{x-y}F(y)dy-\int_{-\infty}^{x}\frac{\sin(x-y)}{x-y}F(y)dy\right]F'(x)dx\nonumber\\
&-&\frac{2}{\pi^{2}\sqrt{2N}}\int_{-\infty}^{\infty}\int_{-\infty}^{\infty}\frac{\sin(x-y)}{x-y}\:\mathrm{Si}(x-y)F'(x)F(y)dxdy+O\left(N^{-1}\right),\nonumber
\eea
where $\mu_{N}^{(GUE)}$ and $\mathcal{V}_{N}^{(GUE)}$ for $N\rightarrow\infty$ are given in (\ref{guem}) and (\ref{guev}) respectively.
\end{theorem}

\subsection{LOE}
We now specialize results obtained in Section 4.1 to the situation  where $w$ is taken to be the {\it square root} of the Laguerre weight, namely,
$$
w(x)=x^{\frac{\alpha}{2}}\mathrm{e}^{-\frac{x}{2}},\;\;\alpha>-2,\;x\in \mathbb{R}^+.
$$
Again we choose a special  $\psi_{j}$ so that  $M^{(1)}$ is as simple as possible. Let
\be
\psi_{2n+1}(x):=\frac{d}{dx}\varphi_{2n}^{(\alpha+1)}(x),\;\;\psi_{2n}(x):=\widetilde{\varphi}_{2n}^{(\alpha+1)}(x),\;\;n=0,1,2,\ldots,\label{psila}
\ee
where $\varphi_{j}^{(\alpha+1)}(x)$ and $\widetilde{\varphi}_{j}^{(\alpha+1)}(x)$ are given by
\be
\varphi_{j}^{(\alpha+1)}(x):=\frac{L_{j}^{(\alpha+1)}(x)}{c_{j}^{(\alpha+1)}}x^{\frac{\alpha}{2}+1}\mathrm{e}^{-\frac{x}{2}},\;\;j=0,1,2,\ldots,\label{phijx}
\ee
$$
\widetilde{\varphi}_{j}^{(\alpha+1)}(x):=\frac{L_{j}^{(\alpha+1)}(x)}{c_{j}^{(\alpha+1)}}x^{\frac{\alpha}{2}}\mathrm{e}^{-\frac{x}{2}},\;\;j=0,1,2,\ldots.
$$
It's easy to see that
$$
\int_{0}^{\infty}\varphi_{j}^{(\alpha+1)}(x)\widetilde{\varphi}_{k}^{(\alpha+1)}(x)dx=\delta_{jk},\;\;j,k=0,1,2,\ldots.
$$
The next theorem shows that (\ref{psila}) satisfies (\ref{psijx}),
%$\psi_{j}(x)=\pi_{j}(x)x^{\frac{\alpha}{2}}\mathrm{e}^{-\frac{x}{2}}, j=0,1,\cdots$,
namely, $\psi_{j}(x)x^{-\frac{\alpha}{2}}\mathrm{e}^{\frac{x}{2}}$ is a polynomial of degree $j$.
\begin{theorem}
$\psi_{j}(x)x^{-\frac{\alpha}{2}}\mathrm{e}^{\frac{x}{2}},\: j=0,1,2,\ldots$ is a polynomial of degree $j$.
\end{theorem}

\begin{proof}
We prove this by considering  two cases, $j$ even and $j$  odd. It is clear for even $j$, that
$\psi_{2j}\:x^{-\alpha/2}\:{\rm e}^{x/2}$ is up to a constant multiple of $L_{2j}^{(\alpha+1)}(x).$
%then we have
%\bea
%\psi_{2n}(x)x^{-\frac{\alpha}{2}}\mathrm{e}^{\frac{x}{2}}
%&=&\widetilde{\varphi}_{2n}(x)x^{-\frac{\alpha}{2}}\mathrm{e}^{\frac{x}{2}}\nonumber\\
%&=&\frac{1}{c_{2n}^{(\alpha+1)}}L_{2n}^{(\alpha+1)}(x)x^{\frac{\alpha}{2}}\mathrm{e}^{-\frac{x}{2}}
%\cdot x^{-\frac{\alpha}{2}}\mathrm{e}^{\frac{x}{2}}\nonumber\\
%&=&\frac{1}{c_{2n}^{(\alpha+1)}}L_{2n}^{(\alpha+1)}(x),\;\;n=0,1,\cdots,\nonumber
%\eea
%obviously $\psi_{2n}(x)x^{-\frac{\alpha}{2}}\mathrm{e}^{\frac{x}{2}}$ is a polynomial of degree $2n$.

If $j$ is odd, then we find,
\bea
\psi_{2n+1}(x)x^{-\frac{\alpha}{2}}\mathrm{e}^{\frac{x}{2}}
&=&\left[\varphi_{2n}^{(\alpha+1)}(x)\right]'x^{-\frac{\alpha}{2}}\mathrm{e}^{\frac{x}{2}}\nonumber\\
&=&\left(\frac{1}{c_{2n}^{(\alpha+1)}}L_{2n}^{(\alpha+1)}(x)x^{\frac{\alpha}{2}+1}\mathrm{e}^{-\frac{x}{2}}\right)'x^{-\frac{\alpha}{2}}\mathrm{e}^{\frac{x}{2}}
\nonumber\\
&=&\frac{1}{c_{2n}^{(\alpha+1)}}\left[x \left(L_{2n}^{(\alpha+1)}(x)\right)'+\left(\frac{\alpha}{2}+1\right)L_{2n}^{(\alpha+1)}(x)-\frac{1}{2}x L_{2n}^{(\alpha+1)}(x)\right],\;\;n=0,1,2,\ldots.\nonumber
\eea
Thus $\psi_{2n+1}(x)x^{-\alpha/2}\:\mathrm{e}^{\frac{x^{2}}{2}}$ is a polynomial of degree $2n+1$.
The proof is complete.
\end{proof}
We use (\ref{psila}) to compute
$M^{(1)}:=\left(\int_{0}^{\infty}\psi_{j}(x)\varepsilon\psi_{k}(x)dx\right)_{j,k=0}^{N-1}$, resulting in the following theorem.
\begin{theorem}
$$
M^{(1)}=\begin{pmatrix}
0&1&0&0&\cdots&0&0\\
-1&0&0&0&\cdots&0&0\\
0&0&0&1&\cdots&0&0\\
0&0&-1&0&\cdots&0&0\\
\vdots&\vdots&\vdots&\vdots&&\vdots&\vdots\\
0&0&0&0&\cdots&0&1\\
0&0&0&0&\cdots&-1&0
\end{pmatrix}_{N\times N}.
$$
\end{theorem}

\begin{proof}
Let $m_{jk}$ be the $(j,k)$-entry of $M^{(1)}$,
%$$
%m_{jk}:=\int_{0}^{\infty}\psi_{j}(x)\varepsilon\psi_{k}(x)dx,\;\; j,k=0,1,\cdots,N-1.
%$$
again separating into  four cases: $(j,k)=({\rm even}, {\rm odd})$, $({\rm odd}, {\rm even})$,
  $(\rm {even}, {\rm even})$ and  $({\rm odd},{\rm odd})$. We find,
  for the $({\rm even},{\rm odd})$ case,
\bea
m_{2j,2k+1}
&=&\int_{0}^{\infty}\psi_{2j}(x)\varepsilon\psi_{2k+1}(x)dx\nonumber\\
&=&\int_{0}^{\infty}\widetilde{\varphi}_{2j}^{(\alpha+1)}(x)\left(\varepsilon D\varphi_{2k}^{(\alpha+1)}(x)\right)dx\nonumber\\
&=&\int_{0}^{\infty}\widetilde{\varphi}_{2j}^{(\alpha+1)}(x)\varphi_{2k}^{(\alpha+1)}(x)dx\nonumber\\
&=&\delta_{jk}.\nonumber
\eea
For the $({\rm odd},{\rm even})$ case, note that $M^{(1)}$ is antisymmetric, then
\bea
m_{2j+1,2k}
&=&-m_{2k,2j+1}\nonumber\\
&=&-\delta_{jk}.\nonumber
\eea
The $({\rm even},{\rm even})$ case,
\bea
m_{2j,2k}
&=&\int_{0}^{\infty}\psi_{2j}(x)\varepsilon\psi_{2k}(x)dx\nonumber\\
&=&\int_{0}^{\infty}\widetilde{\varphi}_{2j}^{(\alpha+1)}(x)\varepsilon\widetilde{\varphi}_{2k}^{(\alpha+1)}(x)dx\nonumber\\
&=&\frac{1}{2c_{2j}^{(\alpha+1)}c_{2k}^{(\alpha+1)}}\int_{0}^{\infty}L_{2j}^{(\alpha+1)}(x)x^{\frac{\alpha}{2}}\mathrm{e}^{-\frac{x}{2}}\left(\int_{0}^{x}
L_{2k}^{(\alpha+1)}(y)y^{\frac{\alpha}{2}}\mathrm{e}^{-\frac{y}{2}}dy-\int_{x}^{\infty}L_{2k}^{(\alpha+1)}(y)y^{\frac{\alpha}{2}}\mathrm{e}^{-\frac{y}{2}}dy\right)dx\nonumber\\
&=&\frac{1}{2c_{2j}^{(\alpha+1)}c_{2k}^{(\alpha+1)}}\int_{0}^{\infty}L_{2j}^{(\alpha+1)}(x)x^{\frac{\alpha}{2}}\mathrm{e}^{-\frac{x}{2}}\left(2\int_{0}^{x}
L_{2k}^{(\alpha+1)}(y)y^{\frac{\alpha}{2}}\mathrm{e}^{-\frac{y}{2}}dy-\int_{0}^{\infty}L_{2k}^{(\alpha+1)}(y)y^{\frac{\alpha}{2}}\mathrm{e}^{-\frac{y}{2}}dy\right)dx\nonumber\\
&=&\frac{1}{2c_{2j}^{(\alpha+1)}c_{2k}^{(\alpha+1)}}\bigg(2\int_{0}^{\infty}L_{2j}^{(\alpha+1)}(x)x^{\frac{\alpha}{2}}\mathrm{e}^{-\frac{x}{2}}dx\int_{0}^{x}
L_{2k}^{(\alpha+1)}(y)y^{\frac{\alpha}{2}}\mathrm{e}^{-\frac{y}{2}}dy\nonumber\\
&-&\int_{0}^{\infty}L_{2j}^{(\alpha+1)}(x)x^{\frac{\alpha}{2}}\mathrm{e}^{-\frac{x}{2}}dx\int_{0}^{\infty}L_{2k}^{(\alpha+1)}(y)y^{\frac{\alpha}{2}}
\mathrm{e}^{-\frac{y}{2}}dy\bigg)\nonumber\\
&=&\frac{1}{2c_{2j}^{(\alpha+1)}c_{2k}^{(\alpha+1)}}\left[\frac{2^{\alpha+2}\Gamma\left(j+1+\frac{\alpha}{2}\right)\Gamma\left(k+1+\frac{\alpha}{2}\right)}{j!\:k!}
-\frac{2^{\alpha+2}\Gamma\left(j+1+\frac{\alpha}{2}\right)\Gamma\left(k+1+\frac{\alpha}{2}\right)}{j!\:k!}\right]\nonumber\\
&=&0,\nonumber
\eea
where we have used the fact
$$
\int_{0}^{\infty}L_{2j}^{(\alpha+1)}(x)x^{\frac{\alpha}{2}}\mathrm{e}^{-\frac{x}{2}}dx=\frac{2^{\frac{\alpha}{2}+1}\Gamma\left(j+1+\frac{\alpha}{2}\right)}{j!},\;\;j=0,1,2,\ldots.
$$
Finally, the $({\rm odd},{\rm odd})$ case. If $j<k$,
\bea
m_{2j+1,2k+1}
&=&\int_{0}^{\infty}\psi_{2j+1}(x)\varepsilon\psi_{2k+1}(x)dx\nonumber\\
&=&\int_{0}^{\infty}\left[\varphi_{2j}^{(\alpha+1)}(x)\right]'\varepsilon D\varphi_{2k}^{(\alpha+1)}(x)dx\nonumber\\
&=&\int_{0}^{\infty}\left[\varphi_{2j}^{(\alpha+1)}(x)\right]'\varphi_{2k}^{(\alpha+1)}(x)dx\nonumber\\
&=&\int_{0}^{\infty}\frac{1}{c_{2j}^{(\alpha+1)}}\left[x \left(L_{2j}^{(\alpha+1)}(x)\right)'+\left(\frac{\alpha}{2}+1\right)L_{2j}^{(\alpha+1)}(x)-\frac{1}{2}x L_{2j}^{(\alpha+1)}(x)\right]x^{\frac{\alpha}{2}}\mathrm{e}^{-\frac{x}{2}}\nonumber\\
&&\frac{1}{c_{2k}^{(\alpha+1)}}L_{2k}^{(\alpha+1)}(x)x^{\frac{\alpha}{2}+1}\mathrm{e}^{-\frac{x}{2}}dx\nonumber\\
&=&\frac{1}{c_{2j}^{(\alpha+1)}c_{2k}^{(\alpha+1)}}\int_{0}^{\infty}L_{2k}^{(\alpha+1)}(x)\left[x \left(L_{2j}^{(\alpha+1)}(x)\right)'+\left(\frac{\alpha}{2}
+1\right)L_{2j}^{(\alpha+1)}(x)-\frac{1}{2}x L_{2j}^{(\alpha+1)}(x)\right]x^{\alpha+1}\mathrm{e}^{-x}dx\nonumber\\
&=&0,\nonumber
\eea
since, $2j+1$, the degree of the polynomials
$x \left(L_{2j}^{(\alpha+1)}(x)\right)'+\left(\frac{\alpha}{2}+1\right)L_{2j}^{(\alpha+1)}(x)-\frac{1}{2}x L_{2j}^{(\alpha+1)}
(x)$, is less than $2k$.

If $j>k$, due to the fact that $M^{(1)}$ is antisymmetric,
\bea
m_{2j+1,2k+1}
&=&-m_{2k+1,2j+1}\nonumber\\
&=&0,\nonumber
\eea
hence
$$
m_{2j+1,2k+1}=0,\;\;j,k=0,1,\ldots.
$$
It is the desired result for $M^{(1)}$. The proof is complete.
\end{proof}

We begin here a series of computations analogues to those in derivation of the LSE problem, ending up with
an expression for $\left[G_N^{(1)}(f)\right]^2$ as a scalar Fredholm determinant.

It's obvious that $\left(M^{(1)}\right)^{-1}=-M^{(1)}$, so $\mu_{2j,2j+1}=-1, \mu_{2j+1,2j}=1$, and $\mu_{jk}=0$ for other cases,
hence
\bea
K_{N,1}^{(2,2)}(x,y)
&=&\sum_{j,k=0}^{N-1}\mu_{jk}\varepsilon\psi_{j}(x)\psi_{k}(y)\nonumber\\
&=&-\sum_{j=0}^{\frac{N}{2}-1}\varepsilon\psi_{2j}(x)\psi_{2j+1}(y)+\sum_{j=0}^{\frac{N}{2}-1}\varepsilon\psi_{2j+1}(x)\psi_{2j}(y)\nonumber\\
&=&\sum_{j=0}^{\frac{N}{2}-1}\varphi_{2j}^{(\alpha+1)}(x)\widetilde{\varphi}_{2j}^{(\alpha+1)}(y)-\sum_{j=0}^{\frac{N}{2}-1}\varepsilon\widetilde{\varphi}_{2j}^{(\alpha+1)}(x)
\left[\varphi_{2j}^{(\alpha+1)}(y)\right]'.
\nonumber
\eea
Differentiating (\ref{phijx}) with respect to $x$, we find,
$$
\left[\varphi_{j}^{(\alpha+1)}(x)\right]'=\frac{1}{c_{j}^{(\alpha+1)}}
\left[x\left(L_{j}^{(\alpha+1)}(x)\right)'+\frac{\alpha+2-x}{2}L_{j}^{(\alpha+1)}(x)\right]x^{\frac{\alpha}{2}}
\mathrm{e}^{-\frac{x}{2}}.
$$
Recall that the Laguerre polynomials $L_{j}^{(\alpha+1)}$ satisfy the differentiation formulas \cite{Gradshteyn}
\be
x\left(L_{j}^{(\alpha+1)}(x)\right)'=j L_{j}^{(\alpha+1)}(x)-(j+\alpha+1)L_{j-1}^{(\alpha+1)}(x),\;\;j=0,1,2,\ldots,\label{la1}
\ee
\be
x\left(L_{j}^{(\alpha+1)}(x)\right)'=(j+1)L_{j+1}^{(\alpha+1)}(x)-(j+\alpha+2-x)L_{j}^{(\alpha+1)}(x),\;\;j=0,1,2,\ldots.\label{la2}
\ee
The sum of (\ref{la1}) and (\ref{la2}) divided by 2, gives,
$$
x\left(L_{j}^{(\alpha+1)}(x)\right)'=\frac{j+1}{2}L_{j+1}^{(\alpha+1)}(x)-\frac{j+\alpha+1}{2}
L_{j-1}^{(\alpha+1)}(x)-\frac{\alpha+2-x}{2}L_{j}^{(\alpha+1)}(x),
\;\;j=0,1,2,\ldots,
$$
which is the same as,
$$
x\left(L_j^{(\alpha+1)}(x)\right)'+\frac{\alpha+2-x}{2}L_j^{(\alpha+1)}(x)=\frac{j+1}{2}
\:L_{j+1}^{(\alpha+1)}(x)-\frac{j+\alpha+1}{2}\:L_{j-1}^{(\alpha+1)}(x).
$$
Hence the derivative of $\varphi_{j}^{(\alpha+1)}(x)$ becomes,
\bea
\left[\varphi_{j}^{(\alpha+1)}(x)\right]'
&=&\frac{1}{c_{j}^{(\alpha+1)}}\left[\frac{j+1}{2}L_{j+1}^{(\alpha+1)}(x)-\frac{j+\alpha+1}{2}L_{j-1}^{(\alpha+1)}(x)
\right]x^{\frac{\alpha}{2}}
\mathrm{e}^{-\frac{x}{2}}\nonumber\\
&=&\frac{1}{2}\sqrt{(j+1)(j+\alpha+2)}\:\widetilde{\varphi}_{j+1}^{(\alpha+1)}(x)-\frac{1}{2}\sqrt{j(j+\alpha+1)}\:\widetilde{\varphi}_{j-1}^{(\alpha+1)}(x).
\label{phijp}
\eea
So replacing $j$ by $2j$, we see that,
$$
\left[\varphi_{2j}^{(\alpha+1)}(y)\right]'=\sqrt{\left(j+\frac{1}{2}\right)\left(j+\frac{\alpha+2}{2}\right)}\:\widetilde{\varphi}_{2j+1}^{(\alpha+1)}(y)
-\sqrt{j\left(j+\frac{\alpha+1}{2}\right)}\:\widetilde{\varphi}_{2j-1}^{(\alpha+1)}(y).
$$
Continuing, we compute the summation to give,
\bea
&&\sum_{j=0}^{\frac{N}{2}-1}\varepsilon\widetilde{\varphi}_{2j}^{(\alpha+1)}(x)\left[\varphi_{2j}^{(\alpha+1)}(y)\right]'\nonumber\\
&=&\sum_{j=0}^{\frac{N}{2}-1}\sqrt{\left(j+\frac{1}{2}\right)\left(j+\frac{\alpha+2}{2}\right)}\:\varepsilon\widetilde{\varphi}_{2j}^{(\alpha+1)}(x)
\widetilde{\varphi}_{2j+1}^{(\alpha+1)}(y)\nonumber\\
&-&\sum_{j=1}^{\frac{N}{2}-1}\sqrt{j\left(j+\frac{\alpha+1}{2}\right)}\:\varepsilon\widetilde{\varphi}_{2j}^{(\alpha+1)}(x)\widetilde{\varphi}_{2j-1}^{(\alpha+1)}(y)\nonumber\\
&=&\sum_{j=1}^{\frac{N}{2}}\left[\sqrt{\left(j-\frac{1}{2}\right)\left(j+\frac{\alpha}{2}\right)}\:\varepsilon
\widetilde{\varphi}_{2j-2}^{(\alpha+1)}(x)-\sqrt{j\left(j+\frac{\alpha+1}{2}\right)}\:\varepsilon\widetilde{\varphi}_{2j}^{(\alpha+1)}(x)\right]
\widetilde{\varphi}_{2j-1}^{(\alpha+1)}(y)\nonumber\\
&+&\frac{1}{2}\sqrt{N(N+\alpha+1)}\:\varepsilon\widetilde{\varphi}_{N}^{(\alpha+1)}(x)\widetilde{\varphi}_{N-1}^{(\alpha+1)}(y).\nonumber
\eea
From (\ref{phijp}), and using Lemma \ref{de}, we have
%recall that $\varepsilon\:D=I,$ acting on Schwarz function, we have,
\bea
\varphi_{2j-1}^{(\alpha+1)}(x)
&=&\varepsilon\left[\varphi_{2j-1}^{(\alpha+1)}(x)\right]'\nonumber\\
&=&\sqrt{j\left(j+\frac{\alpha+1}{2}\right)}\:\varepsilon\widetilde{\varphi}_{2j}^{(\alpha+1)}(x)-\sqrt{\left(j-\frac{1}{2}\right)
\left(j+\frac{\alpha}{2}\right)}\:\varepsilon\widetilde{\varphi}_{2j-2}^{(\alpha+1)}(x).\nonumber
\eea
The sum simplifies to
$$
\sum_{j=0}^{\frac{N}{2}-1}\varepsilon\widetilde{\varphi}_{2j}^{(\alpha+1)}(x)\left[\varphi_{2j}^{(\alpha+1)}(y)\right]'
=-\sum_{j=1}^{\frac{N}{2}}\varphi_{2j-1}^{(\alpha+1)}(x)\widetilde{\varphi}_{2j-1}^{(\alpha+1)}(y)
+\frac{1}{2}\sqrt{N(N+\alpha+1)}\:\varepsilon\widetilde{\varphi}_{N}^{(\alpha+1)}(x)\widetilde{\varphi}_{N-1}^{(\alpha+1)}(y).
$$
The computations above gives a compact form for $K_{N,1}^{(2,2)}(x,y):$
\bea
K_{N,1}^{(2,2)}(x,y)
&=&\sum_{j=0}^{\frac{N}{2}-1}\varphi_{2j}^{(\alpha+1)}(x)\widetilde{\varphi}_{2j}^{(\alpha+1)}(y)
+\sum_{j=1}^{\frac{N}{2}}\varphi_{2j-1}^{(\alpha+1)}(x)\widetilde{\varphi}_{2j-1}^{(\alpha+1)}(y)
-\frac{1}{2}\sqrt{N(N+\alpha+1)}\:\varepsilon\widetilde{\varphi}_{N}^{(\alpha+1)}(x)\widetilde{\varphi}_{N-1}^{(\alpha+1)}(y)\nonumber\\
&=&\sum_{j=0}^{N-1}\varphi_{j}^{(\alpha+1)}(x)\widetilde{\varphi}_{j}^{(\alpha+1)}(y)-\frac{1}{2}\sqrt{N(N+\alpha+1)}\:\varepsilon
\widetilde{\varphi}_{N}^{(\alpha+1)}(x)\widetilde{\varphi}_{N-1}^{(\alpha+1)}(y)\nonumber\\
&=&S_{N}(x,y)-\frac{1}{2}\sqrt{N(N+\alpha+1)}\:\varepsilon
\widetilde{\varphi}_{N}^{(\alpha+1)}(x)\widetilde{\varphi}_{N-1}^{(\alpha+1)}(y),\nonumber
\eea
%where
%$$
%S_{N}(x,y)=\sum_{j=0}^{N-1}\widetilde{\varphi}_{j}^{(\alpha+1)}(x)\varphi_{j}^{(\alpha+1)}(y).
%$$
%From  the Christoffel-Darboux formula, %we have
%$$
%\sum_{j=0}^{N-1}\frac{L_{j}^{(\alpha+1)}(x)L_{j}^{(\alpha+1)}(y)}{\left(c_{j}^{(\alpha+1)}\right)^{2}}=\frac{N!}{\Gamma(N+\alpha+1)}\cdot
%\frac{L_{N-1}^{(\alpha+1)}(x)L_{N}^{(\alpha+1)}(y)-L_{N}^{(\alpha+1)}(x)L_{N-1}^{(\alpha+1)}(y)}{x-y}.
%$$
%it follows that
%\bea
%S_{N}(x,y)
%&=&\sum_{j=0}^{N-1}\widetilde{\varphi}_{j}(x)\varphi_{j}(y)\nonumber\\
%&=&x^{\frac{\alpha}{2}}\mathrm{e}^{-\frac{x}{2}}y^{\frac{\alpha}{2}+1}\mathrm{e}^{-\frac{y}{2}}\sum_{j=0}^{N-1}\frac{L_{j}^{(\alpha+1)}(x)
%L_{j}^{(\alpha+1)}(y)}{\left(c_{j}^{(\alpha+1)}\right)^{2}}\nonumber\\
%&=&\frac{\Gamma(N+1)}{\Gamma(N+\alpha+1)}\cdot\frac{L_{N-1}^{(\alpha+1)}(x)x^{\frac{\alpha}{2}}\mathrm{e}^{-\frac{x}{2}}\cdot
%L_{N}^{(\alpha+1)}(y)y^{\frac{\alpha}{2}+1}\mathrm{e}^{-\frac{y}{2}}-L_{N}^{(\alpha+1)}(x)x^{\frac{\alpha}{2}}\mathrm{e}^{-\frac{x}{2}}\cdot
%L_{N-1}^{(\alpha+1)}(y)y^{\frac{\alpha}{2}+1}\mathrm{e}^{-\frac{y}{2}}}{x-y}.\nonumber
%\label{cdf}
%\eea
where
$$
S_{N}(x,y):=\sum_{j=0}^{N-1}\varphi_{j}^{(\alpha+1)}(x)\widetilde{\varphi}_{j}^{(\alpha+1)}(y)
=-\sqrt{N(N+\alpha+1)}\frac{\varphi_{N}^{(\alpha+1)}(x)\widetilde{\varphi}_{N-1}^{(\alpha+1)}(y)-\widetilde{\varphi}_{N}^{(\alpha+1)}(y)\varphi_{N-1}^{(\alpha+1)}(x)}{x-y}.
$$
Here we have used the Christoffel-Darboux formula in the last equality.

By Theorem \ref{sca1}, we have the following theorem.
\begin{theorem}
\bea
&&\left[G_{N}^{(1)}(f)\right]^{2}\nonumber\\
&=&\det\bigg(I+S_{N}(f^{2}+2f)-S_{N}\varepsilon f'-S_{N}f\varepsilon f'
-\frac{1}{2}\sqrt{N(N+\alpha+1)}\Big[\varepsilon\widetilde{\varphi}_{N}^{(\alpha+1)}\otimes\widetilde{\varphi}_{N-1}^{(\alpha+1)}
\left(f^{2}+2f\right)\Big]\nonumber\\
&-&\frac{1}{2}\sqrt{N(N+\alpha+1)}\left(\varepsilon\widetilde{\varphi}_{N}^{(\alpha+1)}\otimes\varepsilon\widetilde{\varphi}_{N-1}^{(\alpha+1)}\right)f'
-\frac{1}{2}\sqrt{N(N+\alpha+1)}\left[\varepsilon\widetilde{\varphi}_{N}^{(\alpha+1)}\otimes\varepsilon\left(\widetilde{\varphi}_{N-1}^{(\alpha+1)}f\right)\right]f'\bigg).\nonumber
\eea
\end{theorem}

\subsection{Large $N$ behavior of the LOE moment generating function}
Now we consider the scaling limit of $\left[G_{N}^{(1)}(f)\right]^{2}$. We write
$$
\left[G_{N}^{(1)}(f)\right]^{2}=:\det(I+T),
$$
where
\bea
T:
&=&S_{N}(f^{2}+2f)-S_{N}\varepsilon f'-S_{N}f\varepsilon f'
-\frac{1}{2}\sqrt{N(N+\alpha+1)}\Big[\varepsilon\widetilde{\varphi}_{N}^{(\alpha+1)}\otimes\widetilde{\varphi}_{N-1}^{(\alpha+1)}
\left(f^{2}+2f\right)\Big]\nonumber\\
&-&\frac{1}{2}\sqrt{N(N+\alpha+1)}\left(\varepsilon\widetilde{\varphi}_{N}^{(\alpha+1)}\otimes\varepsilon\widetilde{\varphi}_{N-1}^{(\alpha+1)}\right)f'\nonumber\\
&-&\frac{1}{2}\sqrt{N(N+\alpha+1)}\left[\varepsilon\widetilde{\varphi}_{N}^{(\alpha+1)}\otimes\varepsilon\left(\widetilde{\varphi}_{N-1}^{(\alpha+1)}f\right)\right]f'.\nonumber
\eea
In the computations below, we replace $f(x)$ by $f\left(\sqrt{4Nx}\right)$ and $f'(x)$ by
$$f'\left(\sqrt{4Nx}\right)=\frac{d}{\sqrt{4N}d\sqrt{x}}f\left(\sqrt{4Nx}\right).$$
Similarly to the LSE case, we have the following two theorems.
\begin{theorem}
$$
\lim_{N\rightarrow\infty}\frac{y}{2N}S_{N}\left(\frac{x^{2}}{4N},\frac{y^{2}}{4N}\right)=B^{(\alpha+1)}(x,y),
$$

$$
\lim_{N\rightarrow\infty}\frac{x}{2N}S_{N}\left(\frac{x^{2}}{4N},\frac{x^{2}}{4N}\right)=B^{(\alpha+1)}(x,x),\label{loe1}
$$
where $B^{(\alpha+1)}(x,y)$ and $B^{(\alpha+1)}(x,x)$ are given by (\ref{bxy}) and (\ref{bxx}) with $\alpha$ replaced by $\alpha+1$.
\end{theorem}

\begin{theorem}
$$
\widetilde{\varphi}_{N-1}^{(\alpha+1)}\left(\frac{x^{2}}{4N}\right)= 2\: N^{\frac{1}{2}}\:\frac{J_{\alpha+1}(x)}{x}+O\left(N^{-\frac{3}{2}}\right),\;\;N\rightarrow\infty,
$$

$$
\varepsilon\widetilde{\varphi}_{N}^{(\alpha+1)}\left(\frac{x^{2}}{4N}\right)= N^{-\frac{1}{2}}\left[\int_{0}^{x}J_{\alpha+1}(y)dy-1\right]+O\left(N^{-\frac{5}{2}}\right),
\;\;N\rightarrow\infty,
$$

$$
\varepsilon\widetilde{\varphi}_{N-1}^{(\alpha+1)}\left(\frac{x^{2}}{4N}\right)= N^{-\frac{1}{2}}\int_{0}^{x}J_{\alpha+1}(y)dy+O\left(N^{-\frac{5}{2}}\right),\;\;N\rightarrow\infty,
$$

$$
\varepsilon\left(\widetilde{\varphi}_{N-1}^{(\alpha+1)}f\right)\left(\frac{x^{2}}{4N}\right)= \frac{1}{2}N^{-\frac{1}{2}}
\left[\int_{0}^{x}J_{\alpha+1}(y)f(y)dy-\int_{x}^{\infty}J_{\alpha+1}(y)f(y)dy\right]+O\left(N^{-\frac{5}{2}}\right),\;\;N\rightarrow\infty.\label{loe2}
$$
\end{theorem}

Using Theorem \ref{loe1} and Theorem \ref{loe2} to compute $\mathrm{Tr}\:T$ and $\mathrm{Tr}\:T^{2}$ as $N\rightarrow\infty$. We obtain
\bea
\mathrm{Tr}\:T
&=&\int_{0}^{\infty}B^{(\alpha+1)}(x,x)\left[f^{2}(x)+2f(x)\right]dx\nonumber\\
&-&\frac{1}{2}\int_{0}^{\infty}\left[\int_{0}^{x}J_{\alpha+1}(y)dy-1\right]J_{\alpha+1}(x)\left[f^{2}(x)+2f(x)\right]dx\nonumber\\
&-&\frac{1}{4N}\int_{0}^{\infty}\left[\int_{x}^{\infty}B^{(\alpha+1)}(x,y)dy-\int_{0}^{x}B^{(\alpha+1)}(x,y)dy\right]x\:f'(x)dx\nonumber\\
&-&\frac{1}{4N}\int_{0}^{\infty}\left[\int_{x}^{\infty}B^{(\alpha+1)}(x,y)f(y)dy-\int_{0}^{x}B^{(\alpha+1)}(x,y)f(y)dy\right]x\:f'(x)dx\nonumber\\
&-&\frac{1}{4N}\int_{0}^{\infty}\left[\int_{0}^{x}J_{\alpha+1}(y)dy-1\right]\left[\int_{0}^{x}J_{\alpha+1}(y)dy\right]x\:f'(x)dx\nonumber\\
&-&\frac{1}{8N}\int_{0}^{\infty}\left[\int_{0}^{x}J_{\alpha+1}(y)dy-1\right]\left[\int_{0}^{x}J_{\alpha+1}(y)f(y)dy-\int_{x}^{\infty}J_{\alpha+1}(y)f(y)dy\right]x\:f'(x)dx\nonumber\\
&+&O\left(N^{-2}\right),\;\;N\rightarrow\infty,\nonumber
\eea
and
\bea
&&\mathrm{Tr}\:T^{2}\nonumber\\
&=&\int_{0}^{\infty}\int_{0}^{\infty}B^{(\alpha+1)}(x,y)B^{(\alpha+1)}(y,x)\left[f^{2}(x)+2f(x)\right]\left[f^{2}(y)+2f(y)\right]dxdy\nonumber\\
&-&\int_{0}^{\infty}\int_{0}^{\infty}B^{(\alpha+1)}(x,y)J_{\alpha+1}(x)\left[\int_{0}^{y}J_{\alpha+1}(z)dz-1\right]\left[f^{2}(x)+2f(x)\right]
\left[f^{2}(y)+2f(y)\right]dxdy\nonumber\\
&+&\frac{1}{4}\int_{0}^{\infty}\int_{0}^{\infty}\left[\int_{0}^{x}J_{\alpha+1}(z)dz-1\right]\left[\int_{0}^{y}J_{\alpha+1}(z)dz-1\right]
J_{\alpha+1}(x)J_{\alpha+1}(y)\nonumber\\
&&\left[f^{2}(x)+2f(x)\right]\left[f^{2}(y)+2f(y)\right]dxdy\nonumber\\
&-&\frac{1}{2N}\int_{0}^{\infty}\int_{0}^{\infty}B^{(\alpha+1)}(x,y)\left[\int_{x}^{\infty}
B^{(\alpha+1)}(y,z)dz-\int_{0}^{x}B^{(\alpha+1)}(y,z)dz\right]x\:f'(x)\left[f^{2}(y)+2f(y)\right]dxdy\nonumber\\
&-&\frac{1}{2N}\int_{0}^{\infty}\int_{0}^{\infty}B^{(\alpha+1)}(x,y)\left[\int_{x}^{\infty}
B^{(\alpha+1)}(y,z)f(z)dz-\int_{0}^{x}B^{(\alpha+1)}(y,z)f(z)dz\right]\nonumber\\
&&x\:f'(x)\left[f^{2}(y)+2f(y)\right]dxdy\nonumber\\
&-&\frac{1}{2N}\int_{0}^{\infty}\int_{0}^{\infty}B^{(\alpha+1)}(x,y)\left[\int_{0}^{x}J_{\alpha+1}(z)dz\right]\left[\int_{0}^{y}J_{\alpha+1}(z)dz-1\right]
x\:f'(x)\left[f^{2}(y)+2f(y)\right]dxdy\nonumber\\
&-&\frac{1}{4N}\int_{0}^{\infty}\int_{0}^{\infty}B^{(\alpha+1)}(x,y)\left[\int_{0}^{x}J_{\alpha+1}(z)f(z)dz-\int_{x}^{\infty}J_{\alpha+1}(z)f(z)dz\right]
\left[\int_{0}^{y}J_{\alpha+1}(z)dz-1\right]\nonumber\\
&&x\:f'(x)\left[f^{2}(y)+2f(y)\right]dxdy\nonumber\\
&+&\frac{1}{4N}\int_{0}^{\infty}\int_{0}^{\infty}\left[\int_{x}^{\infty}B^{(\alpha+1)}(y,z)dz
-\int_{0}^{x}B^{(\alpha+1)}(y,z)dz\right]\left[\int_{0}^{x}J_{\alpha+1}(z)dz-1\right]J_{\alpha+1}(y)\nonumber\\
&&x\:f'(x)\left[f^{2}(y)+2f(y)\right]dxdy\nonumber\\
&+&\frac{1}{4N}\int_{0}^{\infty}\int_{0}^{\infty}\left[\int_{x}^{\infty}B^{(\alpha+1)}(y,z)f(z)dz
-\int_{0}^{x}B^{(\alpha+1)}(y,z)f(z)dz\right]\left[\int_{0}^{x}J_{\alpha+1}(z)dz-1\right]J_{\alpha+1}(y)\nonumber\\
&&x\:f'(x)\left[f^{2}(y)+2f(y)\right]dxdy\nonumber\\
&+&\frac{1}{4N}\int_{0}^{\infty}\int_{0}^{\infty}\left[\int_{0}^{x}J_{\alpha+1}(z)dz-1\right]\left[\int_{0}^{y}J_{\alpha+1}(z)dz-1\right]\left[\int_{0}^{x}J_{\alpha+1}(z)dz\right]
J_{\alpha+1}(y)\nonumber\\
&&x\:f'(x)\left[f^{2}(y)+2f(y)\right]dxdy\nonumber\\
&+&\frac{1}{8N}\int_{0}^{\infty}\int_{0}^{\infty}\left[\int_{0}^{x}J_{\alpha+1}(z)dz-1\right]\left[\int_{0}^{y}J_{\alpha+1}(z)dz-1\right]\nonumber\\
&&\left[\int_{0}^{x}J_{\alpha+1}(z)f(z)dz-\int_{x}^{\infty}J_{\alpha+1}(z)f(z)dz\right]J_{\alpha+1}(y)x\:f'(x)\left[f^{2}(y)+2f(y)\right]dxdy\nonumber\\
&+&O\left(N^{-2}\right),\;\;N\rightarrow\infty.\nonumber
\eea

Now we want to see the mean and variance of the linear statistics $\sum_{j=1}^{N}F\left(\sqrt{4Nx_{j}}\right)$, so we need to obtain the coefficients of $\lambda$ and $\lambda^{2}$,
firstly we know
$$
f\left(\sqrt{4Nx}\right)\approx-\lambda F\left(\sqrt{4Nx}\right)+\frac{\lambda^{2}}{2}F^{2}\left(\sqrt{4Nx}\right),
$$
then we replace $f$ with $-\lambda F+\frac{\lambda^{2}}{2}F^{2}$ in the expression of $\mathrm{Tr}\:T$ and $\mathrm{Tr}\:T^{2}$, similar to previous\\ discussions,
denote by $\mu_{N}^{(LOE,\:\alpha)}$ and $\mathcal{V}_{N}^{(LOE,\:\alpha)}$ the mean and variance of the linear statistics $\sum_{j=1}^{N}F\left(\sqrt{4Nx_{j}}\right)$, then we have the following theorem.
\begin{theorem}
As $N\rightarrow\infty$,
\bea
\mu_{N}^{(LOE,\:\alpha)}
&=&\mu_{N}^{(LUE,\:\alpha+1)}-\frac{1}{2}\int_{0}^{\infty}\left[\int_{0}^{x}J_{\alpha+1}(y)dy-1\right]J_{\alpha+1}(x)F(x)dx\nonumber\\
&-&\frac{1}{8N}\int_{0}^{\infty}\left[\int_{x}^{\infty}B^{(\alpha+1)}(x,y)dy-\int_{0}^{x}B^{(\alpha+1)}(x,y)dy\right]x\:F'(x)dx\nonumber\\
&-&\frac{1}{8N}\int_{0}^{\infty}\left[\int_{0}^{x}J_{\alpha+1}(y)dy-1\right]\left[\int_{0}^{x}J_{\alpha+1}(y)dy\right]x\:F'(x)dx
+O\left(N^{-2}\right),\nonumber
\eea
\bea
\mathcal{V}_{N}^{(LOE,\:\alpha)}
&=&2\:\mathcal{V}_{N}^{(LUE,\:\alpha+1)}-\int_{0}^{\infty}\left[\int_{0}^{x}J_{\alpha+1}(y)dy-1\right]J_{\alpha+1}(x)F^{2}(x)dx\nonumber\\
&+&2\int_{0}^{\infty}\int_{0}^{\infty}B^{(\alpha+1)}(x,y)J_{\alpha+1}(x)\left[\int_{0}^{y}J_{\alpha+1}(z)dz-1\right]F(x)F(y)dxdy\nonumber\\
&-&\frac{1}{2}\int_{0}^{\infty}\int_{0}^{\infty}\left[\int_{0}^{x}J_{\alpha+1}(z)dz-1\right]\left[\int_{0}^{y}J_{\alpha+1}(z)dz-1\right]
J_{\alpha+1}(x)J_{\alpha+1}(y)F(x)F(y)dxdy\nonumber\\
&-&\frac{1}{4N}\int_{0}^{\infty}\left[\int_{x}^{\infty}B^{(\alpha+1)}(x,y)dy-\int_{0}^{x}B^{(\alpha+1)}(x,y)dy\right]x\:F(x)F'(x)dx\nonumber\\
&-&\frac{1}{4N}\int_{0}^{\infty}\left[\int_{x}^{\infty}B^{(\alpha+1)}(x,y)F(y)dy-\int_{0}^{x}B^{(\alpha+1)}(x,y)F(y)dy\right]x\:F'(x)dx\nonumber\\
&-&\frac{1}{4N}\int_{0}^{\infty}\left[\int_{0}^{x}J_{\alpha+1}(y)dy-1\right]\left[\int_{0}^{x}J_{\alpha+1}(y)dy\right]x\:F(x)F'(x)dx\nonumber\\
&-&\frac{1}{8N}\int_{0}^{\infty}\left[\int_{0}^{x}J_{\alpha+1}(y)dy-1\right]\left[\int_{0}^{x}J_{\alpha+1}(y)F(y)dy-\int_{x}^{\infty}J_{\alpha+1}(y)F(y)dy\right]
x\:F'(x)dx\nonumber\\
&+&\frac{1}{2N}\int_{0}^{\infty}\int_{0}^{\infty}B^{(\alpha+1)}(x,y)\left[\int_{x}^{\infty}
B^{(\alpha+1)}(y,z)dz-\int_{0}^{x}B^{(\alpha+1)}(y,z)dz\right]x\:F'(x)F(y)dxdy\nonumber\\
&+&\frac{1}{2N}\int_{0}^{\infty}\int_{0}^{\infty}B^{(\alpha+1)}(x,y)\left[\int_{0}^{x}J_{\alpha+1}(z)dz\right]\left[\int_{0}^{y}J_{\alpha+1}(z)dz-1\right]
x\:F'(x)F(y)dxdy\nonumber\\
&-&\frac{1}{4N}\int_{0}^{\infty}\int_{0}^{\infty}\left[\int_{x}^{\infty}B^{(\alpha+1)}(y,z)dz
-\int_{0}^{x}B^{(\alpha+1)}(y,z)dz\right]\left[\int_{0}^{x}J_{\alpha+1}(z)dz-1\right]J_{\alpha+1}(y)\nonumber\\
&&x\:F'(x)F(y)dxdy\nonumber\\
&-&\frac{1}{4N}\int_{0}^{\infty}\int_{0}^{\infty}\left[\int_{0}^{x}J_{\alpha+1}(z)dz-1\right]\left[\int_{0}^{y}J_{\alpha+1}(z)dz-1\right]\left[\int_{0}^{x}J_{\alpha+1}(z)dz\right]
J_{\alpha+1}(y)\nonumber\\
&&x\:F'(x)F(y)dxdy+O\left(N^{-2}\right),\nonumber
\eea
where $\mu_{N}^{(LUE,\:\alpha+1)}$ and $\mathcal{V}_{N}^{(LUE,\:\alpha+1)}$ for $N\rightarrow\infty$ are given in (\ref{luem}) and (\ref{luev}) respectively
 with $\alpha$ replaced by $\alpha+1$.
\end{theorem}

\section{Conclusion}
In this paper, we study the moment generating function of linear statistics of the form $\sum_{j=1}F(x_j),$ namely, the expectation value or average
of
$${\rm exp}\left[-\lambda\:\sum_{j=1}^{N}F(x_j)\right]$$ computed in the ``background"
of the symplectic and orthogonal ensembles. We then specialize to the Gaussian case, where the background weight is the
normal distribution and the Laguerre case, where the background weight is the gamma distribution.

Finally, we compute the large $N$ behavior of the moment generating, where $F(.)$ is suitably scaled  and obtained the
 mean and variance of the corresponding linear statistics in the four cases. The more complex situation with the Jacobi background will be studied
 in the future, in which case the polynomials depend on two parameters $a$ and $b$.

\section{Acknowledgment} 

The Authors would like to thank the Macau Science Foundation for generous support, awarding the grant FDCT/2012/A3.


\begin{thebibliography}{}
%Type = Article
\bibitem{Adler}
{M. Adler}, {\em Spectral statistics of orthogonal and symplectic ensembles}, In {The Oxford Handbook of Random Matrix Theory}, {Oxford University Press}, {Oxford}, {2011}.
%Type = Article
\bibitem{Basor1993}
{E. L. Basor}, {C. A. Tracy}, {\em Variance calculations and the Bessel kernel}, {Journal of Statistical Physics} {73} ({1993}) {415--421}.
%Type = Article
\bibitem{Basor1997}
{E. L. Basor}, {\em Distribution functions for random variables for ensembles of positive Hermitian matrices}, {Commun. Math. Phys.} {188} ({1997}) {327--350}.
%Type = Article
\bibitem{Basor}
{E. L. Basor}, {Y. Chen}, {H. Widom}, {\em Determinants of Hankel Matrices}, {Journal of Functional Analysis} {179} ({2001}) {214--234}.
%Type = Article
\bibitem{de Bruijn}
{N. G. de Bruijn}, {\em On some multiple integrals involving determinants}, {J. Indian Math. Soc.} {19} ({1955}) {133--151}.
%Type = Article
\bibitem{Dieng}
{M. Dieng}, {C. A. Tracy}, {\em Application of random matrix theory to multivariate statistics}, In {\em Random Matrices, Random Processes and Integrable Systems} ({2011}) {443--507}.
%Type = Book
\bibitem{Gohberg}
{I. Gohberg}, {S. Goldberg}, {M.A. Kaashoek}, {\em Classes of Linear Operators, vols. {I}}, {Birkh{\"a}user Verlag, Basel}, {1990}.
%Type = Book
\bibitem{Gradshteyn}
{I.S. Gradshteyn}, {I.M. Ryzhik}, {\em Table of Integrals, Series, and Products: Seventh Edition}, {Elsevier(Singapore) Pte Ltd.}, {Singapore}, {2007}.
%Type = Book
\bibitem{Kuramoto}
{Y. Kuramoto}, {Y. Kato}, {\em Dynamics of One-Dimensional Quantum Systems: Inverse-Square Interaction Models}, {Cambridge University Press}, {Cambridge}, {2009}.
%Type = Book
\bibitem{Lebedev}
{N.N. Lebedev}, {\em Special Functions and Their Applications}, {Dover Publications, INC.}, {New York}, {1972}.
%Type = Book
\bibitem{Mehta}
{M. L. Mehta}, {\em Random Matrices:Third Edition}, {Elsevier(Singapore) Pte Ltd.}, {Singapore}, {2006}.
%Type = Article
\bibitem{Selberg}
{A. Selberg}, {\em Bemerkninger om et multiplet integral}, {Norsk Matematisk Tidsskrift} {26} ({1944}) {71--78}.
%Type = Book
\bibitem{Szego}
{G. Szeg{\"o}}, {\em Orthogonal Polynomials: Fourth Edition}, {American Mathematical Society}, {Providence, RI}, {1975}.
%Type = Article
\bibitem{Tracy1996}
{C. A. Tracy}, {H. Widom}, {\em On orthogonal and symplectic matrix ensembles}, {Commun. Math. Phys.} {177} ({1996}) {727--754}.
%Type = Article
\bibitem{Tracy1998}
{C. A. Tracy}, {H. Widom}, {\em Correlation functions, cluster functions and spacing distributions for random matrices}, {Journal of Statistical Physics} {92} ({1998}) {809--835}.
%Type = Article
\bibitem{Widom}
{H. Widom}, {\em On the relation between orthogonal, symplectic and unitary matrix ensembles}, {Journal of Statistical Physics} {94} ({1999}) {347--363}.


\end{thebibliography}
\end{document}